\documentclass[11pt,a4paper]{article}
\usepackage[T1]{fontenc}
\usepackage{lmodern}
\usepackage[margin=26mm]{geometry}
\usepackage{amsmath,amssymb,amsthm,mathtools,booktabs,array}
\usepackage{microtype,mathrsfs}
\usepackage[hidelinks]{hyperref}
\newcommand{\arxivref}[1]{\href{https://arxiv.org/abs/#1}{arXiv:\nolinkurl{#1}}}
\newcommand{\doiref}[1]{\href{https://doi.org/#1}{doi:\nolinkurl{#1}}}
\hypersetup{bookmarksnumbered=true,pdftitle={Gamma conjecture II and explicit Stokes data for a quadric projective bundle},
pdfauthor={Xiaobin Li},pdfsubject={Gamma conjecture II, cohomological chiral symmetry, explicit Stokes data, and spinor transfer for a quadric projective bundle},
pdfkeywords={quantum cohomology, Gamma conjecture, chiral symmetry, Fano projective bundle, spinor bundle}}
\usepackage{enumitem}
\usepackage{longtable,fancyhdr,needspace}
\numberwithin{equation}{section}
\setlist{nosep}
\newtheorem{theorem}{Theorem}[section]
\newtheorem{proposition}[theorem]{Proposition}
\newtheorem{lemma}[theorem]{Lemma}
\newtheorem{corollary}[theorem]{Corollary}
\theoremstyle{remark}\newtheorem{remark}[theorem]{Remark}
\newcommand{\Q}{\mathbb Q}
\newcommand{\C}{\mathbb C}
\newcommand{\Z}{\mathbb Z}
\newcommand{\PP}{\mathbb P}
\newcommand{\OO}{\mathcal O}
\newcommand{\Ga}{\widehat\Gamma}
\newcommand{\ch}{\operatorname{ch}}
\newcommand{\Ch}{\operatorname{Ch}}
\newcommand{\Td}{\operatorname{Td}}

\newcommand{\SR}{\mathrm{SR}}

\newcommand{\im}{\operatorname{im}}

\newcommand{\rank}{\operatorname{rank}}

\newcommand{\disc}{\operatorname{disc}}
\newcommand{\nablaq}{\nabla_{q\partial_q}}
\newcommand{\nablar}{\nabla_{r\partial_r}}
\newcommand{\Dq}{D_q}
\newcommand{\Dr}{D_r}

\title{Gamma conjecture II and explicit Stokes data\\for a quadric projective bundle}
\author{Xiaobin Li\\[2pt]\small School of Mathematics, Southwest Jiaotong University, Chengdu, China}
\date{9 October 2026}
\begin{document}
\maketitle
\begin{abstract}
We study the connection problem for the quantum differential equation of
the Fano fourfold $X=\PP_{Q^3}(\OO\oplus\OO(1))$. The problem is to
identify its canonical asymptotic solutions with exceptional objects and
compute the associated Stokes data. We use exact Mellin--Bessel integral
representations and uniform vector asymptotics on sectors wider than
$\pi$, followed by analytic continuation that preserves the Gamma
framing. Cohomological chiral symmetry and mutations of exceptional
collections put all eight sections at a common phase. This proves Gamma
conjecture II for $X$ and determines a full exceptional collection and
its Stokes and central connection matrices at the unit Novikov point and
a specified admissible phase. The chiral symmetry also gives a flat
symmetric bilinear pairing, an orthogonality identity for the Stokes
matrix, and an integral unipotent square root of the monodromy matrix
for a full turn. We obtain explicit formulas for the connection data at
the opposite phase; on the Gamma lattice, the half-turn acts by
tensoring with a square root of the anticanonical bundle. The same
spinor integral transform constructs two canonical Gamma sections for
$\PP_{Q^{2m+1}}(\OO\oplus\OO(1))$, $m\geq1$, at positive parameters
with distinct spectrum. Their sectors have opening greater than $\pi$,
and their leading constants are independent of the dimension.
\end{abstract}
\noindent\textbf{Keywords.} Quantum cohomology; Gamma conjecture;
Stokes matrix; Cohomological chiral symmetry; Exceptional collection;
Mellin--Bessel transform; Helix monodromy.

\section{Introduction and main results}\label{sec:introduction}
We study the connection problem for the quantum differential equation of
\begin{equation}\label{model}
 X=\PP_{Q^3}(\OO\oplus\OO(1)),
\end{equation}
where $Q^3\subset\PP^4$ is a smooth quadric and projectivization
parametrizes one-dimensional quotients. The question is which
exceptional objects give the canonical asymptotic solutions, and how the
resulting basis determines the Stokes and central connection matrices at
a prescribed parameter and phase. We answer this question for the
fourfold \eqref{model}, with both curve parameters retained.

Gamma conjecture II predicts that the Gamma sections of a full
exceptional collection form an asymptotic basis of the quantum
connection \cite{GGI,Integral,Iritani}. The Gamma framing fixes the
normalization at the regular singular end, while canonical asymptotics
fix it at the irregular end. Relating these normalizations requires
identifying individual flat sections and following their object labels,
logarithmic branches and signs under continuation. We carry out this
comparison for \eqref{model} and compute both connection matrices in
the resulting specified basis.

\subsection{The main result}
Write $\pi:X\to Q^3$,
$H=\pi^*c_1(\OO_{Q^3}(1))$, $\xi=c_1(\OO_\pi(1))$,
$L=\OO_X(H)$ and $U=\OO_\pi(1)$. The Novikov variables $q,r$
correspond to a line in the section defined by the quotient onto $\OO$
and to a fiber, respectively. The fourfold has Fano index two and
cohomology of rank eight; the calculation uses its entire cohomology.
Let $T$ denote the unique simple eigenvalue of $c_1(X)\star$ with
maximal real part at positive parameters, as established in
Theorem~\ref{spectral}.

Let $S$ be the rank-two spinor bundle on $Q^3$, normalized by
$c_1(S)=-H$, and put $S_0=\pi^*S$, $S_1=S_0\otimes U$.
Let $j:Q^3\hookrightarrow X$ be the section defined by the quotient
onto $\OO$. Set $V=\pi^*(\Omega^1_{\PP^4}(1)|_{Q^3})$, and let $F$
be the bundle given by the nonzero extension, unique up to scalar,
\begin{equation}\label{introextension}
 0\longrightarrow V\otimes U\longrightarrow F\longrightarrow L
 \longrightarrow0.
\end{equation}
Proposition~\ref{geometricobjects09} proves the existence and uniqueness
of this extension and identifies the corresponding mutation.

We use the Gamma framing
\[
 \mathfrak s(E)=(2\pi)^{-2}\mathcal L z^{-\mu}z^\rho
 \bigl(\widehat\Gamma_X\Ch(E)\bigr),
\]
where $\mathcal L$ is the fundamental solution normalized at large radius,
$\mu$ is the grading operator,
$\Ch(E)=\sum_p(2\pi i)^p\ch_p(E)$ and $\rho=c_1(X)\cup$.
For canonical columns ordered by decreasing
$\operatorname{Im}(e^{-i\phi}u_i)$, our Stokes convention is
\begin{equation}\label{introStokesconvention}
 S_{ij}=(y_i(e^{-\pi i}z),y_j(z))_X.
\end{equation}
The pairing is the Poincar\'e pairing, and the first argument is continued
along the indicated path. Section~\ref{sec:conventions} fixes the
remaining conventions, following \cite[Section 2.6]{GGI}.

\begin{theorem}[Gamma II and specified phase data]
\label{mainUnit09}\label{mainGammaII}\label{unitphasetheorem}
At $q=r=1$ and $\phi_0=-\arctan(1/100)$, the Gamma sections of
\begin{equation}\label{introfoundation}
 \bigl(L^{-1},j_*\OO_{Q^3}(-1)[3],\OO,S_0L[-1],S_1,F[1],L,U\bigr)
\end{equation}
form a normalized canonical asymptotic fundamental solution.
This is a full exceptional collection. If the nonreal eigenvalues of
$c_1(X)\star$ are $\pm a\pm ib$, with $a,b>0$, the exponent order is
\[
 (a+ib,-a+ib,T,2,-2,-T,a-ib,-a-ib).
\]
In convention \eqref{introStokesconvention}, its Stokes matrix is
\begin{equation}\label{unitStokes09}
 S_{\phi_0}=\begin{pmatrix}
 1&-1&5&-16&20&-65&14&19\\
 0&1&1&-4&0&-5&5&1\\
 0&0&1&-4&4&-16&5&6\\
 0&0&0&1&0&4&-4&-4\\
 0&0&0&0&1&-4&0&4\\
 0&0&0&0&0&1&-1&-6\\
 0&0&0&0&0&0&1&1\\
 0&0&0&0&0&0&0&1
 \end{pmatrix}.
\end{equation}
If $Y_\infty=\mathcal L z^{-\mu}z^\rho$ and
$Y_{\phi_0}=Y_\infty\mathcal C_{\phi_0}$, then
\[
 \mathcal C_{\phi_0}=(2\pi)^{-2}
 [\widehat\Gamma_X\Ch(B_i)]_{\mathrm{cup}},
\]
where $B_i$ run through \eqref{introfoundation} in the displayed order.
All entries are given by \eqref{charactersunit09}--\eqref{gammalogunit09}.
In particular, $X$ satisfies Gamma conjecture II.
\end{theorem}

Theorem~\ref{mainUnit09} gives the Gamma II correspondence with its
connection data at a specified parameter and phase. The proof identifies
the canonical sections first, then computes their pairings by the Gamma
factorization of Riemann--Roch. This fixes the object order and signs of
the matrices. The section sheaf and the extension bundle in
\eqref{introfoundation} record the mutations along the chosen path.
Other Stokes conventions may give an inverse or transpose.

\subsection{Two spinor columns over odd-dimensional quadrics}
The spinor construction gives two Gamma sections for projective bundles
over odd-dimensional quadrics. Their asymptotics are established directly.
\begin{theorem}[Spinor transfer over odd quadric bases]\label{mainfamily}
For every odd $\nu=2m+1\geq3$ let
$X_\nu=\PP_{Q^\nu}(\OO\oplus\OO(1))$ and let $S_\nu$ be its base
spinor bundle, normalized by
$\ch(S_\nu)=2^{m+1}/(1+e^H)$. At positive $q,r$, the Gamma sections of
$\pi^*S_\nu$ and $\pi^*S_\nu\otimes U$ have full normalized vector
expansions with exponents $2\sqrt r$ and $-2\sqrt r$ on the respective
sectors
\[
 -\frac\pi2<\arg z<\frac\pi2+\frac{2\pi}{\nu+1},\qquad
 -\frac\pi2-\frac\pi\nu<\arg z<\frac\pi2.
\]
Their scalar leading constants are respectively
$-i/(2\sqrt{qr})$ and $i/(2\sqrt{qr})$ in the framing
\[
 E\longmapsto(2\pi)^{-(\nu+1)/2}\mathcal L_\nu
 z^{-\mu_\nu}z^{\rho_\nu}\widehat\Gamma_{X_\nu}\Ch(E).
\]
At points with distinct spectrum these are two canonical Gamma columns.
All orders and their parameter derivatives are uniform on positive
compact parameter sets for fixed $\nu$.
\end{theorem}
The proof in Section~\ref{sec:familytransfer} determines the Gamma
marking before applying the contour estimate in each dimension. It
requires no hypothesis on the transfer of Stokes pairings. The result
identifies two columns; a full Gamma II comparison for every $X_\nu$
would require the remaining columns and their pairings.
Proposition~\ref{familylattice} identifies the two spinor classes as
generators of the finite quotient of $K_0(X_\nu)$ by all line bundle
classes.

\subsection{Chiral symmetry and the connection matrices}
The involution $\mathsf D|_{H^{2p}(X)}=(-1)^p$ anticommutes with
$C=c_1(X)\star$ and pairs its eigenvalues $u$ and $-u$.
This is the case of index two in the standard Fano index symmetry
\cite[Remark 3.1.3]{GGI}. We call it \emph{cohomological chiral
symmetry}; the precise algebraic meaning is fixed in
Section~\ref{sec:chiral}. It relates the spectrum and the
Gamma lattice to the connection matrices with their chosen normalizations.

\begin{corollary}[Chiral compatibility of the specified data]
\label{introchiral}
Let $\mathcal B=(B_1,\ldots,B_8)$ be the collection
\eqref{introfoundation}, set $A=L\otimes U$, and let $M_A$ be the
integer matrix of $E\mapsto E\otimes A$ in this ordered $K$-basis.
Define the signed reversal $J_\epsilon$ by
\[
 J_\epsilon e_i=\epsilon_i e_{9-i},\qquad
 (\epsilon_1,\ldots,\epsilon_8)=(1,1,1,-1,-1,1,1,1).
\]
At the point and phase of Theorem~\ref{mainGammaII}, write
$S=S_{\phi_0}$ and $\mathcal C=\mathcal C_{\phi_0}$, and let $P$
be the Poincar\'e matrix in the cup frame. Then
\begin{equation}\label{introchiraldata}
 \begin{gathered}
 S^{\mathsf T}J_\epsilon S=J_\epsilon,\qquad
 \mathcal C^{\mathsf T}P\mathsf D\mathcal C=J_\epsilon,\\
 M_A=J_\epsilon S,\qquad M_A^2=S^{-\mathsf T}S,\qquad
 \mathcal C M_A=e^{\pi i\rho}\mathcal C.
 \end{gathered}
\end{equation}
The matrix $M_A$ is unipotent, with Jordan type
$J_5(1)\oplus J_3(1)$. With logarithms continued along the path, the ordered collection
$\mathcal B\otimes A$ is canonical at phase $\phi_0-\pi$, with
\[
 S_{\phi_0-\pi}=S_{\phi_0},\qquad
 \mathcal C_{\phi_0-\pi}=e^{\pi i\rho}\mathcal C_{\phi_0}.
\]
All forms and transposes in these identities are complex bilinear.
\end{corollary}

The proof is given in Propositions~\ref{chiralflat} and
\ref{unitChiral}, together with Corollary~\ref{oppositephase}.
The involution $\mathsf D$ on cohomology lifts to the operator
$s(z)\mapsto\mathsf D s(e^{\pi i}z)$ on flat sections. Its square is
the monodromy of a full turn, which acts on the Gamma lattice by tensoring
with $A^2=\omega_X^{-1}$. Thus the lift to the Gamma lattice has a
nontrivial square. The matrix identities follow from the analytic
identification of the Gamma columns, which is the essential step in
proving Gamma conjecture II.

\subsection{The strategy of Proof}
The two spinor sections are the main analytic difficulty. Gamma II is
known for the base $Q^3$, but the base and fiber contributions to the
projective bundle are analytically coupled. The two-dimensional
eigenspace corresponding to $u=\pm2\sqrt r$ is not preserved by the
full quantum connection (Theorem~\ref{aux:leakage}). These exponents
alone therefore give no reduction to a connection of rank two. Line
bundles generate a subgroup of index sixteen in $K_0(X)$
(Proposition~\ref{aux:lattice}), so the spinor classes are also needed
to recover the integral object labels.

We first determine three sections in the Gamma framing. For the structure sheaf,
a positive real integral has a unique strictly convex minimum, which
identifies the normalized section with fastest decay. For $S_0$ and $S_1$,
we match the exact Mellin--Bessel formulas to their Gamma initial data
at the regular end before estimating their asymptotics. The estimates then
control the unbounded imaginary Mellin order, together with the Bessel
order and the logarithmic branches. They give differentiable vector
expansions on sectors of opening greater than $\pi$
(Theorems~\ref{MBmarked} and \ref{twospinorcolumns}). An expansion on a single ray may
allow an exponentially smaller summand, so it need not determine a
canonical column.

We construct the line bundle sections by continuation. A clockwise loop
in the base parameter carries the section of the structure sheaf to the
$L$ section. We justify this transport using the explicit connection
after Laplace duality and uniform separation of the spectral clusters. Reflection gives $L^{-1}$, and the chiral half-turn with transported
logarithms gives $U,LU,L^2U$. Together with the spinor sections, these
produce the full exceptional collection
\[
 (L^{-1},S_0,\OO,L,S_1,U,LU,L^2U).
\]
We keep track of the Gamma label of each section throughout the continuation.

We then take the auxiliary point $q_0=(10^{-2})^3/4$, $r=1$, rotate
these rays to the common phase $\phi_0$, and follow the collection to
$q=1$. Section~\ref{sec:unitphase} explains the coordinated choice of
$q_0$ and $\phi_0$. A real parameterization of the spectrum proves that
this path crosses exactly two walls. The chronological mutations are
$L_3,L_5;L_2,L_6$, with coefficients $0,0;-1,1$
(Theorem~\ref{twowalls09}). At an earlier ray crossing, the canonical
spinor column is represented by $S_1-4U$ rather than $S_1$.
Using the class at each crossing determines the final objects, shifts
and matrix entries.

The proof thus passes from characteristic classes to canonical
sectorial solutions through three steps: identifying the spinor Gamma
markings, proving uniform asymptotics on the required sectors, and
continuing the sections to a common phase at the unit point. It uses the
existing mirror, asymptotic and mutation theories in this explicit
projective bundle calculation.

\subsection{Related works}
Dubrovin formulated a conjectural relation between semisimple quantum
cohomology, exceptional collections and monodromy data
\cite{DubrovinICM}.
Guzzetti established the Stokes--Euler correspondence for projective
spaces \cite{GuzzettiProjective}, and Ueda proved it for Grassmannians
and smooth cubic surfaces \cite{UedaGrassmannians,UedaCubic}.
Iwaki and Takahashi treated orbifold projective lines with at most
three orbifold points and positive orbifold Euler characteristic
\cite{IwakiTakahashi}.
Cruz Morales and van der Put used multisummation and the monodromy
identity to compute Stokes data for differential equations associated
with projective spaces, weighted projective spaces and certain Fano
hypersurfaces~\cite{CruzPut}.

Galkin, Golyshev and Iritani \cite{GGI} formulated the Gamma conjectures, and
Iritani \cite{Integral} constructed the Gamma integral structure. The
relation between exceptional collections, helices and monodromy data is
developed in \cite{CDGHelix}. Xiaowen Hu and Huazhong Ke prove Gamma II
for smooth quadrics \cite{HK}; Sheng computes the Stokes and central connection data for
$LG(2,4)\cong Q^3$~\cite{Sheng}. These are direct precedents on the
base. Fang compares central charges modified by the Gamma class
with integrals over SYZ T-dual branes for compact semipositive toric orbifolds \cite{FangCentral}. Fang and Peng Zhou
prove Gamma II for smooth toric Fano manifolds near the large-radius
limit by comparing thimbles and characteristic cycles \cite{FZ}.
Aleshkin, Fang and Wang use Fourier transforms to compare the central
charges of line bundles with mirror oscillatory integrals for toric Fano
orbifolds \cite[Theorems 4.5 and 5.5]{AFW}; the latter theorem
holds near the large-radius limit.
Tarasov and Varchenko identify equivariant Stokes bases for projective
spaces with exceptional bases in equivariant $K$-theory
\cite{TarasovVarchenko}. Cotti and Varchenko develop the comparison
with equivariant exceptional collections for projective spaces and
Grassmannians \cite{CV,CVSatake}. Our result concerns
the coupled fourfold \eqref{model}; its exceptional collection includes
objects other than line bundles.

Brown's toric fibration theorem and quantum Lefschetz give our
hypergeometric input \cite{Brown,CG,Givental}. Iritani and Koto construct
formal quantum $D$-module decompositions for projective bundles
\cite[Theorem 1.7]{IK},
while Hinault, Yu, Zhang and Zhang prove spectral decomposition theorems
and extend framings for formal and non-Archimedean $F$-bundles,
including uniqueness of the projective bundle decomposition map
\cite{HYYZ}. The reconstruction method of \cite{HYYZ} is also
incorporated in \cite[Section 5.8]{IK}. Iritani studies quantum cohomology through shift
operators and Fourier transforms in \cite{Fourier}.
Cotti's cyclic stratum and Borel--Laplace constructions
reconstruct solution spaces in substantially greater generality
\cite{CottiCyclic,Cotti,CottiML}.

Xu studies regularized limits of Stokes matrices for systems of
Poincar\'e rank one \cite{XuLimits} and derives connection formulas for
a higher rank analog of Painlev\'e VI \cite{XuConnection}.
Tang and Xu construct shrinking solutions from boundary data and
compute their monodromy \cite{TXBoundary}.
Lin and Xu compute Stokes matrices for a special confluent
hypergeometric system and its analogue with coefficients in an
enveloping algebra \cite{LinXu}.
Tang and Xu also express Stokes matrices through infinite matrix
products and relate the quantum construction to Yangians
\cite{TXYangians}.

Our comparison identifies the canonical sectorial continuations of
specified Gamma sections and follows their integral labels to the
chosen phase at the unit point. The general asymptotic existence
theorems with parameters in \cite{CGM} provide the analytic setting;
the identification of these labels requires the explicit integral
formulas and continuation used here.

Shen and Shoemaker study Gamma decompositions in the exceptional curve
specialization \cite{SS}; Cotti's fiberwise theory distinguishes vertical
and absolute quantum spectra \cite{CottiFiberwise}. Our connection
retains both Novikov variables. The positive integral argument follows
the mirror asymptotic approach to Gamma I \cite{GI}, with related period
comparisons in \cite{IritaniMS}. For del Pezzo surfaces of degree at
least three, including non-toric examples, Fang, Wang and Yan Zhou prove the
mirror symmetric Gamma statement to leading order for arbitrary $K$-group
insertions \cite[Theorem 1.5]{FWZ}. Exponential concentration for scalar
quantum periods is proved under a nonnegative convenient weak
Landau--Ginzburg hypothesis in \cite{LKH}; that scalar conclusion does
not identify the spinor sections or their connection constants.

The counterexamples in \cite{Revisiting} require the rightmost
eigenvalue to be distinguished from the spectral radius. Hu, Ke, Li and Su give related counterexamples to
Galkin's lower bound conjecture for toric projective bundles \cite{GalkinBound}.
Hu, Ke, Li and Su also prove propagation of the Gamma I property within a
connected $(\mathrm{SR})$-region and use it, together with additional
spectral and categorical input, to prove Gamma II for del Pezzo surfaces
\cite{Global}. Our direct transport at small parameters and explicit mutations
give an alternative to the multiphase criterion of \cite{HK}.
Convergence of the big Frobenius structure \cite{CottiConv} alone does
not supply the sector estimates needed for this comparison.

The paper is organized as follows: Section~\ref{sec:conventions} constructs the quantum connection and fixes
its Gamma and categorical conventions. Section~\ref{sec:asymptotics}
proves the integral identities and normalized asymptotics.
Section~\ref{sec:transportStokes} completes the continuation, mutations
and matrix calculation. Appendix~\ref{app:mirrorgeometry} gives a
coupled mirror realization and the local degeneration geometry;
Appendix~\ref{app:periodhelix} records quantitative period asymptotics
and the integral helix calculations. The full Gamma II proof concerns the
fourfold \eqref{model}; Section~\ref{sec:familytransfer} treats the
two spinor columns over odd quadrics. A comparison with primitive
integral thimbles, including the required tameness at infinity, lies
outside the scope of this paper.

\section{Quantum cohomology and Gamma framing}\label{sec:conventions}\label{sec:Euler}\label{sec:connection}
\subsection{Geometry and characteristic classes}
Write $\pi:X\to Q^3$, $H=\pi^*c_1(\OO_{Q^3}(1))$, and $\xi=c_1(\OO_\pi(1))$. Projectivization parametrizes rank-one quotients, so $\pi_*\OO_\pi(1)=\OO\oplus\OO(1)$. The classical ring is
\begin{equation}\label{classical}
 H^*(X,\C)=\C[H,\xi]/(H^4,\xi(\xi-H)),\qquad \int_X H^3\xi=2.
\end{equation}
The rational cohomology of $Q^3$ has one generator in each of degrees
$0,2,4,6$. The projective bundle formula doubles this rank, giving
$\dim H^*(X)=8$; the displayed ring therefore includes the entire
cohomology.

Let $\beta$ be a line in the section defined by the quotient onto $\OO$, and $f$ a fiber. The pairings are
\[
(H\cdot\beta,\xi\cdot\beta)=(1,0),\qquad
(H\cdot f,\xi\cdot f)=(0,1).
\]
The Mori cone is generated by $\beta,f$. Set $q=Q^\beta$ and $r=Q^f$. The relative Euler sequence gives
\[
 c_1(X)=2H+2\xi,
\]
so $X$ is Fano of index $2$. For $d\beta+ef$, the anticanonical degree is $2(d+e)$.

The Gamma class is
\begin{equation}\label{gamma}
 \Ga_X=\frac{\Gamma(1+H)^5}{\Gamma(1+2H)}
        \Gamma(1+\xi)\Gamma(1+\xi-H).
\end{equation}
We evaluate all formulas for characteristic classes in the classical ring
\eqref{classical}. We use the ordinary Chern character $\ch$ in
Riemann--Roch and the modified character
$\Ch(E)=\sum_j(2\pi i)^j\ch_j(E)$ in Gamma II.

We fix the classical frame
\begin{equation}\label{flatbasis}
 (b_0,\ldots,b_7)=(1,H,H^2,H^3,\xi,H\xi,H^2\xi,H^3\xi),
\end{equation}
where all products are \emph{cup products}. The basis vectors have degrees
$(0,1,2,3,1,2,3,4)$. The grading operator is
\begin{equation}\label{grading}
 \mu=\operatorname{diag}(-2,-1,0,1,-1,0,1,2).
\end{equation}
When the same monomials occur in the quotient presentation of the quantum
ring, they instead denote iterated quantum products. Section
\ref{sec:connection} gives the exact transition between the two meanings.
The Poincar\'e matrix $P_{ij}=\int_X b_i b_j$ is constant; equivalently,
$\int_X H^a\xi^b=2$ when $b\geq1$ and $a+b=4$, and is zero on homogeneous
classes of other total degree.

\subsection{Exceptional collections and the Gamma framing}
Let $S$ be the normalized rank-two spinor bundle on $Q^3$, with $c_1(S)=-H$. The standard collection is
\[
 (E_1,E_2,E_3,E_4)=(S,\OO,\OO(1),\OO(2)).
\]
It is a full exceptional collection; see Kapranov \cite{Kapranov} and \cite{HK}. Its ordinary Chern character is
\[
 \ch(S)=\frac4{1+e^H}=2-H+\frac{H^3}{12}.
\]
Orlov's projective bundle formula \cite{Orlov} gives the full collection
\begin{equation}\label{collection}
 (\pi^*E_1,\ldots,\pi^*E_4,
 \pi^*E_1\otimes\OO_\pi(1),\ldots,\pi^*E_4\otimes\OO_\pi(1)).
\end{equation}
We denote this ordered collection by $\mathcal F$ and its Euler matrix by
$G=G_X$. Identifying it with an asymptotic collection on a specified
Stokes sector requires an analytic comparison.

\begin{proposition}[Euler matrix]\label{naturalEuler}
The Euler Gram matrix of \eqref{collection} is
\[
 G_X=\begin{pmatrix}G_Q&B\\0&G_Q\end{pmatrix},\qquad
 G_Q=\begin{pmatrix}
 1&4&16&40\\0&1&5&14\\0&0&1&5\\0&0&0&1
 \end{pmatrix},\quad
 B=\begin{pmatrix}
 16&20&56&120\\4&6&19&44\\0&1&6&19\\0&0&1&6
 \end{pmatrix}.
\]
\end{proposition}
\begin{proof}
Use
\[
 \Td(Q^3)=1+\frac32H+\frac{13}{12}H^2+\frac12H^3,
 \qquad\int_{Q^3}H^3=2,
\]
and $\chi(A,B)=\int\ch(A^\vee)\ch(B)\Td(Q^3)$. The cross block is
\[
 B_{ij}=\chi_{Q^3}(E_i,E_j\otimes(\OO\oplus\OO(1))),
\]
since $\pi_*\OO_\pi(1)=\OO\oplus\OO(1)$. The relative $\OO(-1)$
has zero derived pushforward, so the lower block vanishes.
Substituting these data gives the entries of the matrix.
\end{proof}

Writing $\omega=2\pi\sqrt{-1}$, the two spinor Chern characters needed for
Gamma framing are
\[
 \Ch(\pi^*S)=2-\omega H+\frac{\omega^3H^3}{12},\qquad
 \Ch(\pi^*S\otimes\OO_\pi(1))
 =\left(2-\omega H+\frac{\omega^3H^3}{12}\right)e^{\omega\xi}.
\]
The Gamma-framed flat section is
\[
 \mathfrak{s}(E)=(2\pi)^{-2}\mathcal{L}(q,r,z)
 z^{-\mu}z^\rho\bigl(\Ga_X\Ch(E)\bigr),
 \qquad \rho=c_1(X)\cup,
\]
where $\mathcal{L}$ is the standard fundamental solution normalized at
the large-radius limit, with logarithms fixed on a chosen cover.
We use this Gamma integral framing \cite{Integral,GGI}, including its
common scalar, in every comparison of central connection matrices.
The two cohomology classes $\Ga_X\Ch(E)$ are therefore explicit in
the classical algebra of dimension eight. The remaining six are the
corresponding products with $e^{j\omega H}$ and
$e^{j\omega H+\omega\xi}$ for $j=0,1,2$.
These formulas specify the Gamma data at the regular end.
Section~\ref{sec:spinorMB} identifies the two spinor sections analytically.
Obtaining a compatible full set of labels for the line bundles requires
further analytic continuation. Until then, the Euler matrix above
serves only as an independently computed categorical comparison.

\subsection{Canonical columns and mutations}\label{sec:helix}
At a point with distinct eigenvalues $u_i$ of $C=c_1(X)\star$, let
$\psi_i$ be normalized idempotents. A Gamma section is called canonical
at exponent $u_i$ and phase $\phi$ if it is the corresponding column of
a sectorial fundamental solution, with full expansion
$e^{-u_i/z}(\psi_i+O(z))$ on a sector centered at $\phi$ of opening
strictly greater than $\pi$. We use the conventions and uniqueness
statements of \cite[Section 2.5]{GGI} and
\cite[Definition 2.5 and Lemma 2.6]{HK}.
A phase is admissible for the full spectrum when
$e^{-i\phi}(u_i-u_j)\notin\mathbb R$ for $i\ne j$.
The path construction also uses individual rays, which need only avoid
the other eigenvalues.

We define mutations of exceptional objects by
\begin{align}
 R_FE&=\operatorname{Cone}
 (E\to\operatorname{RHom}(E,F)^\vee\otimes F)[-1],\label{rightmutation}\\
 L_EF&=\operatorname{Cone}
 (\operatorname{RHom}(E,F)\otimes E\to F).\label{leftmutation}
\end{align}
Thus on ordered $K$-classes they act by
\begin{equation}\label{rightKmutation}
 R:(a,b)\mapsto(b,a-\chi(a,b)b),\qquad
 L:(a,b)\mapsto(b-\chi(a,b)a,a).
\end{equation}
The derived helix of a full foundation of length eight satisfies
\begin{equation}\label{helixperiod}
 F_{i+8}=\mathsf S_X^{-1}F_i=F_i\otimes\omega_X^{-1}[-4],
 \qquad \omega_X^{-1}=L^2U^2.
\end{equation}
We follow the standard helix--Serre conventions
\cite[Sections 3--5]{CDGHelix} and retain the shift at the derived level.
The foundation has length eight, whereas the Fano index is two.

For later matrix calculations one also has
\begin{equation}\label{baseKrelation}
 [\OO_{Q^3}(3)]=-4[S]+15[\OO]-11[\OO(1)]+5[\OO(2)].
\end{equation}
The Chern characters on both sides agree in every degree.
Since $K_0(Q^3)$ is free on the full exceptional foundation, injectivity
of the rational Chern character gives the integral equality.

\subsection{Mirror series and the quantum ring}
Define, for $k\in\Z$,
\[
 R_k(D,z)=
 \begin{cases}
  \displaystyle\prod_{m=1}^{k}(D+mz)^{-1},&k\geq0,\\[2pt]
  \displaystyle\prod_{m=k+1}^{0}(D+mz),&k<0.
 \end{cases}
\]
The empty product is $1$. Curve classes with negative $e-d$ are
included through the second case.

\begin{proposition}[Mirror series]\label{mirror}
After removing the divisor exponential, we use the normalization
$J=1+O(z^{-1})$. The small $J$-function is
\begin{equation}\label{I}
 J(q,r,z)=q^{H/z}r^{\xi/z}\sum_{d,e\geq0}q^d r^e
 \frac{\prod_{m=1}^{2d}(2H+mz)}{\prod_{m=1}^{d}(H+mz)^5}
 R_e(\xi,z)R_{e-d}(\xi-H,z).
\end{equation}
The mirror map is the identity.
\end{proposition}
\begin{proof}
Apply the Elezi--Brown theorem described in
Section~\ref{sec:introduction} to the split bundle and insert the quadric
hypergeometric series. Equivalently, realize $X$ as the smooth divisor
$2H$ in the toric fivefold $\PP_{\PP^4}(\OO\oplus\OO(1))$ and apply
quantum Lefschetz. The ambient toric divisor classes are five copies of
$H$, together with $\xi$ and $\xi-H$. Their respective curve degrees
$d,e,e-d$ give \eqref{I}. Since every nonzero effective class has
anticanonical degree at least two, the series has no additional
$z^{-1}$ term in the mirror map. All of $H^*(X)$ is ambient in this
realization.
\end{proof}

Put $D_q=zq\partial_q$ and $D_r=zr\partial_r$. They act on the divisor exponential as well as on the series. The following two operators annihilate \eqref{I}:
\begin{align}
 \mathcal P_q&=D_q^4-2q(2D_q+z)(D_r-D_q),\label{Pq}\\
 \mathcal P_r&=D_r(D_r-D_q)-r.\label{Pr}
\end{align}
For example, if $A_{d,e}$ is the cohomology-valued coefficient, the base recurrence is
\[
 (H+dz)^4A_{d,e}
 =2(2H+(2d-1)z)(\xi-H+(e-d+1)z)A_{d-1,e}.
\]
The fiber recurrence is
\[
 (\xi+ez)(\xi-H+(e-d)z)A_{d,e}=A_{d,e-1}.
\]
Taking missing coefficients to be zero extends these identities to the
boundaries: the required vanishing follows from $H^4=0$ and
$\xi(\xi-H)=0$.

\begin{theorem}[Quantum ring]\label{ring}
The small quantum ring is
\begin{equation}\label{ringformula}
 QH^*(X)=\C[q,r,H,\xi]\big/
 \bigl(H^4-4qH(\xi-H),\ \xi(\xi-H)-r\bigr).
\end{equation}
It is free of rank eight over $\C[q,r]$, with quantum monomial basis
\begin{equation}\label{basis}
 1,H,H^2,H^3,\xi,H\xi,H^2\xi,H^3\xi.
\end{equation}
\end{theorem}
\begin{proof}
The principal-symbol relation between annihilators of $J$ and quantum
multiplication, applied to \eqref{Pq}--\eqref{Pr}, gives the two stated
relations. Regard $q,r$ as coefficients and use graded lexicographic
order with $\xi>H$. The leading monomials $H^4$ and $\xi^2$ are
relatively prime, so the two monic polynomials form a Groebner basis.
The quotient therefore has the eight standard monomials \eqref{basis}.
Their images also form a basis of quantum cohomology: corrections of
positive degree lower the ordinary cohomological degree, so the
transition matrix from classical monomials is unipotent triangular
when the monomials are ordered by degree. The surjection defined by the
relations is consequently an isomorphism.
\end{proof}

\begin{lemma}[Convergence of the small series]\label{smallconv}
For fixed $z\ne0$, after removing the divisor exponential the series
\eqref{I} and all its parameter derivatives converge absolutely and
locally uniformly for $(q,r)\in\C^2$.
\end{lemma}
\begin{proof}
Homogeneity reduces the proof to $z=1$. Expanding in the nilpotent
classical cohomology variables changes the factorial bounds by at most
a fixed power of $1+\log(d+e+2)$. For $e=d+n\geq d$, the scalar
factorial coefficient satisfies
\[
 \frac{(2d)!}{(d!)^5(d+n)!n!}
 \leq\frac{4^d}{(d!)^4(n!)^2}.
\]
This majorant is summable with arbitrary bounded weights
$|q|^d|r|^{d+n}$. For $e<d$, the coefficient norm is bounded by a fixed
logarithmic power times
$4^d/((d!)^2(e!)^2)$, because
\[
 \frac{(2d)!(d-e-1)!}{(d!)^5e!}
 \leq\frac{4^d}{(d!)^2(e!)^2}.
\]
The inequality follows from $(d-e-1)!/d!\leq1/e!$ and
$\binom{2d}{d}\leq4^d$.
Multiplying either majorant by any fixed polynomial in $d,e$ preserves
summability. These bounds prove convergence and justify differentiating
the series term by term in the parameters. We choose a branch of the
divisor exponential on the logarithmic cover of $(\C^*)^2$.
\end{proof}

\subsection{The connection in the cup-product frame}
We reduce the system generated by divisor operators to scalar master
functions, following the approach based on cyclic vectors in
\cite{CottiCyclic}. We give the complete first-order system and its
normalized frame for this model.

Let $f$ be any scalar component of \eqref{I}, or any scalar solution of
\eqref{Pq}--\eqref{Pr}. Put
\[
 F=(f,\Dq f,\Dq^2f,\Dq^3f,\Dr f,\Dq\Dr f,
       \Dq^2\Dr f,\Dq^3\Dr f)^{\mathsf T}.
\]
Define matrices $A_q,A_r$ by
\begin{equation}\label{jetmatq}
A_q=\begin{pmatrix}
0&1&0&0&0&0&0&0\\
0&0&1&0&0&0&0&0\\
0&0&0&1&0&0&0&0\\
0&-2qz&-4q&0&2qz&4q&0&0\\
0&0&0&0&0&1&0&0\\
0&0&0&0&0&0&1&0\\
0&0&0&0&0&0&0&1\\
2qrz&4qr&0&0&0&0&0&0
\end{pmatrix},
\end{equation}
\begin{equation}\label{jetmatr}
A_r=\begin{pmatrix}
0&0&0&0&1&0&0&0\\
0&0&0&0&0&1&0&0\\
0&0&0&0&0&0&1&0\\
0&0&0&0&0&0&0&1\\
r&0&0&0&0&1&0&0\\
0&r&0&0&0&0&1&0\\
0&0&r&0&0&0&0&1\\
2qrz&4qr&0&r&0&0&0&0
\end{pmatrix}.
\end{equation}

\begin{proposition}\label{pfaffian}
The scalar system is equivalent to
\begin{equation}\label{jetpfaff}
 \Dq F=A_qF,\qquad \Dr F=A_rF.
\end{equation}
The latter system is integrable and has rank eight on $(\C^*)^2$ for
$z\ne0$.
\end{proposition}
\begin{proof}
All rows follow directly from differentiation except for the reductions
of $\Dq^4f$ and $\Dq^4\Dr f$. The first is \eqref{Pq}; for the second,
commute $\Dr$ through \eqref{Pq} and use
$\Dr^2-\Dq\Dr=r$ on solutions. This gives
\[
 \Dq^4\Dr f=4qr\Dq f+2qrzf.
\]
Conversely, the first rows of \eqref{jetpfaff} identify all components as
the indicated derivatives of $F_0$, and the remaining rows yield the
scalar equations. Direct multiplication gives
\[
 zq\partial_qA_r-zr\partial_rA_q+A_rA_q-A_qA_r=0.
\]
A flat rank-eight first-order system has eight locally independent
solutions. The equivalence therefore gives exactly eight independent
scalar solutions. The hypergeometric components span this space:
their leading logarithmic terms at the large-radius limit are the
eight independent coefficients of $q^{H/z}r^{\xi/z}$, and $H,\xi$
generate the classical cohomology.
\end{proof}

Let $E_{ij}$ denote the matrix unit with indices $0\leq i,j\leq7$, and set
\begin{equation}\label{gauge}
 N=-2qE_{31}+2qE_{34}+2qrE_{70},\qquad T=I+N.
\end{equation}
Here $N^2=0$ and $T^{-1}=I-N$. For $Y=T^{-1}F$, the equations become
\begin{equation}\label{flatdual}
 \Dq Y=M_qY,\qquad \Dr Y=M_rY,
 \qquad M_i=T^{-1}A_i|_{z=0}T.
\end{equation}
Indeed,
$q\partial_qT=(\partial_zA_q)T$ and
$r\partial_rT=(\partial_zA_r)T$, which remove all $z$-linear entries.

\begin{theorem}[Fixed-frame multiplication]\label{fixedframe}
In the classical frame \eqref{flatbasis}, the matrices of multiplication
by $H$ and $\xi$ are $M_H=M_q^{\mathsf T}$ and
$M_\xi=M_r^{\mathsf T}$, respectively. Equivalently, the products are
as follows:
\begin{center}
\renewcommand{\arraystretch}{1.15}
\begin{tabular}{c|l|l}
\toprule
$b$ & $H\star b$ & $\xi\star b$\\
\midrule
$1$ & $H$ & $\xi$\\
$H$ & $H^2$ & $H\xi$\\
$H^2$ & $H^3-2qH+2q\xi$ & $H^2\xi$\\
$H^3$ & $-2qH^2+2qH\xi$ & $H^3\xi$\\
$\xi$ & $H\xi$ & $r+H\xi$\\
$H\xi$ & $H^2\xi$ & $rH+H^2\xi$\\
$H^2\xi$ & $H^3\xi+2qr$ & $2qr+rH^2+H^3\xi$\\
$H^3\xi$ & $2qrH$ & $2qrH+rH^3$\\
\bottomrule
\end{tabular}
\end{center}
All the monomials in this table are classical cup products.
\end{theorem}
\begin{proof}
The gauge computation gives the displayed coefficients. We identify the
resulting frame with the classical frame by applying the corrected
differential operators to the cohomology-valued $J$-function. After removing $q^{H/z}r^{\xi/z}$, the leading terms as
$z\to\infty$ of the uncorrected jets are
\[
 b_0,b_1,b_2,\ b_3+2q(b_4-b_1),\ b_4,b_5,b_6,\ b_7+2qr b_0.
\]
The additional constant term in the fourth component comes only from
$(d,e)=(1,0)$. The last component receives such a term only from
$(d,e)=(1,1)$; terms with $e=0<d$ vanish after $\Dr$ because
$\xi(\xi-H)=0$ in classical cohomology. The degrees of the remaining
components exclude additional constant terms and positive powers of
$z$. Subtracting $N$ removes exactly these corrections.

The corrected jets therefore have leading terms $b_i$ followed only
by negative powers. They give the normalized frame of the quantum
$D$-module determined by the mirror theorem, through Birkhoff
normalization using divisor operators. Differentiation acts on these
solutions, indexed by the classical basis, by the transpose of
multiplication. This gives the table.
The matrices satisfy $[M_H,M_\xi]=0$,
$M_i^{\mathsf T}P=PM_i$, and
$q\partial_qM_\xi=r\partial_rM_H$.
\end{proof}

\begin{corollary}[The full small quantum connection]\label{connectionfull}
In the fixed classical frame, put $C=2(M_H+M_\xi)$ and take $\mu$ from
\eqref{grading}. The small quantum connection, in the convention for
which flat sections have exponential factor $e^{-u/z}$, is
\begin{equation}\label{qconnection}
 \nablaq=q\partial_q+\frac{M_H}{z},\qquad
 \nablar=r\partial_r+\frac{M_\xi}{z},\qquad
 \nabla_{z\partial_z}=z\partial_z-\frac C z+\mu.
\end{equation}
This connection is flat. We use the constant grading operator in the
frame obtained by unipotent normalization; it does not apply directly
in the uncorrected quantum monomial frame.
\end{corollary}
\begin{proof}
The $q,r$ curvature vanishes by the identities in the preceding proof.
The two remaining identities are
\[
 q\partial_q C=M_H+[M_H,\mu],\qquad
 r\partial_r C=M_\xi+[M_\xi,\mu],
\]
which follow entry by entry from the table and the degrees in
\eqref{grading}. These are exactly the mixed $z$-curvature equations.
\end{proof}

\subsection{Cohomological chiral symmetry}\label{sec:chiral}
We use the parity of the \emph{complex} cohomological degree.
The ordinary cohomology of $X$ lies entirely in even real degrees.
Define
\begin{equation}\label{paritymatrix}
 \mathsf D|_{H^{2p}(X)}=(-1)^p,\qquad
 \mathsf D=\operatorname{diag}(1,-1,1,-1,-1,1,-1,1)=e^{\pi i\mu}.
\end{equation}
The term \emph{cohomological chiral symmetry} refers here to the
anticommutation with divisor multiplication proved below.
For non-Hermitian Hamiltonians, this relation is called
\emph{sublattice symmetry}; chiral symmetry in that terminology involves
Hermitian conjugation \cite[Section II]{KSUS}.
Our forms are complex bilinear Poincar\'e forms. The following statements
make no claim about Hermitian Hamiltonians or the classification of
topological phases.

\begin{proposition}[Grading and block form]\label{chiralblocks}
At fixed $q,r$, the operator $\mathsf D$ is an automorphism of the
small quantum algebra and an isometry of its Poincar\'e form. It satisfies
\begin{equation}\label{chiralidentities}
 \mathsf D^2=I,\quad
 \mathsf D M_H\mathsf D=-M_H,\quad
 \mathsf D M_\xi\mathsf D=-M_\xi,\quad
 \mathsf D C\mathsf D=-C,\quad [\mathsf D,\mu]=0.
\end{equation}
Write $\mathcal H_\pm=\ker(\mathsf D\mp I)$. In the ordered cup bases
\[
 \mathcal H_+:(1,H^2,H\xi,H^3\xi),\qquad
 \mathcal H_-:(H,H^3,\xi,H^2\xi),
\]
the anticanonical multiplication has the block form
\begin{equation}\label{chiralblockform}
 C=\begin{pmatrix}0&K_+\\K_-&0\end{pmatrix},\qquad
 C^2=\begin{pmatrix}K_+K_-&0\\0&K_-K_+\end{pmatrix},
\end{equation}
where
\begin{equation}\label{chiralexplicitblocks}
 \begin{split}
 K_+&=\begin{pmatrix}
 0&0&2r&8qr\\2&-4q&0&2r\\2&4q&4&0\\0&2&0&4
 \end{pmatrix},\\
 K_-&=\begin{pmatrix}
 2&-4q&2r&8qr\\0&2&0&2r\\2&4q&0&0\\0&2&4&0
 \end{pmatrix}.
 \end{split}
\end{equation}
\end{proposition}
\begin{proof}
If $\alpha\in H^{2p}(X)$ and $\beta\in H^{2s}(X)$, the coefficient
of $q^d r^e$ in $\alpha\star\beta$ has complex degree
$p+s-2(d+e)$ by the dimension constraint. Its parity is $p+s$,
so $\mathsf D(\alpha\star\beta)=\mathsf D\alpha\star\mathsf D\beta$
without changing $q,r$. The Poincar\'e form pairs degrees summing to
four, hence $\mathsf D^{\mathsf T}P\mathsf D=P$.
Since $H,\xi$ have odd complex degree, the anticommutation identities
follow; the grading operators commute. Finally, ordering the rows and
columns of the multiplication table in Theorem~\ref{fixedframe}
according to $\mathcal H_+\oplus\mathcal H_-$ gives
\eqref{chiralexplicitblocks}.
\end{proof}

\begin{corollary}[Squared spectrum]\label{chiralspectrum}
For all $q,r$, one has
\begin{equation}\label{chiralsquaredpoly}
 \begin{split}
 \det(uI-C)&=\det(u^2I-K_+K_-),\\
 \det(\lambda I-K_+K_-)
  &=(\lambda-4r)\bigl[\lambda^3+(16q-12r)\lambda^2\\
  &\hspace{20mm}+(48r^2-1280qr)\lambda-64r(r+16q)^2\bigr].
 \end{split}
\end{equation}
The generalized eigenspaces at $u$ and $-u$ have the same Jordan
structure and are interchanged by $\mathsf D$. At a point with distinct
eigenvalues, it sends the primitive idempotent at $u$ to that at $-u$
with equal Poincar\'e norm. For $q,r>0$,
zero is not an eigenvalue.
\end{corollary}
\begin{proof}
The determinant identity follows from the block form, first for
$u\ne0$ by a Schur complement and then identically as a polynomial.
Expansion of the four-dimensional determinant gives
\eqref{chiralsquaredpoly}. Anticommutation gives
$\mathsf D(C-uI)^k=(-1)^k(C+uI)^k\mathsf D$; the remaining assertions
follow from the algebra automorphism and isometry in
Proposition~\ref{chiralblocks}. At $u=0$, the determinant is
$256r^2(r+16q)^2$, which is positive for $q,r>0$.
\end{proof}

The block form expresses the standard Fano-index spectral symmetry
\cite[Remark 3.1.3]{GGI} in the fixed cup frame. Squaring reduces the
spectral calculation to four dimensions. It does not split the quantum
connection along the parity summands, since $C$ and divisor
multiplication exchange them. The pairing $u\leftrightarrow-u$
records spectral symmetry; Theorem~\ref{spectral} separately determines
the spectral radius and distinguishes it from the rightmost eigenvalue.

\subsection{Spectrum and idempotents}\label{sec:block}
Let $C(q,r)$ denote multiplication by $2H+2\xi$.
\begin{proposition}[Characteristic polynomial]\label{char}
\begin{align}
 \det(u-C(q,r))={}&(u^2-4r)\bigl[
 u^6+(16q-12r)u^4\notag\\
 &\hspace{19mm}+(48r^2-1280qr)u^2-64r(r+16q)^2\bigr].\label{charpoly}
\end{align}
For $q,r\ne0$, the eight eigenvalues, counted with multiplicity, are
obtained as follows. Two correspond to $H=0$, $\xi=\pm\sqrt r$ and
$u=\pm2\sqrt r$. The other six satisfy
\begin{equation}\label{hroots}
 H^6+4qH^4-16q^2r=0,\quad
 \xi=H+\frac{H^3}{4q},\quad
 u=4H+\frac{H^3}{2q}.
\end{equation}
\end{proposition}
\begin{proof}
The first relation in \eqref{ringformula} gives the cases $H=0$ and
$H^3=4q(\xi-H)$. Combining the latter with the second relation yields
\eqref{hroots}. Substituting the expression for $u$ and multiplying by
the $H=0$ factor gives \eqref{charpoly}. As a polynomial identity, the
formula also holds when eigenvalues collide.
\end{proof}

Write $k=r/q>0$, and let $a>0$ be the unique positive root of
\begin{equation}\label{a}
 a^3+4a^2=16k.
\end{equation}
Define
\begin{equation}\label{T}
 T(q,r)=\sqrt q\,\sqrt a\left(4+\frac a2\right),\qquad
 k_0=\frac{71-17\sqrt{17}}4=0.226801091124942\ldots.
\end{equation}

\begin{theorem}[Two distinct transitions]\label{spectral}
For every $q,r>0$, $T(q,r)$ is the unique simple eigenvalue with maximal real part. The following more precise statements hold.
\begin{enumerate}[label=(\roman*)]
\item The quantum ring is semisimple exactly when $k\ne16/27$.
\item If $0<k<k_0$, its spectral radius is attained by a purely imaginary pair, not by $T$. At $k=k_0$, the peripheral spectrum is $\{T,-T,iT,-iT\}$.
\item If $k>k_0$, its peripheral spectrum is exactly $\{T,-T\}$, each simple. In particular, Property $\mathcal O$ holds at $q=r=1$.
\end{enumerate}
These statements concern the deformed operator at general $q,r$.
Failure of the index-two peripheral condition away from $(1,1)$ does
not contradict the original Conjecture $\mathcal O$ for this $X$.
\end{theorem}
\begin{proof}
Set $s=H^2/q$ on the six-point branch. Then
\[
 s^3+4s^2-16k=0.
\]
After extracting its positive root $a$, the other two roots solve
\[
 s^2+(a+4)s+a(a+4)=0,
\]
with discriminant $(a+4)(4-3a)$. They are both negative when $0<k<16/27$, collide at $a=4/3$, and are conjugate nonreal roots when $k>16/27$.

When these two roots are negative, their four $u$-values are purely imaginary. Also
\[
 \frac{T^2}{q}-4k=3a^2+16a>0,
\]
so $T$ is strictly to the right of the other real eigenvalue $2\sqrt r$, and hence is simple and rightmost. The same holds at the collision.

For $k>16/27$, put $y=u^2/q$ and let $y_0=T^2/q=a(a+8)^2/4$. The other two $y$ roots are conjugate. Their product is $64k(k+16)^2/y_0$. A direct calculation gives
\[
 y_0^3-64k(k+16)^2
 =\frac{a^2(a+8)^2}{16}
 (9a^4+84a^3+496a^2+1024a-1024)>0,
\]
since $a>4/3$. Thus $T$ has strictly larger modulus than all four nonreal eigenvalues.
This proves the assertion about the rightmost eigenvalue throughout the
positive quadrant, as well as the description of the peripheral
spectrum in this range.

For $0<k<16/27$, the negative $s$ root in $(-4,-8/3)$ yields the imaginary pair of largest modulus. Equality of this modulus with $T$ occurs exactly when
\[
 a^4+14a^3+108a^2+192a-256=0.
\]
The left side is strictly increasing for $a>0$ and has exactly one positive root. Eliminating $a$ with \eqref{a} gives
\[
 2k^2-71k+16=0.
\]
The root in $(0,16/27)$ is $k_0$. As $k\downarrow0$, $T\to0$ while the largest imaginary modulus tends to $4\sqrt q$, which fixes the sides of the transition.

The two points with $H=0$ are reduced for $q,r>0$. On the remaining
branch, \eqref{hroots} has repeated roots exactly at $k=16/27$.
This branch remains disjoint from $H=0$ for $r>0$, proving the
semisimplicity criterion.
\end{proof}

\begin{proposition}\label{proj}
The element
\begin{equation}\label{p0}
 p_0=\frac12-\frac{H^3\xi}{4qr}
\end{equation}
is an idempotent for the quantum product. Its multiplication operator
$\mathcal P=M_{p_0}$ has rank two and equals
\begin{equation}\label{Poperator}
 \mathcal P=I-\frac{M_H^3M_\xi}{4qr}.
\end{equation}
Its image is spanned by $p_0$ and
\begin{equation}\label{p1}
 p_1=\xi\star p_0=\frac{\xi-H}{2}-\frac{H^3}{4q}.
\end{equation}
The multiplication on this subspace satisfies
$H\star p_i=0$, $p_0\star p_1=p_1$, and
$p_1\star p_1=rp_0$.
\end{proposition}
\begin{proof}
In quantum monomials the element is
$1-H^{\star3}\star\xi/(4qr)$. On the six-point branch
$H^{\star3}\star\xi=4qr$, whereas on the $H=0$ branch it is zero.
The same identity holds on the length-two local algebras at the caustic,
or directly by the relations of Theorem~\ref{ring}. Hence it is the
projector onto the two $H=0$ points. The table in
Theorem~\ref{fixedframe} gives
$H^{\star3}\star\xi=H^3\xi+2qr$ in the classical frame, proving
\eqref{p0}. Multiplying it by $\xi$ gives \eqref{p1}; a second multiplication
uses $\xi\star\xi=H\star\xi+r$ and $H\star p_0=0$.
\end{proof}

The Poincar\'e pairing on this subspace is
\begin{equation}\label{blockpairing}
 (p_0,p_0)=-\frac1{2qr},\quad
 (p_1,p_1)=-\frac1{2q},\quad (p_0,p_1)=0.
\end{equation}
Here $(a,b)=\int_Xa\cup b$, equivalently the Frobenius pairing for the
quantum algebra.

For $s^2=r$ the two primitive idempotents are
\begin{equation}\label{es}
 e_s=\tfrac12(p_0+s^{-1}p_1)
 =\frac14+\frac{\xi-H}{4s}-\frac{H^3}{8qs}-\frac{H^3\xi}{8qr}.
\end{equation}
They have $C$-eigenvalue $2s$ and square norm $-1/(4qr)$.

\begin{lemma}[Norm on the six-point branch]\label{idempotentnorms}
For a simple root $h$ of $P(t)=t^6+4qt^4-16q^2r$, the primitive
idempotent $e_h$ has square norm
\[
 (e_h,e_h)=\frac{4q}{h^4(3h^2+8q)}.
\]
\end{lemma}
\begin{proof}
In the notation of quantum polynomials, the element
$H P(H)/(h(H-h)P'(h))$ equals $e_h$: it takes the value one at $H=h$
and zero at the other five roots and at the two $H=0$ points.
The quotient by $H-h$ is a polynomial. The Frobenius trace of $H^j$
vanishes for $1\leq j\leq5$, while the quantum ring relations give
$\int_X H^{\star6}=8q$. Its trace is therefore $8q/(hP'(h))$.
Idempotence identifies this trace with the square norm, giving the
formula.
\end{proof}

\section{Gamma-normalized asymptotics}\label{sec:asymptotics}\label{sec:spinorMB}
\subsection{Mellin--Bessel integral representations}
We identify the central charges from their Gamma data at the regular
singular end, then derive their asymptotics. The spinor analysis in \cite{HK}
and the explicit monodromy calculation in \cite{Sheng} concern
$Q^3$; the Bessel factor below accounts for the fiber direction.

For $E\in K_0(X)$ define the unnormalized scalar central charge
\begin{equation}\label{centraldef}
 Z_E(q,r)=\int_X \widehat\Gamma_X\Ch(E)J(q,r,-1).
\end{equation}
We initially use positive $q,r$ and real logarithms. The Poincar\'e pairing
and the standard identity $\mathcal L(-z)^*\mathcal L(z)=I$ give
\begin{equation}\label{scalarmasterGamma}
 (1,\mathfrak s(E)(q,r,z))
   =(2\pi)^{-2}z^{-2}Z_E(q/z^2,r/z^2).
\end{equation}
Logarithms on the right are \emph{transported}:
$\log(q/z^2)=\log q-2\log z$, and likewise for $r$.
Replacing these by principal logarithms after a turn would change the
object. To check \eqref{scalarmasterGamma}, use
$\mathfrak s(E)=(2\pi)^{-2}\mathcal L(z)z^{-\mu}z^\rho
\widehat\Gamma_X\Ch(E)$, move $\mathcal L$ across the pairing, and
apply the degree homogeneity $\deg q=\deg r=2$ to $J(-z)$.
This also fixes the otherwise significant factor $z^{-2}$.

Put $S_0=\pi^*S$ and $S_1=\pi^*S\otimes\OO_\pi(1)$. For a vertical
contour $\operatorname{Re}s=c$, where $0<c<1/2$, define
\begin{align}
 K_0(s)&=\frac{\Gamma(s)^5}{\Gamma(2s)},&
 C(s)&=\frac{2^{2-2s}}{\sqrt\pi}\Gamma(s)^4\Gamma(1/2-s),
 \label{MBkernels}\\
 K_S(s)&=e^{-\pi i s}C(s),&
 \mathcal K_s(r)&=2r^{-s/2}K_s(2\sqrt r).
 \label{Besselkernel}
\end{align}
The symbol $K_s$ on the right of \eqref{Besselkernel} is the modified
Bessel function, not the Mellin kernel $K_S$.
Euler reflection and duplication yield the exact identity
\begin{equation}\label{spinorfactor}
 K_S(s)=\frac{4}{1+e^{2\pi i s}}\,K_0(s).
\end{equation}
Its multiplier is the Chern character of the normalized spinor under
the substitution $H\mapsto2\pi i s$. We use this identity before
estimating the contours, since it exhibits the simple poles $s=n+1/2$
to the right of the initial contour.

\begin{theorem}[Gamma-marked Bessel transforms]\label{MBmarked}
For $q,r>0$, with the above logarithms, all integrals below converge
absolutely and
\begin{align}
 Z_{\OO}(q,r)&=\frac1{2\pi i}\int_{c-i\infty}^{c+i\infty}
        K_0(s)q^{-s}\mathcal K_s(r)\,ds,\label{MBstructure}\\
 Z_{S_0}(q,r)&=\frac1{2\pi i}\int_{c-i\infty}^{c+i\infty}
        K_S(s)q^{-s}\mathcal K_s(r)\,ds.\label{MBspinor}
\end{align}
Let $\mathcal M_r$ denote continuation with
$\log r\mapsto\log r-2\pi i$ while $\log q$ is fixed. Then
\begin{align}
 Z_{S_1}&=\mathcal M_r Z_{S_0}\notag\\
 &=4Z_{\OO}-Z_{S_0}+2\pi i\,\mathcal I(q,r),\label{spinorIidentity}\\
 \mathcal I(q,r)&=\frac1{2\pi i}\int_{c-i\infty}^{c+i\infty}
 C(s)q^{-s}r^{-s/2}I_s(2\sqrt r)\,ds.\label{Itransform}
\end{align}
These formulas represent the central charges of the specified Gamma
objects and retain their integral labels.
\end{theorem}
\begin{proof}
For $c_t>c$, Mellin inversion of $e^{-r/v}$ in the convergent integral
$\int_0^\infty v^{-s}e^{-v-r/v}\,dv/v$ gives
\[
 \mathcal K_s(r)=\frac1{2\pi i}\int_{c_t-i\infty}^{c_t+i\infty}
       \Gamma(t)\Gamma(t-s)r^{-t}\,dt.
\]
Together with \eqref{spinorfactor} this gives double Barnes kernels
\begin{equation}\label{doubleMB}
 F_0(s,t)=\frac{\Gamma(s)^5\Gamma(t)\Gamma(t-s)}{\Gamma(2s)},
 \qquad F_S(s,t)=\frac4{1+e^{2\pi i s}}F_0(s,t).
\end{equation}
For example, take $c=1/4,c_t=3/4$. Stirling's formula bounds $F_0$
by the exponential factor
\[
 \exp\left(-\frac{3\pi}{2}|\operatorname{Im}s|
           -\frac\pi2|\operatorname{Im}t|
           -\frac\pi2|\operatorname{Im}(t-s)|\right)
\]
times a fixed polynomial. The spinor multiplier is bounded on the
chosen line. These bounds justify convergence, differentiation under the integral,
and Fubini's theorem. The Bessel identity is also recorded in \cite{DLMF}.

We verify the differential equations and the Gamma marking. The kernel
ratios are
\begin{equation}\label{Barnesratios}
 \frac{F_0(s+1,t)}{F_0(s,t)}
 =\frac{s^4}{2(2s+1)(t-s-1)},\qquad
 \frac{F_0(s,t+1)}{F_0(s,t)}=t(t-s).
\end{equation}
They remain true for $F_S$. Moving a contour by one unit gives the two
operators \eqref{Pq}--\eqref{Pr} at $z=1$.
In the $s$ shift from $c-1$ to $c$, the factor $s^4$ cancels
the only intervening pole at zero,
and the positive half-integer spinor poles are not crossed. In the $t$
shift, the factors $t(t-s)$ cancel the two intervening pole families.
Thus no inhomogeneous residue term is omitted. The system of eight jets
therefore reconstructs a solution in the fixed frame, with the
appropriate duality convention.

We next identify its leading logarithmic term at $q=r=0$. The two
leading fiber poles are $t=0$ and $t=s$, and the zero base pole gives
\begin{equation}\label{zerologresidue}
 \operatorname{Res}_{s=0}
 K_a(s)q^{-s}\{\Gamma(-s)+\Gamma(s)r^{-s}\},
 \qquad a=0,S.
\end{equation}
This is the $(q^0r^0)$ logarithmic polynomial; the other pole translates
have strictly positive curve degrees. One can first shift the $s$
contour to $-\epsilon$ with $t$ fixed, and then take the two leading
fiber residues. Uniform Stirling estimates on the intervening contours
justify the corresponding iterated regular-singular expansion.
Equivalently the kernel $F_a$ has a convergent Frobenius expansion, since
its reconstructed first-order system has only the nilpotent classical
residues at the two coordinate axes.

The projective-bundle pushforward in the classical ring is
\[
 \pi_* f(H,\xi)=\frac{f(H,H)-f(H,0)}H.
\]
Since $K_0(s)=2\widehat\Gamma_{Q^3}(s)/s^4$, expression
\eqref{zerologresidue} is exactly
\[
 \int_X\widehat\Gamma_X\Ch(E)
       e^{-H\log q-\xi\log r},\qquad E=\OO,S_0.
\]
Indeed the fiber numerator reduces to
$\Gamma(1+H)r^{-H}-\Gamma(1-H)$, divided by $H$.
This explains both the factor $2$ and the sign of each logarithm.
The regular-singular recursion fixes the connection constants: its
operators are invertible, since each is $dI+$ a nilpotent residue for a
positive integer $d$. Its eight initial logarithmic polynomials are
independent because the cup algebra is generated by $H,\xi$.
Thus agreement of the zero-degree polynomial and the differential
system proves \eqref{MBstructure}--\eqref{MBspinor}.

Under the specified clockwise turn, the two fiber branches give
\begin{equation}\label{Besselturn}
 \mathcal M_r\mathcal K_s
 =e^{2\pi i s}\mathcal K_s
   +2\pi i e^{\pi i s}r^{-s/2}I_s(2\sqrt r).
\end{equation}
This follows from
$K_s(xe^{-\pi i})=e^{\pi i s}K_s(x)+\pi i I_s(x)$
\cite[Eq.~10.34.2]{DLMF}. The zero-degree polynomial of the continued
solution has $r^{-H}$ replaced by $e^{2\pi iH}r^{-H}$ in its $\xi=H$
term, and is therefore the marking of $S_1$. Finally
$e^{2\pi i s}K_S=4K_0-K_S$ by \eqref{spinorfactor}, giving
\eqref{spinorIidentity}. The $I$ integral converges by its Poisson
representation and Stirling's formula, as detailed below. This calculation uses analytic
continuation; inserting $e^{2\pi it}$ directly in the original double
vertical integral is not justified.
\end{proof}

\begin{remark}[Scope of the analytic marking]
The proof identifies $S_0,S_1$ in the quantum Gamma lattice defined above.
Identifying the Barnes contours as primitive integral four-dimensional
thimbles of $W$ requires a separate comparison of integral structures.
\end{remark}

\subsection{A Mellin--Bessel transfer estimate}\label{sec:angulargaps}
The same contour argument applies in odd dimension. We prove the estimate
for an odd integer $\nu=2m+1\geq3$; the kernels defined above correspond
to $\nu=3$. Set
\begin{equation}\label{familykernels}
 K_{0,\nu}(s)=\frac{\Gamma(s)^{\nu+2}}{\Gamma(2s)},\qquad
 C_\nu(s)=\frac{2^{m+1-2s}}{\sqrt\pi}
       \Gamma(s)^{\nu+1}\Gamma(1/2-s),\qquad
 K_{S,\nu}(s)=e^{-\pi i s}C_\nu(s).
\end{equation}
The following estimates concern scalar transforms. Their interpretation
as Gamma sections requires the separate marking argument below.
Throughout the statement $\nu$ is fixed; no estimate uniform as
$\nu\to\infty$ is asserted.

\begin{lemma}[Order-uniform transfer bounds]\label{Mellinsectorbound}
Fix a compact subset of $q,r>0$, a real $R>0$ that is not a positive
half-integer, and put $x=2\sqrt r/z$. The absolute integrals over
$\operatorname{Re}s=R$ of the following expressions have the indicated
bounds on closed subsectors of the stated open sectors:
\begin{align}
 C_\nu(s)e^{-\pi i s}(q\sqrt r)^{-s}z^{\nu s}K_s(x):
 &\quad O_R\bigl(|z|^{\nu R+1/2}|e^{-x}|\bigr),
 &-\pi/2<\arg z<\pi/2+\frac{2\pi}{\nu+1};\label{boundS}\\
 C_\nu(s)(q\sqrt r)^{-s}z^{\nu s}K_s(xe^{-\pi i}):
 &\quad O_R\bigl(|z|^{\nu R+1/2}|e^x|\bigr),
 &-\pi/2-\frac\pi\nu<\arg z<0;\label{boundSU}\\
 C_\nu(s)(q/z^{\nu-1})^{-s}(r/z^2)^{-s/2}I_s(x):
 &\quad O_R\bigl(|z|^{\nu R+1/2}|e^x|\bigr),
 &|\arg z|<\pi/2.\label{boundI}
\end{align}
The pole-free kernel $K_{0,\nu}$ in place of $K_{S,\nu}$ in
\eqref{boundS} has the same bound for $|\arg z|<\pi/2$.
All logarithms are transported from the positive ray. The estimates
hold after a fixed number of parameter derivatives, with the resulting
finite loss of powers of $|z|$.
\end{lemma}
\begin{proof}
Write $s=R+i\tau$. We use the Laplace representation with the same branches:
\begin{equation}\label{rotatedK}
 K_s(x)=\frac{\sqrt\pi\,e^{-x}}{\sqrt{2x}\,\Gamma(s+1/2)}
 \int_{e^{i\varphi}\mathbb R_+}e^{-u}u^{s-1/2}
       (1+u/(2x))^{s-1/2}\,du.
\end{equation}
It follows from \cite[Eq.~10.32.8]{DLMF}, initially in the right
half-plane, followed by contour deformation. Require
$|\varphi|<\pi/2$ and $|\varphi+\theta_x|<\pi$, where
$\theta_x=-\arg x$. These inequalities keep the contour away from the
branch point $-2x$. The two signs of $\tau$ may use different contours.
After division by $\Gamma(s+1/2)$, the dependence on the order is bounded
by a polynomial in $|\tau|$ times
\begin{equation}\label{angularKbound}
 \exp\left(\frac\pi2|\tau|-\tau\varphi
                 +\max\{0,-\tau(\varphi+\theta_x)\}\right).
\end{equation}
Indeed, on $u=e^{i\varphi}t$, the argument of
$1+u/(2x)$ lies between zero and $\varphi+\theta_x$. If this latter
angle has distance at least $\eta>0$ from $\pm\pi$, then
\[
 |1+t e^{i\psi}|\geq\sin(\eta/2)(1+t),\qquad
 t\geq0,\quad |\psi|\leq\pi-\eta.
\]
Thus its positive or negative real powers have fixed polynomial
bounds, and $e^{-t\cos\varphi}$ is integrable. The remaining factor
in \eqref{rotatedK} is $O_R(|x|^{-1/2}|e^{-x}|)$.

Put $\theta=\arg z$. Stirling's formula gives
$C_\nu(R+i\tau)=\text{polynomial}\cdot
 e^{-(\nu+2)\pi|\tau|/2}$. Combining all phases with
\eqref{angularKbound}, the coefficients of $|\tau|$ for
\eqref{boundS} are
\begin{align}
 E_{0,+}&=-\frac{\nu-1}{2}\pi-\nu\theta-\varphi
                  +\max\{0,-\varphi-\theta\},\label{familygapplus}\\
 E_{0,-}&=-\frac{\nu+3}{2}\pi+\nu\theta+\varphi
                  +\max\{0,\varphi+\theta\}.\label{familygapminus}
\end{align}
For \eqref{boundSU}, the corresponding coefficients are
\[
 E_{1,+}=-\frac{\nu+1}{2}\pi-\nu\theta-\varphi
                +\max\{0,-\varphi-\theta-\pi\},\qquad
 E_{1,-}=-\frac{\nu+1}{2}\pi+\nu\theta+\varphi
                +\max\{0,\varphi+\theta+\pi\}.
\]
For $\tau>0$ choose $\varphi\to\pi/2$ on $\theta<\pi/2$ in
\eqref{boundS}, and $\varphi=(\pi-\theta)/2$ otherwise. The respective
limiting coefficients are
$-\nu\pi/2-\nu\theta$ and
$-\nu\pi/2-(\nu-1/2)\theta$, both strictly negative on the stated
interval. For $\tau<0$ choose $\varphi\to-\pi/2$. Its coefficients
on the two sides of $\theta=\pi/2$ are
\[
 -\frac{\nu+4}{2}\pi+\nu\theta,
 \qquad -\frac{\nu+5}{2}\pi+(\nu+1)\theta.
\]
The second is negative precisely below the claimed sufficient upper
endpoint $\pi/2+2\pi/(\nu+1)$.

For \eqref{boundSU}, use $\theta_x=\theta+\pi$. For positive $\tau$
choose $\varphi\to\pi/2$ when $\theta\leq-\pi/2$ and
$\varphi=-\theta/2$ otherwise. The coefficients are
$-(\nu+2)\pi/2-\nu\theta$ and
$-(\nu+1)\pi/2-(\nu-1/2)\theta$. For negative $\tau$, the choice
$\varphi\to-\pi/2$ gives respectively
$-(\nu+2)\pi/2+\nu\theta$ and
$-(\nu+1)\pi/2+(\nu+1)\theta$. All are strictly negative for
$-\pi/2-\pi/\nu<\theta<0$. Limiting rotations are taken from within
the allowed region. On a closed subsector they can be replaced by
actual rotations with one fixed negative margin, by a finite covering
of the angular interval.

For \eqref{boundI}, use the Poisson formula
\begin{equation}\label{PoissonI}
 (r/z^2)^{-s/2}I_s(x)=\frac1{\sqrt\pi\Gamma(s+1/2)}
  \int_{-1}^1e^{xu}(1-u^2)^{s-1/2}\,du.
\end{equation}
On closed subsectors of the right half-plane, the real integral is
$O_R(e^{\operatorname{Re}x}|x|^{-R-1/2})$ in absolute value.
The ratio $C_\nu(s)/\Gamma(s+1/2)$ and the factor
$z^{(\nu-1)s}$ give the two margins
$-(\nu+1)\pi/2\pm(\nu-1)\theta<0$ there.
The same $K$ calculation for $K_{0,\nu}$ gives the margins
$-\nu\pi/2\pm\nu\theta$ on $|\theta|<\pi/2$.
These bounds prove all asserted absolute-integral estimates.

To justify the contour shifts, fix a finite strip $c\leq\operatorname{Re}s\leq R$,
where $0<c<1/2$ and neither vertical boundary contains a pole. For a
closed subsector there exist $A,N,\delta>0$, $T_0\geq1$ and $z_0>0$
such that
\begin{equation}\label{highstripbound101}
 |\text{integrand}|\leq
 A(1+|\operatorname{Im}s|)^N e^{-\delta|\operatorname{Im}s|}
 |z|^{\nu\operatorname{Re}s+1/2}|e^{\mp x}|,
 \quad |\operatorname{Im}s|\geq T_0,\quad 0<|z|\leq z_0.
\end{equation}
There is no boundedness assertion at the enclosed half-integer poles.
On the compact parts of the two selected vertical boundaries, bounded
Gamma factors and \eqref{rotatedK} or \eqref{PoissonI} give the
corresponding estimate separately. Equation~\eqref{highstripbound101}
shows that the integrals over the horizontal sides of each rectangle
vanish as its height tends to infinity. Each required asymptotic order therefore
uses only a finite contour shift and finitely many residues.
Parameter differentiation introduces polynomial factors in $s$, finitely
many shifts of Bessel order, and finite powers of $z^{-1}$. We choose $R$
after fixing the number of derivatives. The strict exponential margins
then give bounds on the differentiated remainders. Each shift uses
finitely many residues, so no infinite residue sum is interchanged
with a limit.
\end{proof}

For the fourfold the substitution $\nu=3$ recovers the sectors
$(-\pi/2,\pi)$ and $(-5\pi/6,0)$, and all the shifted-contour
estimates used in the following proof. We also use the general estimate in
Theorem~\ref{familytwocolumn}, where a separate argument establishes the
Gamma marking.

\subsection{Spinor asymptotic sections}
Define positive rational numbers
\begin{equation}\label{anMB}
 a_n=\frac{(1/2)_n^4}{4^n n!},\qquad n\ge0,
\end{equation}
and for $s_0=\pm\sqrt r$ write the formal series
\begin{equation}\label{Aspinformal}
 \mathcal A(z;q,s_0)=
 \sum_{n\ge0}\frac{a_n z^{3n}}{q^n s_0^n}
     \sum_{k=0}^n\frac{(n+k)!}{k!(n-k)!(4s_0)^k}z^k.
\end{equation}
It is a formal asymptotic series; convergence is not asserted.

\begin{theorem}[Scalar connection constants and sectors]\label{spinorscalar}
Let $q,r>0$, and set $s_+=\sqrt r$, $s_-=-\sqrt r$.
The specified Gamma sections satisfy
\begin{align}
 (1,\mathfrak s(S_0))
 &\sim -\frac{i}{2\sqrt q\,s_+}
     e^{-2s_+/z}\mathcal A(z;q,s_+),
 &&-\pi/2<\arg z<\pi,\label{spinorplusasym}\\
 (1,\mathfrak s(S_1))
 &\sim -\frac{i}{2\sqrt q\,s_-}
     e^{-2s_-/z}\mathcal A(z;q,s_-),
 &&-5\pi/6<\arg z<\pi/2.\label{spinorminusasym}
\end{align}
Each statement is a full Poincar\'e expansion on closed subsectors,
uniform on compact positive parameter sets. In particular
\begin{equation}\label{firstscalartwo}
 \mathcal A(z;q,s_0)=
 1+\frac{z^3}{64qs_0}+\frac{z^4}{128qs_0^2}+O(z^6).
\end{equation}
There are no scalar corrections in orders $z,z^2,z^5$.
These are sufficient sectors for the contour proof, not an assertion
that their boundary rays are all Stokes rays or that the sectors are maximal.
\end{theorem}
\begin{proof}
The only positive poles of $K_S$ are $s=n+1/2$. The negative of their
residues is
\begin{equation}\label{spinorresidues}
 -\operatorname{Res}_{s=n+1/2}K_S(s)
 =-2i\pi^{3/2}a_n.
\end{equation}
The minus sign occurs because the initial upward contour is moved to
its right. There are no positive poles for $K_0$.
Substitute $q/z^2,r/z^2$ in \eqref{MBspinor}, keeping the logarithms;
its integrand is $2K_S(s)(q\sqrt r)^{-s}z^{3s}K_s(x)$.
Move the contour to any sufficiently large $R$.
Lemma~\ref{Mellinsectorbound} controls the new contour relative to
$e^{-x}$ uniformly throughout the first sector.
For each residue use the exact half-integer formula
\begin{equation}\label{halfK}
 K_{n+1/2}(x)=\sqrt{\frac\pi{2x}}e^{-x}
 \sum_{k=0}^n\frac{(n+k)!}{k!(n-k)!(2x)^k}.
\end{equation}
Equations \eqref{scalarmasterGamma}, \eqref{spinorresidues} and
\eqref{halfK} give \eqref{spinorplusasym}, including its constant.
Taking $R$ arbitrarily large gives the expansion to every order.

For the second object, on $-5\pi/6<\theta<0$, the continued integral
is exactly
\[
 \frac1{2\pi i}\int 2C(s)(q\sqrt r)^{-s}z^{3s}
                         K_s(xe^{-\pi i})\,ds.
\]
Now $-\operatorname{Res}C$ equals
$2\pi^{3/2}(-1)^n a_n$. The square root in \eqref{halfK}
contributes $i$ under $x\mapsto xe^{-\pi i}$, and its $k$th term
contributes $(-1)^k$. This gives the series
\eqref{spinorminusasym} with the opposite sign of $s_0$.
The second bound of Lemma~\ref{Mellinsectorbound} proves uniformity.

To cross $\theta=0$, use \eqref{spinorIidentity} and \eqref{PoissonI}.
On $|\theta|<\pi/2$, shifting the contour in $\mathcal I$ is justified
by \eqref{boundI}. Its positive half-integer residues give the $e^x$
part of the same series. The other half-integer exponential and the
terms $4Z_{\OO}-Z_{S_0}$ are bounded by a polynomial times $e^{-x}$,
and hence are exponentially smaller relative to $e^x$ on closed
subsectors. For $Z_{\OO}$ this follows by shifting its pole-free contour
to any $R>0$ and using the last assertion of the lemma.
The two sectors overlap, proving the stated union
$-5\pi/6<\theta<\pi/2$. Finally $a_1=1/64$ and the $n=1$
polynomial is $1+z/(2s_0)$, which gives \eqref{firstscalartwo}.
\end{proof}

Recall $e_s$ from \eqref{es} and put
\begin{equation}\label{signedpsispinor}
 \varphi_s=2i\sqrt q\,s\,e_s,\qquad s=\pm\sqrt r.
\end{equation}
Then $(\varphi_s,\varphi_s)=1$. We take $\sqrt q,\sqrt r>0$:
the two leading vectors are $2i\sqrt{qr}\,e_{\sqrt r}$ and
$-2i\sqrt{qr}\,e_{-\sqrt r}$. Their relative sign is fixed by the
integral constants and is retained throughout continuation.

\begin{theorem}[Two Gamma-marked asymptotic columns]\label{twospinorcolumns}
With the same sectors and uniformity as in Theorem~\ref{spinorscalar},
\begin{align}
 \mathfrak s(S_0)&=e^{-2\sqrt r/z}
       \left(\varphi_{\sqrt r}+O(z)\right),\label{Svectorplus}\\
 \mathfrak s(S_1)&=e^{2\sqrt r/z}
       \left(\varphi_{-\sqrt r}+O(z)\right).\label{Svectorminus}
\end{align}
Both estimates extend to complete formal vector expansions. At a
semisimple positive point, they characterize the spinor Gamma columns by
asymptotics on sectors of opening greater than $\pi$. The first may be
centered at a nearby admissible phase in $(0,\pi/2)$, the second at one
in $(-\pi/3,0)$. Section~\ref{sec:transportStokes} constructs a common
asymptotic foundation of eight objects.
\end{theorem}
\begin{proof}
To recover the eight components, for
$h=(1,\mathfrak s(E))$ put $\mathscr D_q=-zq\partial_q$ and
$\mathscr D_r=-zr\partial_r$. In the cup basis, the eight pairings
$y_j=(b_j,\mathfrak s(E))$ are
\begin{align*}
 y_0&=h,&y_1&=\mathscr D_qh,\\
 y_2&=\mathscr D_q^2h,&
 y_3&=(\mathscr D_q^3+2q\mathscr D_q-2q\mathscr D_r)h,\\
 y_4&=\mathscr D_rh,&y_5&=\mathscr D_q\mathscr D_rh,\\
 y_6&=\mathscr D_q^2\mathscr D_rh,&
 y_7&=(\mathscr D_q^3\mathscr D_r-2qr)h.
\end{align*}
These are the same unipotent jet correction as in
Section~\ref{sec:connection}, now with the sign required by the flat
rather than the dual equation. Invert the constant Poincar\'e matrix to
recover the vector. Differentiable uniform estimates follow from
Lemma~\ref{Mellinsectorbound} and the contour proof.

For $h\sim a_s e^{-2s/z}$, where $a_s=-i/(2\sqrt q\,s)$, the leading
pairing vector is
\[
 a_s(1,0,0,-2qs,s,0,0,-2qr).
\]
Direct inversion gives exactly \eqref{signedpsispinor}. The higher
coefficients are obtained by the same derivatives. One may, for example, take narrower sectors centered at $\pi/4$ and
$-\pi/6$, respectively; their openings can still exceed $\pi$.
At a semisimple point, avoid the finite set of inadmissible phases by
a small perturbation. Uniqueness of an asymptotic column on such a sector follows from
\cite[Definition 2.5 and Lemma 2.6]{HK}. At the nonsemisimple point,
only the two vector estimates are used; they remain valid because their
eigenvalues stay simple.
\end{proof}

\begin{remark}[Sectorial uniqueness]
The passage from a scalar asymptotic to a canonical Gamma column uses
both the corrected jets and a sector of opening greater than $\pi$.
A formal recursion alone does not determine a sectorial column:
on a narrower sector one may add an exponentially smaller flat section.
The contour estimate \eqref{angularKbound} supplies the angular control
needed for uniqueness.
\end{remark}

\subsection{Two spinor columns over every odd-dimensional quadric}
\label{sec:familytransfer}
Let $\nu=2m+1\geq3$ and write
\[
 X_\nu=\PP_{Q^\nu}(\OO\oplus\OO(1)),\qquad
 S_{\nu,0}=\pi^*S_\nu,\qquad S_{\nu,1}=S_{\nu,0}\otimes U.
\]
Here $S_\nu$ is the normalized spinor bundle on $Q^\nu$, of rank
$2^m$. We use the same notation $H,\xi,L,U$ and the same quotient convention
on $X_\nu$ as on $X=X_3$. The framing is
\[
 \mathfrak s_\nu(E)=(2\pi)^{-(m+1)}\mathcal L_\nu
 z^{-\mu_\nu}z^{\rho_\nu}\widehat\Gamma_{X_\nu}\Ch(E),\qquad
 \rho_\nu=((\nu-1)H+2\xi)\cup,
\]
where $\mu_\nu|_{H^{2p}}=p-m-1$. All constants below use this framing.

\begin{lemma}[Cyclic normalization in odd dimension]\label{familycyclic}
The classical ring is
$\C[H,\xi]/(H^{\nu+1},\xi(\xi-H))$, with
$\int_{X_\nu}H^\nu\xi=2$, and the small quantum ring is
\begin{equation}\label{familyring}
 QH^*(X_\nu)=\C[q,r,H,\xi]/
 \bigl(H^{\nu+1}-4qH(\xi-H),\ \xi(\xi-H)-r\bigr).
\end{equation}
Its rank is $2\nu+2$. For $s^2=r$, the primitive idempotent on the
branch $H=0,\xi=s$ has the following expression in the cup frame:
\begin{equation}\label{familyidempotent}
 e_{\nu,s}=\frac14+\frac{\xi-H}{4s}
                -\frac{H^\nu}{8qs}-\frac{H^\nu\xi}{8qr},
 \qquad (e_{\nu,s},e_{\nu,s})=-\frac1{4qr}.
\end{equation}
For $q,r>0$ its anticanonical eigenvalue $2s$ is simple.
\end{lemma}
\begin{proof}
Quantum Lefschetz in
$\PP_{\PP^{\nu+1}}(\OO\oplus\OO(1))$, or Brown's fibration theorem
\cite{Brown,CG,Givental}, gives
\begin{equation}\label{familyJ}
 J_\nu=q^{H/z}r^{\xi/z}\sum_{d,e\geq0}q^d r^e
 \frac{\prod_{j=1}^{2d}(2H+jz)}{\prod_{j=1}^{d}(H+jz)^{\nu+2}}
 R_e(\xi,z)R_{e-d}(\xi-H,z).
\end{equation}
Odd quadrics have only ambient cohomology. The anticanonical degree of
$d\beta+ef$ is $(\nu-1)d+2e$, so the normalized series has no
mirror-coordinate correction. The two annihilators are
\begin{equation}\label{familyPDE}
 D_q^{\nu+1}-2q(2D_q+z)(D_r-D_q),\qquad
 D_r(D_r-D_q)-r.
\end{equation}
Their principal symbols give \eqref{familyring}; the two monic leading
monomials $H^{\nu+1},\xi^2$ are coprime. The quantum monomials specialize
to a classical basis by a triangular transformation, proving the rank
and the presentation.

To reconstruct vectors, we use the following normalization.
Use the $2\nu+2$ jets in two blocks,
$F=(f,D_qf,\ldots,D_q^\nu f,D_rf,D_qD_rf,\ldots,D_q^\nu D_rf)^{\mathsf T}$.
All their derivatives reduce by \eqref{familyPDE} and
\[
 D_q^{\nu+1}D_rf=4qrD_qf+2qrzf.
\]
Thus the scalar solution rank is at most $2\nu+2$; the independent
classical logarithmic coefficients of $J_\nu$ give equality. In the
ordered jet vector, indexed from zero, the triangular correction is
\[
 T_\nu=I-2qE_{\nu,1}+2qE_{\nu,\nu+1}
                  +2qrE_{2\nu+1,0},\qquad (T_\nu-I)^2=0.
\]
Direct reduction of the two jet equations shows that $T_\nu^{-1}$
removes their $z$-linear terms. To fix its geometric frame, the only
extra constant terms of the uncorrected $J_\nu$ jets are
$2q(\xi-H)$ in position $\nu$, and $2qr$ in position $2\nu+1$.
They come from $(d,e)=(1,0),(1,1)$, respectively. For $e<d$ the
coefficient has order at most
$z^{-(\nu-1)d-2e-1}$; for $e\geq d$ it has order at most
$z^{-(\nu-1)d-2e}$. The fiber jet kills the $e=0$ contributions because
$\xi(\xi-H)=0$. These estimates exclude every other constant or positive
power. Hence the corrected leading classes are precisely
$(H^j,H^j\xi)_{0\leq j\leq\nu}$.

In quantum-polynomial notation the projector onto $H=0$ is
$1-H^{\star\nu}\star\xi/(4qr)$. The normalization just proved gives
$H^{\star\nu}=H^\nu+2q(\xi-H)$ and
$H^{\star\nu}\star\xi=H^\nu\xi+2qr$ in cup classes. Multiplication
by $\xi$ gives the other generator; dividing their sum by two yields
\eqref{familyidempotent}. Its norm follows by cup integration.
To check simplicity of the exponent, another branch would satisfy
$h^\nu=4qv$, $v(v+h)=r$ and $(\nu+1)h+2v=2s$.
For $h\ne0$ these imply
\[
 h=\frac{4\nu s}{\nu^2-1},\qquad
 v=-\frac{\nu+1}{\nu-1}s.
\]
Since $\nu$ is odd, $h^\nu/(4v)<0$ for real nonzero $s$. This cannot
equal $q>0$, whereas the two $H=0$ points themselves are reduced.
\end{proof}

\begin{proposition}[Gamma marking of the family transform]\label{familymarking}
Define
\[
 Z_E^{(\nu)}=\int_{X_\nu}\widehat\Gamma_{X_\nu}\Ch(E)J_\nu(q,r,-1).
\]
For positive parameters, real logarithms and $0<c<1/2$,
\begin{align}
 Z_{\OO}^{(\nu)}&=\frac1{2\pi i}\int_{c-i\infty}^{c+i\infty}
 K_{0,\nu}(s)q^{-s}\,2r^{-s/2}K_s(2\sqrt r)\,ds,\label{familyMB0}\\
 Z_{S_{\nu,0}}^{(\nu)}&=\frac1{2\pi i}\int_{c-i\infty}^{c+i\infty}
 K_{S,\nu}(s)q^{-s}\,2r^{-s/2}K_s(2\sqrt r)\,ds.\label{familyMBS}
\end{align}
Equivalently, define the unnormalized base Gamma central charge by
\[
 Z^{Q^\nu}_{S_\nu}(Q)=\int_{Q^\nu}\widehat\Gamma_{Q^\nu}
                         \Ch(S_\nu)J_{Q^\nu}(Q,-1).
\]
Then the exact transfer is
\begin{equation}\label{familyconvolution}
 Z^{(\nu)}_{S_{\nu,0}}(q,r)=
 \int_0^\infty Z^{Q^\nu}_{S_\nu}(qv)e^{-v-r/v}\,\frac{dv}{v}.
\end{equation}
The $U$-twist is the clockwise $r$ continuation of \eqref{familyMBS}.
These identities retain the Gamma marking of the specified objects.
The unit component is
\begin{equation}\label{familymaster}
 (1,\mathfrak s_\nu(E)(q,r,z))=
 (2\pi)^{-(m+1)}z^{-(m+1)}
 Z_E^{(\nu)}(q/z^{\nu-1},r/z^2),
\end{equation}
with transported logarithms.
\end{proposition}
\begin{proof}
Hu-Ke's spinor character formula \cite[Theorem 4.1]{HK} is
$\ch(S_\nu)=2^{m+1}/(1+e^H)$. Reflection and duplication therefore give
\begin{equation}\label{familymultiplier}
 K_{S,\nu}(s)=\frac{2^{m+1}}{1+e^{2\pi i s}}K_{0,\nu}(s).
\end{equation}
The double Mellin kernel is
$F_{0,\nu}(s,t)=\Gamma(s)^{\nu+2}\Gamma(t)\Gamma(t-s)/\Gamma(2s)$,
multiplied by the periodic factor in \eqref{familymultiplier} for the
spinor. Its shift ratios are
\[
 \frac{F_{0,\nu}(s+1,t)}{F_{0,\nu}(s,t)}
 =\frac{s^{\nu+1}}{2(2s+1)(t-s-1)},\qquad
 \frac{F_{0,\nu}(s,t+1)}{F_{0,\nu}(s,t)}=t(t-s).
\]
For $\operatorname{Re}t>\operatorname{Re}s>0$, Stirling's formula gives
absolute convergence and justifies Fubini's theorem. Shifting by one gives \eqref{familyPDE}; the
factors $s^{\nu+1}$ and $t(t-s)$ cancel the crossed poles at zero
and the two fiber families, respectively. No positive half-integer is
crossed in this step.

The initial logarithmic polynomial is the residue at $s=0$ of
\[
 K_{a,\nu}(s)q^{-s}\{\Gamma(-s)+\Gamma(s)r^{-s}\},
 \qquad a=0,S.
\]
Use $K_{0,\nu}(s)=2\widehat\Gamma_{Q^\nu}(s)/s^{\nu+1}$,
$\pi_*f(H,\xi)=(f(H,H)-f(H,0))/H$, and
\[
 \widehat\Gamma_{X_\nu}=
 \frac{\Gamma(1+H)^{\nu+2}}{\Gamma(1+2H)}
 \Gamma(1+\xi)\Gamma(1+\xi-H).
\]
The displayed residue is exactly
$\int_{X_\nu}\widehat\Gamma_{X_\nu}\Ch(E)q^{-H}r^{-\xi}$ for
$E=\OO,S_{\nu,0}$. The classical residues at the crossing of the two
coordinate axes are nilpotent. The regular-singular recursions have
invertible operators $dI+$ a nilpotent matrix for positive degrees.
Thus these $2\nu+2$ initial logarithmic coefficients fix the section
uniquely, as in the fourfold argument.
The small series converge with all derivatives: for $e=d+k$ a scalar
majorant is $4^d/((d!)^{\nu+1}(k!)^2)$, and for $e<d$ it is
$4^d/((d!)^{\nu-1}(e!)^2)$, up to fixed logarithmic factors. Both
are summable on compact parameter sets. This also justifies the
regular-singular expansions used to compare initial polynomials.
The base Mellin integral $\int K_{S,\nu}(s)Q^{-s}ds/(2\pi i)$
has the stated base Gamma marking by the same zero-residue computation,
now using $\int_{Q^\nu}H^\nu=2$. Its integral against
$e^{-v-r/v}dv/v$ converges absolutely on $\operatorname{Re}s=c$,
with real-variable majorant $v^{-c}e^{-v-r/v}$. Fubini's theorem gives
\eqref{familyconvolution}, expressing the transfer from the base to the bundle.

The Bessel continuation formula \eqref{Besselturn} now gives
\[
 Z_{S_{\nu,1}}^{(\nu)}=2^{m+1}Z_{\OO}^{(\nu)}
       -Z_{S_{\nu,0}}^{(\nu)}+2\pi i\mathcal I_\nu,
\]
where $\mathcal I_\nu$ is \eqref{Itransform} with $C$ replaced by
$C_\nu$. Its initial polynomial is multiplied by $e^{2\pi i\xi}$,
so the object is exactly $S_{\nu,0}U$.
Finally the usual symplectic identity for $\mathcal L_\nu$ and the
weights $\deg q=\nu-1$, $\deg r=2$ give \eqref{familymaster}.
This is an identification with the quantum Gamma framing, not with
an integral geometric cycle.
\end{proof}

\begin{theorem}[Two-column spinor transfer]\label{familytwocolumn}
For every fixed odd $\nu=2m+1\geq3$ and $q,r>0$, set
$\varphi_{\nu,s}=2i\sqrt q\,s\,e_{\nu,s}$ for $s=\pm\sqrt r$.
Then $(\varphi_{\nu,s},\varphi_{\nu,s})=1$, and
\begin{align}
 \mathfrak s_\nu(S_{\nu,0})
 &=e^{-2\sqrt r/z}\bigl(\varphi_{\nu,\sqrt r}+O(z)\bigr),
 &-\pi/2<\arg z<\pi/2+\frac{2\pi}{\nu+1},\label{familyplus}\\
 \mathfrak s_\nu(S_{\nu,1})
 &=e^{2\sqrt r/z}\bigl(\varphi_{\nu,-\sqrt r}+O(z)\bigr),
 &-\pi/2-\frac\pi\nu<\arg z<\pi/2.\label{familyminus}
\end{align}
Both formulas extend to full differentiable expansions on closed
subsectors, locally uniformly in the positive parameters. At a point
with distinct spectrum, they identify two canonical Gamma columns.
For $X_\nu$ with $\nu>3$, the result concerns these two columns and
does not assert Gamma II for a complete foundation.

More precisely, put $s_\pm=\pm\sqrt r$ and
\begin{equation}\label{familyseries}
 \mathcal A_\nu(z;q,s_\pm)=\sum_{k\geq0}
 \frac{(1/2)_k^{\nu+1}}{4^k k!}\frac{z^{\nu k}}{q^ks_\pm^k}
 \sum_{j=0}^k\frac{(k+j)!}{j!(k-j)!(4s_\pm)^j}z^j.
\end{equation}
The corresponding scalar expansions on the same sectors are
\begin{equation}\label{familyscalar}
 (1,\mathfrak s_\nu(S_{\nu,0/1}))
 \sim-\frac{i}{2\sqrt q\,s_\pm}e^{-2s_\pm/z}
               \mathcal A_\nu(z;q,s_\pm),
\end{equation}
where $0/1$ is paired with $+/-$, respectively. The connection constant
is independent of $\nu$ for the specified framing in each dimension.
\end{theorem}
\begin{proof}
The negative residue of $K_{S,\nu}$ at $s=k+1/2$ is
\[
 -i\,2^m\pi^{m+1/2}\frac{(1/2)_k^{\nu+1}}{4^k k!}.
\]
Insert the scaling in \eqref{familymaster} into \eqref{familyMBS}.
Its integrand is $2K_{S,\nu}(s)(q\sqrt r)^{-s}z^{\nu s}K_s(x)$.
Shift to any fixed sufficiently large $R$ and use
Lemma~\ref{Mellinsectorbound}. The exact half-integer Bessel formula
\eqref{halfK} yields \eqref{familyseries}. At the first pole the factor
$z^{(\nu+1)/2}$ cancels the scalar framing power $z^{-(m+1)}$;
$2^m\pi^{m+1/2}\sqrt\pi/(2\pi)^{m+1}=1/2$. This gives the precise
constant in \eqref{familyscalar}.

For the second flat section, use $K_s(xe^{-\pi i})$ on
$-\pi/2-\pi/\nu<\arg z<0$. The residues of $C_\nu$, the continued
square root and the half-integer Bessel polynomial give the same
series with $s_-=-\sqrt r$. On $|\arg z|<\pi/2$, use instead
$2^{m+1}Z_{\OO}^{(\nu)}-Z_{S_{\nu,0}}^{(\nu)}+2\pi i\mathcal I_\nu$.
The Poisson representation gives its $e^x$ residue series. The remaining
terms are bounded by a polynomial times $e^{-x}$ and are therefore
exponentially smaller there. The pole-free $K_{0,\nu}$ bound proves this assertion for
$Z_{\OO}^{(\nu)}$ without assuming a separate Gamma I theorem.
The exact continuation identity identifies the two representations on their overlap.
For each required order, choose $R$ large enough to obtain the
corresponding differentiable remainder bound. This gives a full
asymptotic expansion without assuming that the formal series converges.

For the full vector put $h=(1,\mathfrak s_\nu(E))$ and
$\mathscr D_q=-zq\partial_q$, $\mathscr D_r=-zr\partial_r$.
The normalized jets in Lemma~\ref{familycyclic} give the pairings
\[
 y_{j,0}=\mathscr D_q^jh,\quad
 y_{j,1}=\mathscr D_q^j\mathscr D_rh\quad(0\leq j<\nu),
\]
\[
 y_{\nu,0}=(\mathscr D_q^\nu+2q\mathscr D_q-2q\mathscr D_r)h,
 \qquad
 y_{\nu,1}=(\mathscr D_q^\nu\mathscr D_r-2qr)h.
\]
For $h\sim a_s e^{-2s/z}$, $a_s=-i/(2\sqrt q\,s)$, their leading
values are $a_s$, $-2qs\,a_s$, $s\,a_s$, $-2qr\,a_s$ in the
positions $(0,0),(\nu,0),(0,1),(\nu,1)$, and zero elsewhere.
Cup-pairing inversion gives exactly $2i\sqrt q\,s\,e_{\nu,s}$.
The differentiable remainder bounds justify the same reconstruction at
every order. The sector openings are
$\pi+2\pi/(\nu+1)$ and $\pi+\pi/\nu$, strictly greater than $\pi$.
The canonical columns are unique at points with distinct spectrum
\cite[Definition 2.5 and Lemma 2.6]{HK}.
\end{proof}

For example, on the sixfold $X_5$ the two sufficient sectors are
$(-\pi/2,5\pi/6)$ and $(-7\pi/10,\pi/2)$, and
\begin{equation}\label{sixfoldexample}
 \mathcal A_5(z;q,s)=1+\frac{z^5}{256qs}
                         +\frac{z^6}{512qs^2}+O(z^{10}).
\end{equation}
This two-column result does not determine a Stokes basis of twelve
columns. Proposition~\ref{familylattice} below computes the size of the
missing line lattice.

\begin{corollary}[Dimension-dependent half-turn]\label{familyhalfturn}
Let $\mathsf D_\nu|_{H^{2p}}=(-1)^p$. On transported logarithms,
\begin{equation}\label{familyhalfturneq}
 \mathfrak s_\nu(E\otimes L^mU)(z)
 =(-1)^{m+1}\mathsf D_\nu
          \mathfrak s_\nu(E)(e^{\pi i}z).
\end{equation}
Two applications give the anticanonical twist
$L^{2m}U^2=\omega_{X_\nu}^{-1}$.
\end{corollary}
\begin{proof}
In \eqref{familymaster} a half-turn decreases the scaled logarithms of
$q$ and $r$ by $2m\pi i$ and $2\pi i$ and multiplies $z^{-(m+1)}$
by $(-1)^{m+1}$. The marking argument identifies these turns with
$L^m$ and $U$. Every corrected pairing of complex degree $p$ changes
by $(-1)^p$ when $z$ is replaced by $-z$, since both Novikov weights
are even. Pairing inversion gives \eqref{familyhalfturneq}.
The squared sign is one and $\mathsf D_\nu^2=I$.
\end{proof}

Since $c_1(X_\nu)=2mH+2\xi$, the same degree-parity argument gives
cohomological chiral symmetry in this family. The scalar factor in
\eqref{familyhalfturneq} records the dimension:
$e^{-\pi i\mu_\nu}=(-1)^{m+1}\mathsf D_\nu$.
For $\nu=3$ this is exactly the half-turn used in Section~\ref{sec:transportStokes}.
For larger $\nu$, the line twist and the available phase intervals
differ. The full calculation of the walls for eight columns in the next
section therefore applies only to the fourfold.

\subsection{The section of structure sheaf}\label{sec:markedjump}
On $(\mathbb R_{>0})^4$, with coordinates $(x,y,k,v)$, set
\begin{equation}\label{positivepotential}
 \mathscr W=x+y+\frac{qv(k+1)^2}{xyk}+v+\frac rv.
\end{equation}
Let $d\ell=d\log x\,d\log y\,d\log k\,d\log v$ on the positive real
locus, oriented in this order.

\begin{proposition}[Exact positive integral]\label{positiveintegral}
For $q,r>0$, with real logarithms,
\begin{equation}\label{positivecentral}
 Z_{\OO}(q,r)=\int_{(\mathbb R_{>0})^4}e^{-\mathscr W}\,d\ell.
\end{equation}
Consequently, for $z>0$,
\begin{equation}\label{positiveflatmaster}
 (1,\mathfrak s(\OO)(q,r,z))
 =\frac1{(2\pi)^2z^2}
   \int_{(\mathbb R_{>0})^4}e^{-\mathscr W/z}\,d\ell.
\end{equation}
These identities determine the Gamma marking. They do not identify the
real integration locus as a primitive basis element of a global
rapid decay homology lattice.
\end{proposition}
\begin{proof}
Apply Mellin inversion to the exponential of
$qv(k+1)^2/(xyk)$, on a line $\operatorname{Re}s=c>0$. The $x$ and $y$
integrals each give $\Gamma(s)$, while
\[
 \int_0^\infty\left(\frac{k}{(k+1)^2}\right)^s\frac{dk}{k}
 =\frac{\Gamma(s)^2}{\Gamma(2s)}.
\]
The remaining $v$ integral is
$\int_0^\infty v^{-s}e^{-v-r/v}\,dv/v=\mathcal K_s(r)$.
Together with the Mellin factor $\Gamma(s)$, the resulting integrand is
exactly $K_0(s)q^{-s}\mathcal K_s(r)$ from \eqref{MBstructure}.
Absolute convergence justifies the changes in the order of integration:
before integration over $s$, the absolute values of the real integrals
are bounded by the same expressions at $s=c$, and
$|\Gamma(c+i\tau)|$ is integrable in $\tau$. Thus \eqref{positivecentral} follows from the previously proved
Gamma-marked Barnes identity. To obtain \eqref{positiveflatmaster},
replace $(q,r)$ by $(q/z^2,r/z^2)$ and make the substitutions
$(x,y,v)=(X/z,Y/z,V/z)$, keeping $k$ fixed. The logarithmic measure is
unchanged and the exponent becomes $-\mathscr W(X,Y,k,V)/z$.
Equation \eqref{scalarmasterGamma} supplies the prefactor.
\end{proof}

For $q,r>0$ let $h,v>0$ be the unique solution of
\begin{equation}\label{positivesaddlehv}
 h^3=4qv,\qquad v(v+h)=r,
 \qquad T=4h+2v.
\end{equation}
Thus $h$ is the positive root on the sextic branch and $T=T(q,r)$.
In the formulas below, $v$ refers to this critical coordinate.

\begin{theorem}[Positive saddle with fixed Gamma normalization]
\label{positiveGammaall}
On compact subsets of $\mathbb R_{>0}^2$, the scalar unit component has a
complete differentiable expansion as $z\to0^+$, beginning
\begin{equation}\label{Oscalarfirst}
 (1,\mathfrak s(\OO))=
 \frac{e^{-T/z}}{\sqrt{h^3(2h+3v)}}
 \left(1+\mathfrak c_1(h,v)z+O(z^2)\right),
\end{equation}
where
\begin{equation}\label{Ocorrection}
 \mathfrak c_1(h,v)=
 -\frac{5(2h^3+9h^2v+12hv^2+6v^3)}{4h(2h+3v)^3}.
\end{equation}
If $e_T$ is the primitive quantum idempotent of the simple eigenvalue
$T$ and $w_T=(e_T,e_T)>0$, the full vector expansion begins
\begin{equation}\label{Ovectorpositive}
 \mathfrak s(\OO)=e^{-T/z}\left(\psi_T+O(z)\right),
 \qquad \psi_T=\frac{e_T}{\sqrt{w_T}},\qquad
 w_T=\frac1{h^3(2h+3v)}.
\end{equation}
In particular the Gamma I property holds at every positive parameter
point, by a direct argument independent of global propagation. At the
unit point this is the original Gamma conjecture I.
\end{theorem}
\begin{proof}
In logarithmic coordinates the potential \eqref{positivepotential} is a
sum of positive exponentials whose exponent vectors span $\mathbb R^4$.
Its Hessian is therefore positive definite everywhere. Its sublevel
sets are compact: $v+r/v$ bounds $v$ away from zero and infinity, $x+y$
bounds $x,y$ above, and the mixed term then bounds $xy$ away from zero
and bounds $k+k^{-1}$. Thus $x,y,k,v$ are all bounded away from zero and
infinity on a sublevel set. Strict convexity and properness imply a
unique minimum. Its critical equations give
$(x,y,k,v)=(h,h,1,v)$ with \eqref{positivesaddlehv}.

The logarithmic Hessian at that point is
\begin{equation}\label{positiveHessian}
 \mathcal H=\begin{pmatrix}
 2h&h&0&-h\\
 h&2h&0&-h\\
 0&0&h/2&0\\
 -h&-h&0&2h+2v
 \end{pmatrix},\qquad
 \det\mathcal H=h^3(2h+3v).
\end{equation}
The four-dimensional Laplace expansion, with constant amplitude in
logarithmic coordinates, has leading factor
$(2\pi z)^2/\sqrt{\det\mathcal H}$. The prefactor in
\eqref{positiveflatmaster} cancels $(2\pi z)^2$. The minimum, Hessian,
and all local derivatives vary smoothly on a positive compact parameter
set. Compactness gives a uniform gap off a small neighborhood of the
minimum. The unbounded tail is controlled by integrability of
$e^{-\mathscr W}$: if $\mathscr W\ge T+\delta$ there, then for small
positive $z$ its integral is bounded by a constant times
$e^{-(T+\delta)/z}$. The same estimates, with polynomial factors in the positive exponential
terms, justify parameter differentiation and the Laplace expansion to
every order. This proves the asymptotic expansion with the stated
uniformity.

The first correction follows by Gaussian contraction. The seven
exponent vectors and weights at the minimum are
\[
 \begin{gathered}
 m_1=(1,0,0,0),\ m_2=(0,1,0,0),\ m_3=(-1,-1,1,1),\\
 m_4=(-1,-1,0,1),\ m_5=(-1,-1,-1,1),\
 m_6=(0,0,0,1),\ m_7=(0,0,0,-1),\\
 (a_1,\ldots,a_7)=(h,h,h/4,h/2,h/4,v,v+h).
 \end{gathered}
\]
Write $B_{ij}=m_i^{\mathsf T}\mathcal H^{-1}m_j$. Gaussian contraction
of the quartic term and the square of the cubic term gives
\begin{equation}\label{Wickfirst}
 -\frac18\sum_i a_iB_{ii}^2
 +\frac1{12}\sum_{i,j}a_ia_jB_{ij}^3
 +\frac18\sum_{i,j}a_ia_jB_{ii}B_{jj}B_{ij}.
\end{equation}
Substituting \eqref{positiveHessian} yields \eqref{Ocorrection}.

We recover the full vector from its scalar component. Implicit differentiation in \eqref{positivesaddlehv} gives
\begin{equation}\label{hvderivatives}
 \begin{pmatrix}q\partial_qh&r\partial_rh\\
 q\partial_qv&r\partial_rv\end{pmatrix}
 =\frac1{2(2h+3v)}
 \begin{pmatrix}h(h+2v)&h(h+v)\\-hv&3v(h+v)\end{pmatrix}.
\end{equation}
In particular $q\partial_qT=h$ and $r\partial_rT=h+v$.
The corrected jets in the proof of Theorem~\ref{twospinorcolumns}
therefore give, at leading order, the evaluation of every cup-basis
class at this quantum point, multiplied by $\sqrt{w_T}$.
For a primitive idempotent the evaluation functional is
$b\mapsto(b,e_T)/w_T$. Inverting the Poincar\'e pairing consequently
recovers $e_T/\sqrt{w_T}$. Lemma~\ref{idempotentnorms} gives
\[
 w_T=\frac{4q}{h^4(3h^2+8q)}=\frac1{h^3(2h+3v)},
\]
which proves \eqref{Ovectorpositive} and its normalization.
Finally Theorem~\ref{spectral} shows that $T$ is simple and strictly
rightmost for every positive $(q,r)$, including the caustic. The vector
estimate identifies the structure-sheaf Gamma section with the
fastest-decaying line, which is the Gamma I property. Property
$\mathcal O$ holds at the unit point by the same spectral theorem.
\end{proof}

\begin{remark}[Uniqueness along the positive ray]
\label{fastestrayremark}
At a positive point with distinct spectrum, choose canonical asymptotic
fundamental solutions for small admissible phases $\epsilon$ and
$-\epsilon$. Their sectors contain the positive ray. Expanding
$\mathfrak s(\OO)$ in either basis, every column other than the $T$ column
would grow relative to $e^{-T/z}$ on that ray, since
$\operatorname{Re}u<T$. Thus \eqref{Ovectorpositive} forces all those
coefficients to vanish and fixes the remaining coefficient to one.
Uniqueness of the fastest-decaying line gives this identification;
convergence of the original real contour on a sector wider than
$\pi$ is not required. The spinor exponent $2\sqrt r$ is not rightmost,
so its canonical identification requires the wider sectors proved in
Section~\ref{sec:spinorMB}.
\end{remark}

\subsection{Reflection and parameter monodromy}
In this subsection, we take positive $q,r$ and a logarithmic branch
symmetric under complex conjugation near the positive ray. We take all
conjugations in the fixed real cup frame.

\begin{lemma}[Reflection of the Gamma framing]\label{Gammareflection}
For an algebraic vector bundle $E$ on $X$,
\begin{equation}\label{reflectionGamma}
 \mathfrak s(E^\vee)(q,r,z)
   =\overline{\mathfrak s(E)(q,r,\bar z)}.
\end{equation}
Moreover the following are identities in $K_0(X)$:
\begin{equation}\label{spinordualclasses}
 [S_0^\vee]=4[\OO]-[S_0],\qquad
 [S_1^\vee]=4[\OO]-[S_1].
\end{equation}
\end{lemma}
\begin{proof}
The quantum connection and its large-radius normalization have real
coefficients in the cup frame at positive parameters. The operators
$\mu,\rho$ and $\widehat\Gamma_X$ are real. Since
$\overline{\Ch(E)}=\Ch(E^\vee)$, reflection of the normalized formula
for $\mathfrak s(E)$ proves \eqref{reflectionGamma}; logarithms are
reflected rather than reset. Equivalently, uniqueness with the reflected Gamma data at the regular
singular end gives the same identity.

For the classes, the spinor exact sequence
$0\to S\to\OO_{Q^3}^{\oplus4}\to S(1)\to0$ and
$S^\vee\cong S(1)$ give the first identity. The second uses
$LU^{-1}=1+L-U$, a consequence of the projective-bundle relation:
\[
 [S_1^\vee]=[S_0]LU^{-1}
 =[S_0](1+L-U)=4[\OO]-[S_1].
\]
The second computation is a $K$-class identity, not a proposed short
exact sequence of vector bundles on $X$.
\end{proof}

The eigenvectors $\varphi_{\sqrt r}$ and $\varphi_{-\sqrt r}$ from
\eqref{signedpsispinor} are purely imaginary in this frame. Hence the
negative of the reflected Gamma section has the \emph{same} normalized
leading eigenvector. Introduce
\begin{align}
 Y_+^+&=\mathfrak s(S_0),&
 Y_+^-&=-\mathfrak s(S_0^\vee)
       =\mathfrak s(S_0)-4\mathfrak s(\OO),\label{pluslaterals}\\
 Y_-^-&=\mathfrak s(S_1),&
 Y_-^+&=-\mathfrak s(S_1^\vee)
       =\mathfrak s(S_1)-4\mathfrak s(\OO).\label{minuslaterals}
\end{align}
The subscript indicates the exponent $\pm2\sqrt r$; the superscript
indicates a small positive or negative admissible phase. Reflection of
Theorems~\ref{spinorscalar} and \ref{twospinorcolumns} proves the sectors
\[
 \begin{array}{c|c|c}
 \text{column}&\text{angular sector}&\text{admissible centers available}\\\hline
 Y_+^+&(-\pi/2,\pi)&(0,\pi/2)\\
 Y_+^-&(-\pi,\pi/2)&(-\pi/2,0)\\
 Y_-^-&(-5\pi/6,\pi/2)&(-\pi/3,0)\\
 Y_-^+&(-\pi/2,5\pi/6)&(0,\pi/3).
 \end{array}
\]
In each row, the indicated phase can be the center of a smaller sector
whose opening is strictly greater than $\pi$. This identification uses
the full sector estimates for the spinor sections, including their
remainders; the leading term on a single ray would not suffice.

Let $\mathcal M_q$ be continuation with $\log q\mapsto\log q-2\pi i$.
\begin{proposition}[Marked parameter monodromy]\label{MBhelix}
On these Gamma-marked master functions,
\[
 \mathcal M_q Z_E=Z_{E\otimes\OO(H)},\qquad
 \mathcal M_r Z_E=Z_{E\otimes\OO_\pi(1)}.
\]
For the Bessel kernel, the second operation satisfies
\begin{equation}\label{fiberanalyticrelation}
 (\mathcal M_r-1)(\mathcal M_r-e^{2\pi i s})\mathcal K_s=0.
\end{equation}
Consequently the marked solution space satisfies
$\mathcal M_r^2-(1+\mathcal M_q)\mathcal M_r+\mathcal M_q=0$.
In the Gamma foundation of eight objects, these are the matrices $T_H,T_U$
of \eqref{twistmatrices}; a positive $z$ turn acts by
$\mathcal M_q^2\mathcal M_r^2=\mathcal R$.
\end{proposition}
\begin{proof}
The zero-degree logarithmic polynomial or the exact Gamma framing
identifies the two tensor operations with the specified clockwise
parameter turns. At the kernel level, put
$J_s(r)=r^{-s/2}I_s(2\sqrt r)$. It is invariant under clockwise continuation,
whereas \eqref{Besselturn} is
$\mathcal M_r\mathcal K_s=e^{2\pi is}\mathcal K_s+
2\pi i e^{\pi is}J_s$. This proves \eqref{fiberanalyticrelation}.
Multiplication by $e^{2\pi is}$ is $\mathcal M_q$ inside the integral.
Thus the Bessel continuation supplies the analytic counterpart of
$U^2-(1+L)U+L=0$, with the Gamma marking already fixed.
Finally \eqref{scalarmasterGamma} makes a positive $z$ turn two clockwise
turns in each scaled parameter; $z^{-2}$ is single-valued. The helix
matrix follows from Section~\ref{sec:helix}.
\end{proof}

The eight Gamma-marked functions are linearly independent: their classes
at the regular singular end form a full exceptional $K$-basis, and the
Chern character framing is invertible. This independence statement
concerns the quantum Gamma lattice; comparison with geometric cycles
of rapid decay remains a separate question.
Section~\ref{sec:transportStokes} proves the identification at a common phase.

\section{Canonical transport and Stokes data}\label{sec:transportStokes}\label{sec:closure}\label{sec:GammaII}

We obtain the remaining columns from the three initial ones by a parameter
loop, reflection, and a half-turn. We then bring all eight columns to a
common phase and transport the resulting exceptional collection to the
unit parameter point.
\subsection{Chiral half-turn and separated spectral clusters}
Recall the cohomological chiral involution $\mathsf D$ from
\eqref{paritymatrix}. For a section on the logarithmic cover of the
$z$-plane, define
\begin{equation}\label{chiralflatoperator}
 (\mathscr Hs)(q,r,z)=\mathsf D s(q,r,e^{\pi i}z),
 \qquad b_{\rm ch}(s,t)=(\mathsf D s(q,r,z),t(q,r,z))_X.
\end{equation}
The argument $e^{\pi i}z$ includes the indicated upper half-turn.

\begin{proposition}[Equivariance and a flat symmetric form]
\label{chiralflat}
The operator $\mathscr H$ intertwines all three operators of the small
quantum connection \eqref{qconnection}. The form $b_{\rm ch}$ is
symmetric, nondegenerate, and constant on pairs of flat sections.
Its matrix in the cup frame is $P\mathsf D=\mathsf D P$.
If
\[
 s(z)\sim e^{-u/z}\sum_{k\geq0}v_kz^k
\]
on a sector of opening greater than $\pi$ centered at $\phi$, then
\begin{equation}\label{chiralallorders}
 (\mathscr Hs)(z)\sim e^{u/z}
       \sum_{k\geq0}(-1)^k\mathsf Dv_kz^k
\end{equation}
on its translate by $-\pi$. In particular, the operator preserves
canonical normalization. On the logarithmic cover,
$\mathscr H^2s(z)=s(e^{2\pi i}z)$.
\end{proposition}
\begin{proof}
Put $w=e^{\pi i}z$. The identities \eqref{chiralidentities} show that
the sign change $1/w=-1/z$ is compensated by conjugation with $\mathsf D$
in the divisor and Euler terms. The operator $z\partial_z$ is carried to
$w\partial_w$, and $\mathsf D$ commutes with $\mu$. This proves
equivariance for the full connection, including the grading term.

For the form, divisor multiplication is self-adjoint for $P$, whereas
$\mu^{\mathsf T}P+P\mu=0$. Hence
\[
 M_i^{\mathsf T}P\mathsf D+P\mathsf D M_i=0\quad(i=H,\xi),
 \qquad \mu^{\mathsf T}P\mathsf D+P\mathsf D\mu=0.
\]
These identities show that the pairing is constant in $q,r,z$. Symmetry and
nondegeneracy follow from $\mathsf D^{\mathsf T}P=P\mathsf D$ and
$\mathsf D^2=I$. Substitution gives \eqref{chiralallorders}, and
$\mathsf D$ preserves the norm of the leading idempotent.
Applying the substitution twice gives a full turn with the logarithm
transported along it.
\end{proof}

\begin{lemma}[Marked half-turn]\label{halfturn}
Here $LU=L\otimes U$ denotes the line bundle $\OO_X(H+\xi)$.
Transporting the logarithms along the half-turn gives
\begin{equation}\label{halfturnformula}
 \mathfrak s(E\otimes LU)(q,r,z)
   =\mathsf D\,\mathfrak s(E)(q,r,e^{\pi i}z).
\end{equation}
If the section on the right is canonical before rotation at exponent
$u$ and phase $\phi$, the section on the left is canonical at exponent
$-u$ and phase $\phi-\pi$, with leading vector $\mathsf D\psi_u$.
For positive parameters, reflection also sends $E$ to $E^\vee$,
$u$ to $\bar u$, and $\phi$ to $-\phi$.
\end{lemma}
\begin{proof}
The scalar master formula \eqref{scalarmasterGamma} gives
\[
 h_E(e^{\pi i}z)=h_{E\otimes LU}(z),\qquad h_E=(1,\mathfrak s(E)),
\]
because $(e^{\pi i}z)^{-2}=z^{-2}$, whereas each scaled parameter
logarithm decreases by $2\pi i$. Proposition~\ref{MBhelix} identifies
these two clockwise parameter turns with tensoring by $L$ and $U$.
The logarithms at the endpoint are determined by this continuation.

The corrected derivatives in the proof of
Theorem~\ref{twospinorcolumns} recover all cup pairings. A pairing with a cup class of complex
degree $p$ changes by $(-1)^p$ under $z\mapsto e^{\pi i}z$: each
correction has the same degree modulo two, since $q,r$ have degree two.
As the Poincar\'e pairing pairs degrees $p$ and $4-p$, inversion of that
pairing yields exactly \eqref{halfturnformula}. The three matrix identities above, together with the cyclicity of the
unit master function, give an equivalent proof. The asymptotic assertion follows by substitution; the
normalization is preserved because $\mathsf D$ is an isometry.
Reflection follows from Lemma~\ref{Gammareflection} with reflected
logarithms. The sector estimates are transported with the vector section;
a scalar limit along a ray would not suffice.
\end{proof}

Two applications of \eqref{halfturnformula} give tensoring by
$(LU)^{\otimes2}=\omega_X^{-1}$, the full helix action on $K_0(X)$.
Thus $\mathscr H$ is not an involution on the Gamma lattice even though
$\mathsf D^2=I$ on cohomology. A derived lift compatible with the helix normalization is the functor
$E\mapsto E\otimes LU[-2]$, which squares to
$\mathsf S_X^{-1}(E)=E\otimes\omega_X^{-1}[-4]$; the even shift
is invisible to the Gamma framing but is part of the derived statement.
As discussed in Section~\ref{sec:helix}, a single half-turn does not
shift the original exceptional foundation by four objects.

\begin{corollary}[Gamma pairing and opposite phases]\label{oppositephase}
For $A=LU$ and Gamma-framed sections, the symmetric pairing is
\begin{equation}\label{chiralgammapairing}
 b_{\rm ch}(\mathfrak s(E),\mathfrak s(F))
 =\chi(E\otimes A,F)=\chi(E,F\otimes A^{-1}).
\end{equation}
Suppose an ordered exceptional collection $\mathcal E$ gives a
canonical Gamma basis at an admissible phase $\phi$. With transported
logarithms and the same object order, $\mathcal E\otimes A$ gives a
canonical basis at $\phi-\pi$, and
\begin{equation}\label{oppositephasematrices}
 S_{\phi-\pi}=S_\phi,\qquad
 \mathcal C_{\phi-\pi}=e^{\pi i\rho}\mathcal C_\phi.
\end{equation}
The central matrices use the cup frame and the normalization at the
regular end fixed in Theorem~\ref{mainGammaII}.
\end{corollary}
\begin{proof}
Evaluate \eqref{halfturnformula} at $e^{-\pi i}z$. The resulting
identity $\mathsf D\mathfrak s(E)(z)
=\mathfrak s(E\otimes A)(e^{-\pi i}z)$, followed by the
Gamma Riemann--Roch pairing \cite[Section 3.4]{GGI}, gives
\eqref{chiralgammapairing}. Its symmetry also follows from Serre
duality in dimension four and $A^2=\omega_X^{-1}$.

For the canonical basis, Lemma~\ref{halfturn} sends the exponents
$u_i$ to $-u_i$. Their projected order is unchanged, since
\[
 \operatorname{Im}(e^{-i(\phi-\pi)}(-u_i))
 =\operatorname{Im}(e^{-i\phi}u_i).
\]
The Euler form is preserved by tensoring both arguments with $A$,
which proves the Stokes identity in the stated convention. Finally,
$\Ch(A)=e^{2\pi i(H+\xi)}=e^{\pi i\rho}1$ gives the central
identity directly from the Gamma columns. To return these rotated
columns to the original phase $\phi$, one must apply the appropriate
Stokes mutations in addition to \eqref{oppositephasematrices}.
\end{proof}

Fix $r=1$. Put $\zeta=e^{2\pi i/3}$ and
\[
 q=\rho^3/4,\qquad |\rho|=\epsilon,\qquad
 v_+(h)=\sqrt{1+h^2/4}-h/2,
\]
where the square root is holomorphic near zero with value one. The
three roots of
\begin{equation}\label{clusterheq}
 h^3=\rho^3v_+(h)
\end{equation}
give the cluster near the positive center through
\begin{equation}\label{clusterueq}
 u(h)=3h+2\sqrt{1+h^2/4}.
\end{equation}
The other three values are their negatives, and the two remaining values
are exactly $2,-2$.

\begin{lemma}[Uniform root discs]\label{clusterdiscs}
For $0<\epsilon\leq 10^{-2}$ and any argument of $\rho$, equation
\eqref{clusterheq} has exactly one root $h_j$ in each disc
$|h-\rho\zeta^j|<\epsilon^2$, for $j=0,1,2$. Moreover
\begin{equation}\label{clusterbound}
 |u(h_j)-(2+3\rho\zeta^j)|<4\epsilon^2.
\end{equation}
These six values and $\pm2$ are distinct. At $\rho=\epsilon>0$ write
them as
\begin{equation}\label{endpointspectrum}
 T,\ a+ib,\ a-ib,\ -T,\ -a-ib,\ -a+ib,\ 2,\ -2.
\end{equation}
Then
\begin{equation}\label{abcintervals}
 2+2\epsilon<T<2+4\epsilon,\qquad
 2-2\epsilon<a<2-\epsilon,\qquad
 2\epsilon<b<3\epsilon.
\end{equation}
\end{lemma}
\begin{proof}
For $|h|\leq 0.02$, the chosen square root has positive real part and
\[
 |v_+(h)-1|
 \leq |h|/2+|h|^2/4<|h|.
\]
On the boundary of each disc, the other two roots of
$h^3-\rho^3$ are at distance at least
$(\sqrt3-\epsilon)\epsilon$. Consequently
\[
 |h^3-\rho^3|
 \geq (\sqrt3-\epsilon)^2\epsilon^4>2\epsilon^4,
 \qquad
 |\rho^3(v_+(h)-1)|<(1+\epsilon)\epsilon^4.
\]
Rouch\'e's theorem proves the assertion for each disc. The discs are
disjoint. Since $|h|\leq\epsilon+\epsilon^2$, formula
\eqref{clusterueq} gives
\[
 |u-(2+3\rho\zeta^j)|
 \leq 3\epsilon^2+|h|^2/2<4\epsilon^2.
\]
Eliminating $v$ from $h^3=4qv$ and $v(v+h)=1$ gives the sextic already
proved for the quantum ring. The negative roots with $v=-v_+(h)$ give
$-u(h)$, so the displayed list accounts for all eight values.
Within each cluster, the distance from a vertex to the center is at least
$3\epsilon-4\epsilon^2$, and the distance between vertices is at least
$3\sqrt3\epsilon-8\epsilon^2$. Between the two clusters, the leading
separation is four, with errors of less than $7\epsilon$.
This proves distinctness. Conjugation and the positive root identify
$T,a,b$ at the positive endpoint; \eqref{abcintervals} follows from
\eqref{clusterbound}. All these bounds are uniform around the parameter circle.
\end{proof}

\subsection{The Laplace-dual connection}
\label{sec:smalldual}
We carry out the path argument directly in the explicit connection in
two parameters. This gives the required analytic continuation;
convergence of individual big Frobenius germs would not by itself
establish their gluing along a path. We use the Laplace construction of
\cite[Section 2.5]{GGI} and \cite[Section 3]{HK}, specialized to the small
quantum connection with the signs fixed below.
Cruz Morales and Milanov reformulate the Dubrovin conjecture in terms
of the second structure connection and twisted reflection
vectors~\cite{CruzMilanov}.

Write $Q_q=q\partial_q$, $Q_r=r\partial_r$, $M_q=M_H$, $M_r=M_\xi$,
and $C=2(M_H+M_\xi)$. On the locus with distinct eigenvalues, set
\[
 R_\lambda=(\lambda I-C)^{-1}.
\]
Throughout this subsection, $R_\lambda$ denotes the resolvent.
\begin{proposition}[Laplace dual of the small quantum connection]
\label{smallLaplaceDual}
The equations
\begin{equation}\label{smallLaplaceEq}
 \partial_\lambda\widehat y=-R_\lambda\mu\widehat y,
 \qquad Q_i\widehat y=M_iR_\lambda\mu\widehat y\quad(i=q,r)
\end{equation}
define an integrable meromorphic connection. Its finite poles are
logarithmic at the eight distinct values $u_j$; infinity is regular
singular. For each normalized idempotent $\psi_j$, there is a unique
holomorphic local solution of its $\lambda$ equation with
$\widehat y(u_j)=\psi_j$. It is also horizontal in the parameter directions when the normalized
idempotent is continued consistently.
\end{proposition}
\begin{proof}
The identities already proved for the small quantum connection are
\begin{gather*}
 [M_q,M_r]=0,\quad Q_qM_r=Q_rM_q,\quad
 Q_iC=M_i+[M_i,\mu],\quad [M_i,C]=0.
\end{gather*}
They imply integrability directly. For example, put
$A=R_\lambda\mu$, $B_i=-M_iR_\lambda\mu$.
Since $Q_iR_\lambda=R_\lambda(Q_iC)R_\lambda$ and
$\partial_\lambda R_\lambda=-R_\lambda^2$, the mixed curvature is
\[
 Q_i A-\partial_\lambda B_i+[B_i,A]
 =R_\lambda\bigl(Q_iC-M_i-[M_i,\mu]\bigr)R_\lambda\mu=0.
\]
The remaining curvature is, after the same substitutions,
\[
 Q_qB_r-Q_rB_q+[B_q,B_r]
 =-(Q_qM_r-Q_rM_q)R_\lambda\mu=0.
\]
In this reduction the terms containing $[M_i,\mu]$ cancel the matrix
commutator, using $[M_q,M_r]=[M_i,R_\lambda]=0$.
This also fixes the order of $\mu$ and the resolvent in
\eqref{smallLaplaceEq}.

Let $P_j$ be the rank-one spectral projector. The residue of the
$\lambda$ connection at $u_j$ is $N_j=P_j\mu$.
Self-adjointness of $C$ and skew-adjointness of $\mu$ imply
$P_j\mu P_j=0$, hence $N_j^2=0$ and $N_j\psi_j=0$.
Writing $t=\lambda-u_j$ and
$R_\lambda\mu=N_j/t+\sum_{k\geq0}B_kt^k$, the coefficients in the analytic
expansion $\widehat y=\sum_{n\geq0}y_nt^n$ are determined by
\begin{equation}\label{dualrecurrence}
 (nI+N_j)y_n=-\sum_{k=0}^{n-1}B_k y_{n-1-k},\qquad
 y_0=\psi_j,\quad n\geq1.
\end{equation}
The inverse is $n^{-1}I-n^{-2}N_j$. At this regular singular point, a geometric majorant proves convergence
of the recursion on every smaller disc free of the other poles, locally
uniformly in the parameters. Thus the normalized germ is unique and depends holomorphically on the
parameters.

In coordinates $(q,r,t)$ along the moving divisor, the parameter
connection matrix is
$\bigl((Q_i u_j)I-M_i\bigr)R_\lambda\mu$. Its residue vanishes because
$Q_i u_j=m_{ij}$, the eigenvalue of $M_i$ on $P_j$.
For $k\ne j$, differentiating $C\psi_j=u_j\psi_j$ gives
\[
 (\psi_k,Q_i\psi_j)
  =\frac{m_{ik}-m_{ij}}{u_j-u_k}(\psi_k,\mu\psi_j).
\]
Together with $(\psi_j,Q_i\psi_j)=0$, which follows from the constant
unit norm, this is precisely the initial tangential horizontality
condition. Integrability and uniqueness of \eqref{dualrecurrence}
then imply that the entire germ is horizontal. Finally
$R_\lambda\mu=\mu/\lambda+O(\lambda^{-2})$, proving regular
singularity at infinity.
\end{proof}

If the ray $u_j+\mathbb R_{>0}e^{i\phi}$ contains no other pole, continue
this germ along it and set
\begin{equation}\label{smallLaplaceIntegral}
 y(z)=\frac1z\int_{u_j+\mathbb R_{\geq0}e^{i\phi}}
                     \widehat y(\lambda)e^{-\lambda/z}\,d\lambda.
\end{equation}
Regular singularity at infinity bounds the growth along the ray by
powers and logarithms, so the integral converges for $|\arg z-\phi|<\pi/2$.
Integration by parts, using $(u_j-C)\widehat y(u_j)=0$, proves its
$z$ equation. Differentiation in the parameters includes the term from the moving
endpoint $-(Q_i u_j)\psi_j e^{-u_j/z}/z$; this cancels exactly because
$M_i\psi_j=(Q_i u_j)\psi_j$. The integral is therefore a flat section of the full small quantum
connection, including the parameter directions and the $z$ direction.
Watson's lemma at the analytic endpoint gives
$y=e^{-u_j/z}(\psi_j+O(z))$ with its complete expansion. A small rotation of the ray without crossing a pole gives the same
integral on overlapping sectors and extends the expansion to a sector
of opening greater than $\pi$.

It follows that if a continuously labelled $u(t)$ and a continuous
phase $\phi(t)$ satisfy
\begin{equation}\label{admissibleraycondition}
 \bigl(u(t)+\mathbb R_{>0}e^{i\phi(t)}\bigr)
 \cap\bigl(\operatorname{Spec}C(t)\setminus\{u(t)\}\bigr)=\varnothing,
\end{equation}
then the continued normalized column remains canonical with that phase.
A finite covering of the compact path by the above pole-free local
constructions proves the assertion. This is the specialization of \cite[Proposition 3.6]{HK} to the small
quantum connection; the eigenvalue need not remain rightmost. At fixed parameters the same contour deformation
keeps the column unchanged within a connected component of allowed ray
phases. Thus the construction provides the analytic continuation needed
along the parameter circle.

\subsection{Line bundle columns and an exceptional collection}
\begin{proposition}[Line columns from clockwise transport]
\label{lineLtransport}
At $(q_*,1)=(\epsilon^3/4,1)$ the Gamma section of $L$ is canonical at
$u=a-ib$ with phase $-2\pi/3$. The section of $L^{-1}$ is canonical at
$a+ib$ with phase $2\pi/3$. They admit any phase in $(-\pi,0)$ and $(0,\pi)$, respectively.
\end{proposition}
\begin{proof}
Set
\[
 q(t)=q_*e^{-2\pi it},\quad r(t)=1,\quad
 \rho(t)=\epsilon e^{-2\pi it/3},\quad
 \alpha(t)=-2\pi t/3\qquad(0\leq t\leq1).
\]
Continue the root in the disc centered at $\rho(t)$ and denote its value
by $u_0(t)$. At $t=0$ it is $T$; at $t=1$ it is $a-ib$.
Take the ray phase $\phi(t)=\alpha(t)$. The initial column is
$\mathfrak s(\OO)$ by Theorem~\ref{positiveGammaall} and
Remark~\ref{fastestrayremark}.

For either other vertex $u_j$ in the same cluster,
\[
 \operatorname{Re}\left(e^{-i\alpha}(u_j-u_0)\right)
 \leq-\tfrac92\epsilon+8\epsilon^2<0.
\]
For its center $2$ the bound is $-3\epsilon+4\epsilon^2<0$.
For every value in the cluster near the negative center, including $-2$, its
difference $d$ from $u_0$ obeys $|d+4|<7\epsilon$. If
$-\pi/3\leq\alpha\leq0$, then
\[
 \operatorname{Re}(e^{-i\alpha}d)<-2+7\epsilon<0.
\]
If $-2\pi/3\leq\alpha\leq-\pi/3$, then
\[
 \operatorname{Im}(e^{-i\alpha}d)<-2\sqrt3+7\epsilon<0.
\]
In neither case is $d$ a positive multiple of $e^{i\alpha}$.
Thus \eqref{admissibleraycondition} is verified for the whole path.
Lemma~\ref{clusterdiscs} also verifies distinct spectrum throughout it.
The Frobenius points are therefore tame along this path; tameness of
the Landau--Ginzburg potential at infinity is not required.

We can therefore transport the canonical column. At the endpoint the logarithm of
$q$ has decreased by $2\pi i$, and the marked Galois identity
$\mathcal M_q\mathfrak s(\OO)=\mathfrak s(L)$ identifies the resulting
flat section. The leading vector retains the normalization and sign determined by
continuation from the initial point. Reflection proves
the assertion for $L^{-1}$.

At the endpoint $a-ib$ has minimal imaginary part among all eight
values (shared only by $-a-ib$). Every difference from $a-ib$ has
nonnegative imaginary part. Hence no outgoing ray in the open lower half-plane meets another value,
and the ray rotation lemma at fixed parameters allows every phase in
$(-\pi,0)$. The analogous argument in the upper half-plane proves the
last assertion for $L^{-1}$.
\end{proof}

\begin{corollary}[Six line columns with Gamma marking]\label{sixlinecolumns}
At $(q_*,1)$, the six distinct exponents and available phases are
\begin{center}
\renewcommand{\arraystretch}{1.15}
\begin{tabular}{c|c|c}
\toprule
object & exponent & available phases used here\\
\midrule
$L^{-1}$ & $a+ib$ & $(0,\pi)$\\
$\OO$ & $T$ & a neighborhood of $0$\\
$L$ & $a-ib$ & $(-\pi,0)$\\
$U$ & $-a-ib$ & $(-\pi,0)$\\
$LU$ & $-T$ & a neighborhood of $-\pi$\\
$L^2U$ & $-a+ib$ & $(-2\pi,-\pi)$\\
\bottomrule
\end{tabular}
\end{center}
Every entry denotes a canonical vector column, and hence has a full
sectorial expansion on a sector centered at its chosen phase with
opening strictly greater than $\pi$. The normalizations are fixed by
Gamma marking, reflection and \eqref{halfturnformula}.
\end{corollary}
\begin{proof}
The first three rows follow from Proposition~\ref{lineLtransport} and the
positive Gamma I column. Apply the marked half-turn to
$L^{-1},\OO,L$: their twists by $LU$ are respectively $U,LU,L^2U$.
The initial phases for the last three rows are $-\pi/3,-\pi,-5\pi/3$.
The value $-a-ib$ is a minimum of the imaginary parts and $-a+ib$ is a
maximum, so rotating the rays without crossing a pole gives the stated phase
intervals, with their indicated lifts to the logarithmic cover.
In particular $-5\pi/3$ may be increased to $-\pi-\epsilon$ without
crossing a forbidden direction. This construction changes the first quadric window. It does not identify
the canonical column of $L^2$ with that of $L^{-1}$, since these line
objects have different classes in $K_0(X)$.
\end{proof}

The eight objects used below are
\begin{equation}\label{adaptedfoundation}
 \mathcal E=(L^{-1},S_0,\OO,L,S_1,U,LU,L^2U).
\end{equation}
\begin{proposition}[Fullness of the adapted foundation]\label{adaptedfull}
The ordered collection \eqref{adaptedfoundation} is full exceptional.
\end{proposition}
\begin{proof}
On $Q^3$ both
\[
 (S,\OO,\OO(1),\OO(2))
 \quad\hbox{and}\quad
 (\OO(-1),S,\OO,\OO(1))
\]
are full exceptional collections. To obtain the second collection, move the last member
of the first collection left through the preceding three. The
helix--Serre formula yields
$\mathsf S_{Q^3}(\OO(2))=\OO(-1)[3]$ in the mutation convention of this
paper. Shifting this first member by $[-3]$ gives the displayed
collection; shifts preserve exceptionality and fullness.
Apply Orlov's projective bundle formula with the second collection in
the untwisted block and the first collection in the $U$-twisted block.
Semiorthogonality between the two blocks is independent of the choice
of foundation within either block. This proves \eqref{adaptedfoundation}.
\end{proof}
\subsection{A common phase and its exceptional collection}\label{sec:unitphase}
The preceding construction identifies the eight individual columns.
To compare them at the unit point, fix
\begin{equation}\label{fixedphase09}
 m=10^{-2},\qquad \phi_0=-\arctan m,\qquad
 q_0=\frac{m^3}{4}=\frac1{4\,000\,000},\qquad r=1.
\end{equation}
This choice sets $\rho_0=(4q_0)^{1/3}=m$. The root bounds above then give
the inequalities used below to determine the initial ordering. The
starting point and the phase are therefore chosen together.

Logarithms start real at $q_0$ and remain real on the return path
$q\in[q_0,1]$. We retain the displayed lift of the phase throughout, rather than
reducing it modulo $2\pi$. The Gamma columns retain their chosen leading
signs as they cross each wall.

Write $\mathscr L_EF$ and $\mathscr R_FE$ for the left and right mutations
with the cone conventions of Section~\ref{sec:conventions}. The calligraphic notation distinguishes left mutation from the line
bundle $L=\OO_X(H)$. Define the following objects in the derived category:
\begin{align}
 A_0&=\mathscr R_{\OO}S_0\cong S_0L[-1],\notag\\
 B_0&=L^{-1}U[3],& C_0&=\mathscr L_U(LU),\label{objects09}\\
 B_1&=\mathscr L_{\OO}B_0,& C_1&=\mathscr L_L C_0.\notag
\end{align}
Here $C_0,C_1$ denote objects and should be distinguished from the
anticanonical multiplication matrix $C$. Their $K$-classes are
\begin{equation}\label{objectclasses09}
\begin{gathered}
 [A_0]=[S_0]-4[\OO],\qquad [B_0]=-[L^{-1}U],\qquad
 [C_0]=[LU]-5[U],\\
 [B_1]=[\OO]-[L^{-1}U],\qquad [C_1]=[LU]-5[U]-[L].
\end{gathered}
\end{equation}

The target ordered foundation is
\begin{equation}\label{unitfoundation09}
 \mathcal B_1=(L^{-1},B_1,\OO,A_0,S_1,C_1,L,U).
\end{equation}
Its exponent order at $q=r=1$ is
\begin{equation}\label{unitorder09}
 (a+ib,-a+ib,T,2,-2,-T,a-ib,-a-ib).
\end{equation}
We prove Theorem~\ref{mainUnit09} by following this ordered collection.

We construct the columns using the spinor sector estimates of
Section~\ref{sec:spinorMB}, the positive Gamma column, and canonical
transport. We also use the mutation law for canonical columns, recalled
below with its sign and rotation convention. The resulting Euler matrix
is determined by this analytic construction. For the flat nonsymmetric pairing put
\begin{equation}\label{Stokesconvention09}
 [y_i,y_j)=(y_i(e^{-\pi i}z),y_j(z))_X,\qquad S_{ij}=[y_i,y_j).
\end{equation}
For Gamma sections this pairing equals $\chi(E_i,E_j)$ by the Gamma
factorization of Hirzebruch--Riemann--Roch. With the ordering just stated,
$S$ is upper triangular with diagonal one. These are the conventions of
\cite[Proposition 2.6.4 and Corollary 2.6.6]{GGI}; the opposite Stokes
matrix in that convention is $S_-=S^{\mathsf T}$.

More locally, when the outgoing ray of a canonical column $v$ turns
counterclockwise through a pole whose canonical column at that crossing
is $w$, its replacement is
\begin{equation}\label{leftanalytic09}
 v\longmapsto v-[w,v)w.
\end{equation}
Clockwise crossing gives $v\mapsto v-[v,w)w$. In each application below,
both $v$ and the column $w$ at the crossing have been identified as Gamma
sections. The analytic mutation law therefore allows us to compute its
coefficient from the established flat pairing. This requires the
representative of $w$ at that crossing, with its current continuation.

For an ordered exceptional pair $(a,b)$ the corresponding elementary
operations on $K$-classes are
\[
 R_i:(a,b)\mapsto(b,a-\chi(a,b)b),\qquad
 L_i:(a,b)\mapsto(b-\chi(a,b)a,a).
\]
An indexed word in these operations always refers to the \emph{current}
ordered foundation, with indices beginning at one. Each displayed list is performed from left to right.

At $q=q_0$, use $T,a,b$ from Lemma~\ref{clusterdiscs}, with
$\epsilon=m$. In particular
\begin{equation}\label{smallinequalities09}
 2+2m<T<2+4m,\qquad 2-2m<a<2-m,\qquad 2m<b<3m.
\end{equation}
The earlier results give canonical columns for $L^{-1},L,U$ in the upper,
lower and lower open half-planes of ray phases, respectively; $LU$ has
its initial phase near $-\pi$ and $L^2U$ has its initial phase below
$-\pi$. The two spinor columns and their reflected representatives have
the sectors in Section~\ref{sec:asymptotics}.

\begin{proposition}[The canonical foundation at the initial point]
\label{smallfoundation09}
At $(q_0,1)$ and phase $\phi_0$, all eight canonical columns are the
Gamma sections of
\begin{equation}\label{smallbasis09}
 \mathcal B_0=(L^{-1},\OO,A_0,B_0,L,S_1,C_0,U),
\end{equation}
in decreasing order of projected height. This is a full exceptional
collection, obtained from \eqref{adaptedfoundation} by the chronological
word
\begin{equation}\label{smallword09}
 R_2,\ L_7,\ L_6,\ L_5,\ L_4,\ L_7.
\end{equation}
\end{proposition}
\begin{proof}
We first determine the canonical columns and then identify the
corresponding exceptional collection. Six columns follow directly from the preceding results. The structure sheaf remains canonical near phase
zero. Reflection gives $A_0=S_0L[-1]$ at exponent $2$ on the negative side,
and $S_1$ is canonical at exponent $-2$. The columns of $L$ and $U$ are
canonical throughout $(-\pi,0)$. For $L^{-1}$, rotate its known upper
half-plane ray down through zero to $\phi_0$. The first pole it could cross at a negative angle is $T$, at angle
$-\arctan(b/(T-a))$. This is strictly below $\phi_0$ by
\eqref{smallinequalities09}, so this last short rotation has no crossing.

For $LU$, rotate the ray based at $-T$ counterclockwise from $-\pi$ to
$\phi_0$. The only crossed pole is $-a-ib$. Its column is $\mathfrak s(U)$,
and its direction is $-\arctan(b/(T-a))$. The next possible pole $a-ib$
has direction $-\arctan(b/(T+a))$, strictly above $\phi_0$. Indeed
\[
 \frac{b}{T-a}>m,\qquad \frac{b}{T+a}<m.
\]
The coefficient in \eqref{leftanalytic09} is $\chi(U,LU)=5$; hence the
column at the target phase is $\mathfrak s(C_0)$.

The ray based at $-a+ib$ needs more care. Begin with its already marked
column $\mathfrak s(L^2U)$ at a phase just below $-\pi$, and rotate
counterclockwise to $\phi_0$. The crossed poles and the columns to use
\emph{at those crossings} are
\begin{equation}\label{upperrotationtable09}
\begin{array}{c|c|c|c}
 \text{pole}&\text{direction from }-a+ib&\text{crossing class }w&[w,v)\\\hline
 -T&-\pi+\arctan\dfrac b{T-a}&LU&5\\[3pt]
 -2&-\pi+\arctan\dfrac b{2-a}&S_1-4U&4\\[3pt]
 -a-ib&-\pi/2&U&5\\[3pt]
 a-ib&-\arctan\dfrac ba&L&0
\end{array}
\end{equation}
The last column is evaluated against the current $v$, after all preceding
rows. The directions are strictly ordered as displayed. The remaining directions at negative angles, towards $2$ and $T$, are greater than $\phi_0$
because $b/(a+2)<m$ and $b/(T+a)<m$. Both inequalities follow from
\eqref{smallinequalities09}; also $b/a>m$. Thus this is the entire list.

The second row requires the lateral representative at the crossing. At that angle, which
lies in $(-\pi,-\pi/2)$, the canonical $-2$ column with the fixed leading
vector is $\mathfrak s(S_1-4U)$. This follows from applying the marked
half-turn to $S_0$ and changing its sign: $[S_0LU]=4[U]-[S_1]$.
Using $S_1$ itself would give the wrong representative at this crossing.
At the first row $LU$ is still canonical, since its own first crossing
is much later, at $-\arctan(b/(T-a))$. At the last two rows the lower
half-plane columns $U,L$ are available.

Applying \eqref{leftanalytic09} gives in succession
\begin{align*}
 L^2U&\longmapsto L^2U-5LU\\
 &\longmapsto L^2U-5LU-4S_1+16U\\
 &\longmapsto L^2U-5LU-4S_1+11U=-L^{-1}U.
\end{align*}
The fourth coefficient is zero. The final class is $[B_0]$, so additivity
of the Gamma framing identifies the actual canonical column with
$\mathfrak s(B_0)$. The coefficients $5,4,5,0$ are obtained by pairing the identified
classes with the fixed Euler form, as justified by the analytic mutation
law.

At the chosen phase, projected heights may be divided by
$\cos\phi_0>0$. Their order is
\[
 b+ma>mT>2m>b-ma>ma-b>-2m>-mT>-b-ma.
\]
The required nontrivial comparisons are precisely the strict slope
inequalities used above. This gives \eqref{smallbasis09} in canonical
order.

For categorical fullness, perform \eqref{smallword09} on
$(L^{-1},S_0,\OO,L,S_1,U,LU,L^2U)$. The first move produces $A_0$.
The next three left mutations move $L^2U$ through $LU,U,S_1$ and yield
$L^{-1}U[3]$ by the quadric helix identity, within the fully faithful
$U$-block. The next move exchanges this object with $L$; the two objects are orthogonal in all degrees. In the nontrivial direction this reduces
to the cohomology of $\OO_{Q^3}(-2)\oplus\OO_{Q^3}(-1)$, which vanishes;
the reverse direction has the acyclic fiber factor $\OO_\pi(-1)$.
The final move is $\mathscr L_U(LU)=C_0$. Mutations and orthogonal
exchanges preserve exceptionality and fullness. This gives \eqref{smallbasis09}. The categorical mutation order differs
from the order of the individual ray crossings in
\eqref{upperrotationtable09}.
\end{proof}

\subsection{The return path and its two walls}
The following parameterization controls the entire return path to $q=1$
and determines every crossing of a phase wall.

\begin{lemma}[Positive spectral parameter]\label{realparameter09}
For $r=1$ and $0<q<27/16$, there is a unique $\tau\in(1,2)$ with
\begin{equation}\label{qoftau09}
 q=Q(\tau):=\frac{(\tau^2-1)^3}{4\tau^2}.
\end{equation}
The eight distinct critical values are $\pm T$, $\pm2$ and $\pm a\pm ib$,
where
\begin{equation}\label{spectralparam09}
 T=4\tau-\frac2\tau,\qquad
 a=\frac{(\tau+1)(2-\tau)^{3/2}}{\sqrt\tau},\qquad
 b=\frac{(\tau-1)(\tau+2)^{3/2}}{\sqrt\tau}.
\end{equation}
Here $a,b>0$. On $(1,2)$, $T,b$ increase, $a$ decreases, and the two
ratios
\begin{equation}\label{monotonewallratios09}
 \frac b{a+2},\qquad \frac b{T+a}
\end{equation}
are strictly increasing. Moreover
\begin{equation}\label{separationestimate09}
 \frac b{T-a}\geq\frac1{\sqrt3}.
\end{equation}
\end{lemma}
\begin{proof}
For the positive saddle write $h/v=x$ in \eqref{positivesaddlehv}.
Then $v=(1+x)^{-1/2}$, $h=x(1+x)^{-1/2}$. Taking
$\tau=\sqrt{1+x}$ gives \eqref{qoftau09} and the formula for $T$.
Substitution in the characteristic polynomial yields the exact
factorization
\[
 (u^2-4)(u^2-T^2)
 \left(u^4-2(a^2-b^2)u^2+(a^2+b^2)^2\right),
\]
with $a^2=(2-\tau)^3(\tau+1)^2/\tau$ and
$b^2=(\tau-1)^2(\tau+2)^3/\tau$. Positivity fixes the square roots in
\eqref{spectralparam09}. Also
\[
 Q'=\frac{(\tau^2-1)^2(2\tau^2+1)}{2\tau^3}>0,
 \qquad Q(1)=0,\quad Q(2)=27/16.
\]
For the remaining claims, differentiation gives
\begin{align}
 \frac{a'}a&=\frac{2\tau^2+1}{\tau(\tau-2)(\tau+1)}<0,&
 \frac{b'}b&=\frac{2\tau^2+1}{\tau(\tau-1)(\tau+2)}>0,\label{abderivative09}\\
 \frac{d}{d\tau}\log\frac bT
 &=\frac{(2\tau^2+1)(\tau^2-\tau+1)}
 {\tau(\tau-1)(\tau+2)(2\tau^2-1)}>0.\notag
\end{align}
The signs of these derivatives show that the first ratio in \eqref{monotonewallratios09} is increasing.
For the second, its derivative has numerator
\[
 b'(T+a)-b(T'+a')=(b'T-bT')+(b'a-ba')>0.
\]

Finally $T'\leq6$ and $-a'\leq3$ on $[1,2]$. To verify the latter after
squaring positive quantities, use
\[
 9\tau^3-(2\tau^2+1)^2(2-\tau)
 =(\tau-1)(4\tau^4-4\tau^3+9\tau^2+\tau+2)\geq0.
\]
Also $(\tau+2)^3/\tau$ is increasing, since its derivative is
$2(\tau-1)(\tau+2)^2/\tau^2$. Thus
$b\geq3\sqrt3(\tau-1)$ and $T-a\leq9(\tau-1)$, proving
\eqref{separationestimate09}.
\end{proof}

\begin{theorem}[Exactly two walls and four pair crossings]
\label{twowalls09}
On $q\in[q_0,1]$, $r=1$, the fixed phase $\phi_0$ has exactly two
inadmissible parameter values $q_A<q_B$. They are the unique solutions
of
\begin{equation}\label{walls09}
 b=m(a+2)\quad(q=q_A),\qquad
 b=m(T+a)\quad(q=q_B).
\end{equation}
At each wall two disjoint adjacent pairs cross simultaneously. Following
the canonical foundation from \eqref{smallbasis09}, the chronological
mutation word on the return path is
\begin{equation}\label{returnword09}
 L_3,\ L_5\ ;\ L_2,\ L_6,
\end{equation}
where the semicolon separates the two walls. The corresponding Euler
coefficients are $0,0;-1,1$. Consequently the unit foundation is
\eqref{unitfoundation09}.
\end{theorem}
\begin{proof}
Divide each projected height by $\cos\phi_0$. The four positive heights
are
\[
 b+ma,\qquad mT,\qquad 2m,\qquad b-ma;
\]
the remaining four are their negatives. The first is always the largest:
$b-m(T-a)>0$ by \eqref{separationestimate09}, and all other comparisons
are immediate. The fourth remains positive because $b/a$ is increasing
and $b/a>m$ at $q_0$. Also $mT>2m$. Thus the only possible changes in
order are the crossings of $b-ma$ with $2m$ and $mT$, exactly
\eqref{walls09}.

At $q_0$ both ratios in \eqref{monotonewallratios09} are below $m$, by
\eqref{smallinequalities09}. At $q=1$ one has
$9/5<\tau<19/10$, obtained by substituting these rational numbers in
$Q$. In particular $b>4$, $a<2$ and $T<7$, so both ratios are greater
than $m$. Strict monotonicity proves existence, uniqueness and
transversality of each wall. Since $T>2$, the ratio $b/(T+a)$ is smaller
than $b/(a+2)$ at the same point. Hence $q_A<q_B$.

The canonical order on the three successive intervals is
\begin{align*}
 &(a+ib,T,2,-a+ib,a-ib,-2,-T,-a-ib),\\
 &(a+ib,T,-a+ib,2,-2,a-ib,-T,-a-ib),\\
 &(a+ib,-a+ib,T,2,-2,-T,a-ib,-a-ib).
\end{align*}
At the first wall the pair $(A_0,B_0)$ exchanges places with coefficient
zero, as does $(L,S_1)$. The former is orthogonal in all degrees:
up to shifts its nontrivial Hom direction reduces to
$R\Gamma_{Q^3}(S(-1)\oplus S)$, and the reverse direction again has
fiber $\OO_\pi(-1)$. For the second pair,
$R\operatorname{Hom}_X(L,S_1)=R\Gamma_{Q^3}(S(-1)\oplus S)=0$;
the reverse direction has the acyclic fiber factor $\OO_\pi(-1)$ as well. Thus these are the orthogonal
exchanges $L_3,L_5$.

After them the order of objects is
$(L^{-1},\OO,B_0,A_0,S_1,L,C_0,U)$. At the second wall the pairs
$(\OO,B_0)$ and $(L,C_0)$ cross, with
\[
 \chi(\OO,B_0)=-1,\qquad \chi(L,C_0)=1.
\]
The pole to the right of the ray based at the lower value crosses the
ray as its angle decreases through $\phi_0$. Equivalently, the ray turns
counterclockwise across the pole in the local moving frame. The
orientation in \eqref{leftanalytic09} therefore gives left mutations.
The replacements are $B_0+\OO$ and $C_0-L$. This proves
\eqref{returnword09} and gives \eqref{unitfoundation09}. Each wall has
two disjoint pairs with distinct locations, so the two operations at a
single wall commute. There are no other crossings by the height argument.
Between walls canonical transport applies on the real parameter path,
and no logarithm acquires a turn. The Gamma labels at the regular end are therefore the indicated
integral linear combinations.
\end{proof}

The walls also have the following exact rational isolating intervals:
\begin{equation}\label{rationalwallintervals09}
\begin{split}
 \frac{1007653852}{10^9}<\tau_A<\frac{1007653853}{10^9},\\
 \frac{1007742514}{10^9}<\tau_B<\frac{1007742515}{10^9}.
\end{split}
\end{equation}
To obtain these rational bounds, let
$A=a^2$, $B=b^2$, and let $c=2$ or $c=T$. All are rational functions at
a rational $\tau$. Put $D=B-m^2(A+c^2)$. If $D\leq0$, then
$b-m(a+c)<0$. If $D>0$, its sign is the sign of
$D^2-4m^4c^2A$. Substitution of the rational endpoints in
\eqref{rationalwallintervals09} gives opposite strict signs. Together
with the monotonicity established above, this proves the stated bounds.

\subsection{Geometric objects and the connection matrices}
The unit foundation includes objects other than vector bundles. The two
new objects have simple geometric descriptions that determine their
shifts as well as their classes in $K_0(X)$.

\begin{proposition}[A section sheaf and a nontrivial extension]
\label{geometricobjects09}
Let $j:Q^3\hookrightarrow X$ be the section given by the quotient onto
$\OO$, and set
\[
 V=\pi^*(\Omega^1_{\PP^4}(1)|_{Q^3}).
\]
A unique nonzero extension class, up to scalar, defines a bundle $F$ in
\begin{equation}\label{extensionF09}
 0\longrightarrow VU\longrightarrow F\longrightarrow L\longrightarrow0.
\end{equation}
The objects in \eqref{objects09} satisfy
\begin{equation}\label{geometricB1C1}
 B_1\cong j_*\OO_{Q^3}(-1)[3],\qquad C_0\cong VU[1],\qquad
 C_1\cong F[1].
\end{equation}
Thus \eqref{unitfoundation09} can also be written
\begin{equation}\label{geometricfoundation09}
 (L^{-1},j_*\OO_{Q^3}(-1)[3],\OO,S_0L[-1],S_1,F[1],L,U).
\end{equation}
\end{proposition}
\begin{proof}
The section $j$ is the divisor of class $\xi-H$: the map
$\pi^*\OO_{Q^3}(1)\to\OO_\pi(1)$ vanishes precisely on the quotient
onto $\OO$. Its normal bundle is $\OO_{Q^3}(-1)$. Consequently
\[
 0\longrightarrow\OO_X\longrightarrow L^{-1}U
   \longrightarrow j_*\OO_{Q^3}(-1)\longrightarrow0.
\]
Moreover $R\Gamma_X(L^{-1}U)=\C$ in degree zero, since its pushforward is
$\OO_{Q^3}(-1)\oplus\OO_{Q^3}$. Left mutation of $B_0=L^{-1}U[3]$
through $\OO$ is therefore the cone of the displayed inclusion shifted
by three, proving the assertion for $B_1$.

The restricted Euler sequence gives
$0\to V\to\OO_X^{\oplus5}\to L\to0$. Tensoring by $U$ identifies
$\mathscr L_U(LU)$ with $VU[1]$. On the base, $V_Q=\Omega^1_{\PP^4}(1)|_{Q^3}$
is acyclic: the map $H^0(\OO_{Q^3})^{\oplus5}\to H^0(\OO_{Q^3}(1))$
is an isomorphism. Its twist by $-1$ has only $H^1=\C$, from
$0\to V_Q(-1)\to\OO_{Q^3}(-1)^{\oplus5}\to\OO_{Q^3}\to0$.
All the stated vanishings for line bundles follow, for example, by
restricting the corresponding sequences on projective space to the
quadric. The projection formula then gives
\[
 R\operatorname{Hom}_X(L,VU)
 =R\Gamma_{Q^3}(V_Q(-1)\oplus V_Q),\qquad
 \operatorname{Ext}^1_X(L,VU)=\C
\]
and no other Ext groups. Thus the mutation map $L\to VU[1]$ is a nonzero
generator of that Ext group. Its cone is $F[1]$, with $F$ from
\eqref{extensionF09}. A nonzero scalar change gives an isomorphic
extension bundle. This proves \eqref{geometricB1C1}.
\end{proof}

\begin{proof}[Completion of the proof of Theorem~\ref{unitphasetheorem}]
Proposition~\ref{smallfoundation09} identifies every canonical column at
the fixed phase and the starting point. Theorem~\ref{twowalls09} then determines every wall crossing on the path
to $q=1$, including the mutation word in the chosen foundation. The
resulting eight objects remain a full exceptional collection. The Gamma
normalization and the mutation formulas fix their signs throughout the
path and at the endpoint. This proves the assertion about the canonical
basis.

Let $D_0,D_1$ be the integral column changes from the natural $K$-basis to
$\mathcal B_0,\mathcal B_1$. Both have determinant one. Their columns are
given by \eqref{objectclasses09} and
$[L^{-1}]=4[S_0]-11[\OO]+5[L]-[L^2]$; the matrix $D_1$ is displayed in
\eqref{unitbasischange}. Applying \eqref{Stokesconvention09} to
the now identified Gamma basis gives
\[
 S_{\phi_0}=D_1^{\mathsf T}GD_1,
\]
which is exactly the displayed matrix \eqref{unitStokes09}.
This matrix is at the specified unit point,
phase, object order and normalization. Negative entries reflect the
actual shifts and mutations; changing leading signs changes the matrix
by the corresponding diagonal sign conjugation. Other conventions may use the inverse or transpose; comparing them
with \eqref{unitStokes09} requires the corresponding change of convention.

For the central connection, choose the fixed frame at the regular end
$Y_\infty=\mathcal L z^{-\mu}z^\rho$ on the same logarithmic cover, so that
$Y_{\phi_0}=Y_\infty\mathcal C_{\phi_0}$. Its columns are
\begin{equation}\label{centralunit09}
 \mathcal C_{\phi_0}=(2\pi)^{-2}
 \left[\widehat\Gamma_X\Ch(B)\right]_{B\in\mathcal B_1,\,\mathrm{cup}}.
\end{equation}
The Chern characters of its ordered objects are
\begin{equation}\label{charactersunit09}
 e^{-H},\quad 1-e^{\xi-H},\quad 1,\quad
 (2-H+H^3/12)-4,\quad
 (2-H+H^3/12)e^\xi,\quad
 e^{H+\xi}-5e^\xi-e^H,\quad e^H,\quad e^\xi.
\end{equation}
To convert the ordinary Chern character to the modified one, multiply
its component of degree $p$ by $(2\pi i)^p$. In the cup algebra the Gamma factor has the compact
expression
\begin{equation}\label{gammalogunit09}
 \log\widehat\Gamma_X
 =-2\gamma_{\rm E}(H+\xi)+\zeta(2)H^2
   +\frac{\zeta(3)}3(4H^3-2H^2\xi).
\end{equation}
Its exponential is truncated in complex degree four. The fourth-degree
term in its logarithm is zero in this cup ring. Equations
\eqref{charactersunit09}--\eqref{gammalogunit09} determine all 64 entries
of \eqref{centralunit09} by multiplication in
$\C[H,\xi]/(H^4,\xi(\xi-H))$, with the common factor $(2\pi)^{-2}$
retained. These characteristic class formulas therefore determine the central
connection matrix explicitly.
\end{proof}

The determinants and full turn satisfy
\begin{equation}\label{unitdethelix09}
 \det S_{\phi_0}=1,\qquad \det\mathcal C_{\phi_0}=\frac1{16},\qquad
 S_{\phi_0}^{-\mathsf T}S_{\phi_0}=D_1^{-1}\mathcal R D_1.
\end{equation}
For the middle identity the determinant of the ordinary Chern character is
$1/16$, the sum of the cup degrees is $16$, and the eight scalar
normalizations contribute $(2\pi)^{-16}$. Multiplication by the Gamma
class has determinant one. The full turn still has Jordan type
$J_5(1)\oplus J_3(1)$. These algebraic identities are consequences of
the marked matrices; the identification of their columns is provided
by the preceding analytic argument.

\begin{proposition}[Chiral constraints on the specified matrices]
\label{unitChiral}
In the notation of Corollary~\ref{introchiral}, all identities
\eqref{introchiraldata} hold. Moreover,
\begin{equation}\label{chiralEulerequality}
 M_A^{\mathsf T}S M_A=S,\qquad
 \operatorname{rank}(M_A-I)^k=(8,6,4,2,1,0)_k\quad(0\leq k\leq5).
\end{equation}
If $Q=S^{-\mathsf T}S-I$, then $Q^5=0$ and
\begin{equation}\label{chiralrootformula}
 M_A=I+\frac12Q-\frac18Q^2+\frac1{16}Q^3-\frac5{128}Q^4.
\end{equation}
Thus the displayed expression, despite its rational coefficients, is
an integer matrix. It is the unique unipotent square root of
$S^{-\mathsf T}S$ over $\Q$.
\end{proposition}
\begin{proof}
Put $a=H+\xi$. The relative tangent line has first Chern class
$2\xi-H$, whose square is $H^2$ in the cup ring. Multiplying its
Todd class by that of $Q^3$ and then by $e^{-a}$ gives
\begin{equation}\label{chiraltodd}
 e^{-a}\Td(X)=1-\frac{H^2}{12}.
\end{equation}
Consequently, \eqref{chiralgammapairing} can be computed directly as
\begin{equation}\label{chiralRR}
 \chi(B_i\otimes A,B_j)=
 \int_X\ch(B_i^\vee)\ch(B_j)\left(1-\frac{H^2}{12}\right).
\end{equation}
Here $B_i^\vee=R\mathcal Hom(B_i,\OO_X)$ denotes the derived dual.
Substituting the eight characters \eqref{charactersunit09} shows
that all entries vanish except those with $i+j=9$. For
$(i,j)=(1,8),(2,7),(3,6),(4,5)$ the entries are respectively
$1,1,1,-1$, and the form is symmetric. Its matrix is therefore
$J_\epsilon$, with the two central negative signs as stated.
For instance, the first of these four entries is
$\int_X e^a(1-H^2/12)=30/24-6/24=1$.

Since $S_{ij}=\chi(B_i,B_j)$ and columns of $M_A$ express
$[B_i\otimes A]$ in $\mathcal B$, we obtain
\[
 M_A^{\mathsf T}S=J_\epsilon,\qquad
 M_A^{\mathsf T}S M_A=S.
\]
The second identity follows by tensoring both arguments of the Euler
pairing with $A$. Thus $J_\epsilon M_A=S$. As
$J_\epsilon^{\mathsf T}=J_\epsilon^{-1}=J_\epsilon$, this gives
$M_A=J_\epsilon S$ and $S^{\mathsf T}J_\epsilon S=J_\epsilon$.
It follows that $J_\epsilon S J_\epsilon=S^{-\mathsf T}$ and hence
$M_A^2=S^{-\mathsf T}S$, in agreement with
\eqref{unitdethelix09}.

For the central matrix, its column formula gives
$\mathcal C M_A=e^{\pi i\rho}\mathcal C$ immediately.
The normalization of the symmetric form can also be checked without
using any monodromy matrix. Equation~\eqref{gammalogunit09} gives
\[
 (\mathsf D\widehat\Gamma_X)\widehat\Gamma_X
 =\exp(2\zeta(2)H^2)=1+\frac{\pi^2}{3}H^2
 =\Ch\left(1-\frac{H^2}{12}\right).
\]
Here $\Ch$ on a cohomology class means multiplication of its complex
degree-$p$ component by $(2\pi i)^p$. Taking top degree and using
$(2\pi i)^4=(2\pi)^4$ shows that the $(i,j)$ entry of
$\mathcal C^{\mathsf T}P\mathsf D\mathcal C$ equals
\eqref{chiralRR}. The two factors $(2\pi)^{-2}$ in the Gamma columns
are therefore essential, and the result is exactly $J_\epsilon$.

Finally, under the rational Chern character, tensoring by $A$ is
multiplication by $e^a$. Since $A$ is ample, hard Lefschetz and the
Betti numbers $(1,2,2,2,1)$ give two Jordan strings, of lengths five and three, for multiplication
by $a$. Also
$a^4=15H^3\xi\ne0$. The operator $e^a-I$ is $a$ times an invertible
polynomial in $a$, so it has the same nilpotent Jordan type. This
proves the ranks and the unipotent Jordan type of $M_A$.
The finite binomial expansion of $(I+Q)^{1/2}$ gives
\eqref{chiralrootformula}; uniqueness follows from the unique
nilpotent logarithm of a unipotent matrix in characteristic zero.
Integrality follows from its identification with tensoring by $A$
in an integral exceptional basis.
\end{proof}

The signed reversal has a spectral meaning. The ordered exponents
in Theorem~\ref{mainGammaII} satisfy $u_{9-i}=-u_i$.
The flat form $b_{\rm ch}$ pairs precisely these opposite exponents;
its two central negative signs record the spinor normalizations and
shifts. Thus \eqref{introchiraldata} relates the chiral spectrum to the
Gamma marking and provides algebraic checks on both connection matrices,
once their columns have been identified by the sectorial analysis.

This proves Gamma II for the fourfold at the specified phase, with an
explicit mutation word along the return path and the analytic mutation
theorem of \cite{GGI}. For higher-dimensional bases, the result remains
the two-column theorem of Section~\ref{sec:familytransfer}. A geometric
integral mirror theorem would require further work: neither
\eqref{unitStokes09} nor the unimodularity of its Euler form identifies a
Barnes contour with a primitive thimble. Here the analytic identification
follows from the wide-sector estimates of Section~\ref{sec:spinorMB} and
the transport with Gamma normalization in Section~\ref{sec:closure}.

\appendix
\section{Mirror geometry and degeneration of the quantum connection}
\label{app:mirrorgeometry}
The coupled mirror realizes the quantum algebra and its collisions
geometrically. The constructions below concern the critical algebra
and the associated complex meromorphic connection. They do not
establish the comparison of integral cycles used in toric Gamma II
\cite{FZ}.

\subsection{The coupled critical model}\label{app:completedpotential}\label{aux:sec:mirrorcompletion}

We construct the completion using the model of the three-dimensional
quadric given by Gorbounov and Smirnov \cite{GS} and its comparison
with the canonical mirror by Pech and Rietsch \cite{PR}.
After a shift of the third coordinate, the potential is
\[
 W_Q^{\rm quad}(\alpha,\beta,\gamma)
 =\beta(1+\gamma)+\frac{Q\alpha^2}{(\alpha\beta-1)\gamma}
\]
on $\{\gamma(\alpha\beta-1)\ne0\}\subset\mathbb A^3$.
To couple this model to the fiber variable, substitute $Q=qv$ and add
$v+r/v$.
Set
\begin{equation}\label{aux:completedspace}
 Y=\{(\alpha,\beta,\gamma,v)\in
       \mathbb A^2\times(\C^*)^2:\alpha\beta-1\ne0\},
\end{equation}
and define
\begin{align}
 W_{q,r}&=\beta(1+\gamma)
 +\frac{qv\alpha^2}{(\alpha\beta-1)\gamma}+v+\frac rv,
 \label{aux:completedpotential}\\
 \Omega&=\frac{d\alpha\wedge d\beta\wedge d\gamma\wedge dv}
 {\gamma(\alpha\beta-1)v}.\label{aux:completedvolume}
\end{align}
Both are regular on $Y$, and $\Omega$ is nowhere zero. Here $q,r\ne0$.
Here \emph{completed} means the specified open enlargement of a Laurent
chart. It implies neither a proper compactification nor homological
mirror symmetry.

On $\beta\ne0$ put
\[
 x=\beta,\qquad y=\beta\gamma,\qquad k=\alpha\beta-1.
\]
This chart is $(\C^*)^4$, with inverse
$\alpha=(k+1)/x$, $\beta=x$, $\gamma=y/x$. In these coordinates
\begin{equation}\label{aux:Laurentpotential}
 L_{q,r}=x+y+\frac{qv(k+1)^2}{xyk}+v+\frac rv,
 \qquad \Omega=d\log x\wedge d\log y\wedge d\log k\wedge d\log v.
\end{equation}
The added divisor $\beta=0$ is essential for the finite critical scheme.

\begin{lemma}[Coupling of relative critical equations]\label{aux:couplinglemma}
Let $F_Q$ be a regular one-parameter family on a fixed smooth affine
variety $Z$, and put $\widetilde F=F_{qv}+v+r/v$ on $Z\times\C^*$.
Its critical algebra is the relative Jacobi algebra of $F_{qv}$, followed
by the additional relation
\[
 [Q\partial_QF_Q]+v-r/v=0.
\]
Under this description $[q\partial_q\widetilde F]=[Q\partial_QF_Q]$.
Parameter-dependent changes of coordinates do not change these classes.
\end{lemma}
\begin{proof}
First quotient by all derivatives in directions tangent to $Z$.
The remaining equation is $v\partial_v\widetilde F=0$, which is the
displayed relation. A change of the lift of a parameter derivative adds a
vector field tangent to $Z$ applied to $F_Q$; its class in the relative
Jacobi algebra is zero. This also proves the marking statement.
\end{proof}

\paragraph{The critical scheme.}
For an affine smooth variety, $\operatorname{Jac}(W)$ means its coordinate
ring modulo the ideal generated by the derivatives of $W$ along all
vector fields.

\begin{theorem}[A marked realization of the critical algebra]\label{aux:Jacobiiso}
Over $R=\C[q^{\pm1},r^{\pm1}]$ there is an isomorphism
\begin{equation}\label{aux:Jacobiisoeq}
 \operatorname{Jac}(W_{q,r})\cong
 R[H,\xi]/\bigl(H^4-4qH(\xi-H),\ \xi(\xi-H)-r\bigr),
 \qquad H=[\beta],\quad \xi=[\beta+v].
\end{equation}
The critical scheme is finite flat of degree eight over
$\operatorname{Spec}R$. The isomorphism respects the markings
\begin{equation}\label{aux:mirrorMarkings}
 [q\partial_qW]=H,\qquad [r\partial_rW]=\xi,
 \qquad [W]=2H+2\xi.
\end{equation}
The isomorphism holds scheme-theoretically on the nonsemisimple divisor,
including the nonreduced fibers.
\end{theorem}
\begin{proof}
Use the invertible coordinate change, with $Q=qv$,
\[
 a=\beta,\qquad b=Q\alpha/\gamma,\qquad c=Q/\gamma,
 \qquad D=ab-c.
\]
It identifies $Y$ with $\{cD v\ne0\}$ and gives
\begin{equation}\label{aux:canonicalcoupled}
 W=a+\frac{b^2}{D}+\frac{Qa}{c}+v+\frac rv,
 \qquad\Omega=\frac{da\wedge db\wedge dc\wedge dv}{cD v}.
\end{equation}
The three vertical derivatives are
\[
 W_a=1-\frac{b^3}{D^2}+\frac Qc,\qquad
 W_b=\frac{b(ab-2c)}{D^2},\qquad
 W_c=\frac{b^2}{D^2}-\frac{Qa}{c^2}.
\]
The ideals generated by $b$ and by $ab-2c$ are comaximal because $c$
is invertible. Hence the relative critical algebra splits into two
factors before imposing the $v$ equation. In the first factor, $b=0$
forces $a=0$ and $c=-Q$. In the second, $ab=2c$, and the derivative
equations force
\[
 c=D=Q,\qquad b=a^2/2,\qquad a^3=4Q.
\]
The $v$ equation in the two factors is respectively $v^2=r$ and
$v^2+av-r=0$. Consequently the full critical algebra is
\begin{equation}\label{aux:criticalsplit}
 R[v^{\pm1}]/(v^2-r)\ \times\
 R[a^{\pm1},v^{\pm1}]/(a^3-4qv,\ v^2+av-r).
\end{equation}
On the quantum side set $v=\xi-H$. The relation $v(v+H)=r$ makes $v$
a unit, and the other relation factors as $H(H^3-4qv)=0$. These two
factors are comaximal since $4qv$ is a unit. The resulting product is
exactly \eqref{aux:criticalsplit}. This proves \eqref{aux:Jacobiisoeq}, including
nonreduced fibers, and finite flatness follows from the eight-element
basis of Theorem~\ref{ring}.

At a relative critical point the canonical formula gives
$[Q\partial_QW^{\rm quad}]=[Qa/c]=a$ in both factors. Thus
$[r\partial_rW]=r/v=a+v$, and
$[W]=3a+v+r/v=4a+2v$. Lemma~\ref{aux:couplinglemma} ensures that these
computations remain valid in the original coordinates. They are exactly
\eqref{aux:mirrorMarkings}.
\end{proof}

\begin{corollary}[The two points outside the Laurent chart]\label{aux:addedpoints}
The two added critical points and their values are
\begin{equation}\label{aux:addedpointsformula}
 p_s=(0,0,-1,s),\qquad s^2=r,\qquad W(p_s)=2s.
\end{equation}
The remaining critical points, with multiplicities at a collision, are
\begin{equation}\label{aux:toruspointsformula}
 p_h=(2/h,h,1,h^3/(4q)),\qquad
 h^6+4qh^4-16q^2r=0,
 \qquad W(p_h)=4h+\frac{h^3}{2q}.
\end{equation}
In particular, the finite Jacobi algebra of the Laurent chart has rank
six, while that of $Y$ has rank eight.
\end{corollary}
\begin{proof}
Translate the two factors of \eqref{aux:criticalsplit} to the original
coordinates. On the second factor $h=a=\beta$ is a unit, so that all
six points are in the Laurent chart. On the first factor $\beta=0$,
$\alpha=0$, and $\gamma=-1$.
\end{proof}

\begin{remark}[Jacobi and Gauss--Manin ranks]
The rank statement concerns Jacobi algebras. The Gauss--Manin system of
the Laurent chart can also receive contributions from critical behavior
at infinity, so its rank need not be six. For the quadric base, the
Laurent potential has three finite critical points, yet restriction from
the completed potential induces an isomorphism of Gauss--Manin systems
of rank four \cite[Lemmas 4.1.2--4.1.3]{GS}. Counting finite critical points
therefore does not suffice to determine the Gauss--Manin rank for the
fourfold.
\end{remark}

\subsection{Residue pairing and the scalar period}
In coordinates $x_1,\ldots,x_4$ with
$\Omega=\rho\,dx_1\wedge\cdots\wedge dx_4$, define the residue trace by
\[
 \ell_B(f)=\sum_{p\in\operatorname{Crit}(W)}
 \operatorname{Res}_p
 \frac{f\rho^2\,dx_1\wedge\cdots\wedge dx_4}
 {W_{x_1}\cdots W_{x_4}}.
\]
At a Morse point this is $f(p)\rho(p)^2/\det\operatorname{Hess}_p(W)$.
The square of $\rho$ makes the expression coordinate independent, as
required by the pairing on the Jacobi classes $f\Omega$.

\begin{theorem}[Comparison of Frobenius pairings]\label{aux:residueiso}
The marked isomorphism \eqref{aux:Jacobiisoeq} identifies the pairing
$\ell_B(fg)$ with the Poincar\'e Frobenius pairing of small quantum
cohomology. It also identifies the trace at the nonsemisimple divisor,
where local Grothendieck residues replace the formula at a Morse point.
At the two added Morse points,
\begin{equation}\label{aux:addedweight}
 \frac{\rho(p_s)^2}{\det\operatorname{Hess}_{p_s}W}
 =-\frac1{4qr}.
\end{equation}
\end{theorem}
\begin{proof}
The ordinary Hessian at $p_s$, in the order
$(\alpha,\beta,\gamma,v)$, is
\[
 \begin{pmatrix}
 2qs&0&0&0\\0&0&1&0\\0&1&0&0\\0&0&0&2/s
 \end{pmatrix}.
\]
Its determinant is $-4q$, and $\rho(p_s)=1/s$. This proves
\eqref{aux:addedweight} and shows that these points are Morse for all
$q,r\ne0$. At $p_h$ direct differentiation gives
\[
 \det\operatorname{Hess}_{p_h}W
 =\frac{4q(3h^2+8q)}{h^2},\qquad
 \rho(p_h)=\frac{4q}{h^3},
\]
so its residue weight is $4q/(h^4(3h^2+8q))$.

To identify the pairing, it remains to compare the traces on a basis.
Put $P(h)=h^6+4qh^4-16q^2r$. The weight on the six-point factor is
$8q/(hP'(h))$, and
\[
 h^{-1}\equiv\frac{h^5+4qh^3}{16q^2r}\pmod{P(h)}.
\]
For any polynomial $R$ the elementary partial-fraction identity
\[
 \sum_{P(h)=0}\frac{R(h)}{P'(h)}
 =[h^5]\bigl(R(h)\bmod P(h)\bigr)
\]
holds when $P$ is square-free. Apply it to
$8qh^{-1}h^i(h+h^3/(4q))^j$, for $0\le i\le3$ and $0\le j\le1$,
and include the two added points. Polynomial division gives
\begin{equation}\label{aux:residuetraces}
 \bigl(\ell_B(1),\ell_B(H),\ell_B(H^2),\ell_B(H^3),
 \ell_B(\xi),\ell_B(H\xi),\ell_B(H^2\xi),\ell_B(H^3\xi)\bigr)
 =(0,0,0,0,0,0,0,2).
\end{equation}
These are exactly the quantum traces of the eight quantum monomials.
Indeed every monomial other than the last has degree less than four;
positive curve corrections lower the cohomological degree and cannot
contribute to integration. The last has leading classical term
$H^3\xi$ and lower-degree corrections, so its integral is $2$.
Trace equality on the algebra basis proves equality of the Frobenius
pairings on all products.

Finally the critical algebra is finite flat and supported inside $Y$
for $q,r\ne0$. Local residues vary holomorphically as isolated critical
points coalesce, using a union of fixed small residue contours. Thus the
trace identity extends to the caustic. At a non-Morse point, the trace is defined by the local residue rather
than by division by the vanishing Hessian determinant.
\end{proof}

\paragraph{The scalar period.}
\begin{proposition}\label{aux:constanttermperiod}
For the potential \eqref{aux:Laurentpotential},
\begin{equation}\label{aux:CTperiod}
 \operatorname{CT}_{x,y,k,v}e^{tL_{q,r}}
 =\sum_{e\ge d\ge0}
 \frac{(2d)!}{(d!)^5e!(e-d)!}\,q^dr^e t^{2(d+e)}
 =G_{q,r}(t).
\end{equation}
This is an identity of formal power series, and hence of their entire
functions wherever the convergent exponential integral is used.
\end{proposition}
\begin{proof}
Choose $d$ factors of $qv(k+1)^2/(xyk)$. Cancellation of $x$ and $y$
requires $d$ factors of each of $x,y$. If there are $e$ factors of $r/v$,
there must be $e-d\ge0$ factors of $v$. The coefficient of $k^0$ in
$(k+1)^{2d}k^{-d}$ is $\binom{2d}{d}$. The total number of chosen factors
is $2(d+e)$. The multinomial coefficient, divided by the factorial in
the exponential, is therefore the coefficient in \eqref{aux:CTperiod}.
There are no odd total degrees. Absolute convergence on a fixed compact
real torus permits termwise integration for the analytic version.
\end{proof}

\subsection{The exponential de Rham transform}\label{app:exponentialtransform}\label{aux:sec:transform}

Integral reconstruction of quantum differential solutions for projective
bundles is established in \cite{CottiCyclic,Cotti,CottiML,IK}. Here we give
an exponential transform with a reciprocal term that matches the coupled
potential and the system in the fixed frame, and we verify its rank and
scalar generator. The theorem identifies a \emph{complex meromorphic
differential module}; the choice of integral rapid-decay cycles and a
Gamma frame requires additional work.

Let $\mathcal E$ be the rank-four cyclic quadric module, with generator
$e$ and relation
\begin{equation}\label{aux:basemaster}
 \bigl(D_Q^4-2Q(2D_Q+z)\bigr)e=0,
 \qquad D_Q=zQ\partial_Q.
\end{equation}
It is a cyclic presentation of the known quadric quantum system; the
completed base model realizes that system \cite{GS,PR}. Pull it back by
$Q=qv$, tensor with the exponential module whose connection on its
unit generator is $z\nabla=d(v+r/v)$, and take de Rham cohomology in $v$.
For clarity, work over $K=\C(q,r,z)$, with $q,r,z$ invertible, and set
\begin{equation}\label{aux:transformdef}
 \mathcal T=H^1_{\rm dR}\left(\C^*_v,
 \left.\mathcal E\right|_{Q=qv}\otimes
 \mathcal E^{(v+r/v)/z}\right).
\end{equation}
All cohomology and parameter connections in this appendix are algebraic.

\begin{lemma}[Rank doubling for a reciprocal exponential twist]\label{aux:rankdoubling}
Let $\mathcal F$ be a rank-$N$ connection on $\C^*$ whose formal slopes
at both $0$ and $\infty$ are strictly less than one. For $r,z\ne0$,
its tensor product with $\mathcal E^{(v+r/v)/z}$ has vanishing $H^0_{\rm dR}$
and $\dim H^1_{\rm dR}=2N$.
\end{lemma}
\begin{proof}
The added exponential has slope one at both ends. Since these slopes
strictly exceed those of $\mathcal F$, no leading exponential can cancel:
all $N$ twisted slopes at each end are one. Thus the two irregularities
are $N,N$, and neither end has a formal horizontal vector. The same is
true for the dual, so the middle extension has $H^0=H^2=0$.
The Euler--Poincar\'e formula for a connection on the projective line
\cite[Proposition 3.1]{Arinkin} gives
\[
 \chi=2N-\bigl((N+N)+(N+N)\bigr)=-2N.
\]
There is no regular formal summand at either puncture; therefore middle
extension and meromorphic extension have the same de Rham cohomology.
Equivalently one may use directly
$\chi_{\rm dR}(\C^*,\mathcal F\otimes\mathcal E^{(v+r/v)/z})
=-\operatorname{irr}_0-\operatorname{irr}_\infty$.
On the affine curve the de Rham complex has no cohomology above degree
one. These observations prove the assertion.
\end{proof}

\begin{theorem}[Cyclic transform and the two-parameter connection]\label{aux:transformtheorem}
The module $\mathcal T$ in \eqref{aux:transformdef} has rank eight over $K$.
Writing $c_k=[v^k e\,dv/v]$, it is generated under the $q,r$ parameter
derivations by $c_0$, and its scalar relations are precisely the
rank-eight system \textup{(\ref{Pq})}--\textup{(\ref{Pr})}. In particular it is isomorphic,
over $K$, to the scalar differential module in
Proposition~\ref{pfaffian}. With the homogeneous normalization
$f=z^{-2}c_0$, its full $q,r,z$ system, followed by the gauge
\textup{(\ref{gauge})}, is the dual quantum connection in the fixed frame.
\end{theorem}
\begin{proof}
At $v=0$ the pullback of \eqref{aux:basemaster} is regular singular.
At infinity its scalar Newton polygon has a horizontal edge of length
one and an edge of slope $1/3$ of length three. In fact, when written in the Euler operator, the nonzero coefficients of
$D_v^4-4qvD_v-2qzv$ have valuations $(0,-1),(1,-1),(4,0)$ at infinity.
Its slopes are $0,1/3,1/3,1/3$. Lemma~\ref{aux:rankdoubling} with $N=4$
therefore proves that the rank is eight. The calculation uses the
irregularities at both ends and does not depend on counting finite
critical points of the Laurent polynomial.

In the de Rham quotient, the derivative of $v^k e$ gives
\begin{equation}\label{aux:momentrecursion}
 c_{k+1}=(D_r-D_q-kz)c_k,\qquad D_r c_k=r c_{k-1}
 \quad(k\in\Z).
\end{equation}
Indeed the $v$ connection acts as $D_Q+v-r/v$ on $e$, and the derivative
of $v^k$ contributes $kz$. The $r$ connection on the exponential factor
acts as $r/v$. This proves both formulas and justifies integration by
parts algebraically without specifying a contour. Since the pulled-back
module is generated as a module over Laurent polynomials by
$e,D_Qe,D_Q^2e,D_Q^3e$, all its one-form classes are generated by the
$D_q^j c_k$. Equation \eqref{aux:momentrecursion} generates every $c_k$ from
$c_0$, using $r^{-1}D_r$ for negative powers. Thus $c_0$ is cyclic.

The base equation now yields
\[
 D_q^4c_0=2q(2D_q+z)c_1
 =2q(2D_q+z)(D_r-D_q)c_0.
\]
Also $D_r(D_r-D_q)c_0=D_r c_1=rc_0$.
Consequently there is a surjection from the rank-eight cyclic module
of Proposition~\ref{pfaffian} onto $\mathcal T$. Equal ranks make it an
isomorphism over $K$. This proves the $q,r$ assertion at the level of differential modules,
without an independence assumption on a collection of integrals.

For the $z$ assertion use the raw homogeneous base normalization
$Z e=-3D_Qe/z$, where $Z=z\partial_z$ and $[Z,D_Q]=D_Q$.
It is compatible with \eqref{aux:basemaster}. The exponential twist gives
\[
 Zc_0=-\frac1z(3D_qc_0+c_1+rc_{-1})
       =-\frac2z(D_q+D_r)c_0.
\]
Hence $f=z^{-2}c_0$ satisfies
$Zf=-2f-2(D_q+D_r)f/z$.
For the eight derivative jets $F$ of $f$ this becomes
\[
 ZF=\mu F-\frac2z(A_q+A_r)F.
\]
The gauge $T$ satisfies the exact homogeneity identity
\[
 \mu T-T\mu=2(q\partial_qT+r\partial_rT).
\]
Combining this with the cancellation of the $z$-linear entries in
Section~\ref{sec:connection} gives, for $Y=T^{-1}F$,
\[
 D_qY=M_H^{\mathsf T}Y,\quad
 D_rY=M_\xi^{\mathsf T}Y,\quad
 ZY=\mu Y-\frac{2(M_H+M_\xi)^{\mathsf T}}zY.
\]
These are exactly the equations for the dual of \textup{(\ref{qconnection})}.
\end{proof}

\begin{remark}[The remaining integral mirror comparison]\label{aux:complexboundary}
An oscillatory notation for this transform is
$\int f_B(qv,z)e^{(v+r/v)/z}\,dv/v$.
The proof constructs the algebraic de Rham module and determines its
rank. Contour admissibility, an integral basis, and Gamma normalization
are separate requirements. It also leaves open the comparison of the
full derived direct image of the fourfold potential with the iterated
module in top degree. A complete comparison must include the fourfold
Brieskorn lattice, its higher residue pairing, and its integral
rapid-decay homology. Theorem~\ref{aux:residueiso} identifies the $z=0$
Frobenius algebra, while Theorem~\ref{aux:transformtheorem} identifies a
complex meromorphic connection. The integral Stokes comparison remains
to be established.
\end{remark}

\subsection{Spectral collisions and the local critical germ}\label{app:caustics}\label{aux:sec:caustic}
\paragraph{The collision divisors.}
\begin{proposition}\label{aux:complexloci}
On $(\C^*)^2$, the quantum algebra is semisimple if and only if
$16q-27r\ne0$. The anticanonical operator has eight distinct eigenvalues
if and only if
\[
 (16q-27r)(16q+r)(64q+27r)\ne0.
\]
On the two additional divisors, eigenvalues collide while the quantum
algebra remains semisimple.
\end{proposition}
\begin{proof}
On the six-point branch, repeated roots of
$h^6+4qh^4-16q^2r$ satisfy
$2h^3(3h^2+8q)=0$. Since $q,r\ne0$, this forces
$h^2=-8q/3$ and $r=16q/27$. The two $h=0$ points of the quantum ring are
reduced: the Jacobian of the two defining relations is nonzero there.
They are distinct from the six-point branch.

For anticanonical eigenvalues, the resultant of the polynomial
\textup{(\ref{charpoly})} and its derivative gives
\begin{equation}\label{aux:fulldiscriminant}
 \disc_u\det(u-C)
 =2^{80}q^8r^8(16q-27r)^6(16q+r)^2(64q+27r)^4.
\end{equation}
This identity can also be obtained by factoring into the quadratic and
sextic terms in \textup{(\ref{charpoly})} and using
$\disc(fg)=\disc(f)\disc(g)\operatorname{Res}(f,g)^2$.
For example, the two additional factorizations are
\[
 \left.\det(u-C)\right|_{r=-16q}
 =u^2(u^2+64q)(u^4+208qu^2+32768q^2),
\]
\[
 \left.\det(u-C)\right|_{r=-64q/27}
 =\frac{(27u^2+256q)^2
 (729u^4+25488qu^2+2166784q^2)}{531441}.
\]
The already established algebra criterion shows that these collisions
occur in semisimple algebras. Consequently $C$ is diagonalizable there.
\end{proof}

\paragraph{The ramified splitting.}
Fix $q\ne0$ and put
\[
 \delta=\frac rq-\frac{16}{27}.
\]
Choose one of the two square roots $h_0$ of $-8q/3$. It labels one of the
two colliding pairs.

\begin{theorem}[Local spectral expansion]\label{aux:puiseux}
On a chosen square-root cover of a neighborhood of $\delta=0$, the two
anticanonical eigenvalues near $u_0=8h_0/3$ are
\begin{equation}\label{aux:uexpansion}
 u_\pm=h_0\left(\frac83+\frac9{16}\delta
       \ \pm\frac{27\sqrt{-1}}{64}\delta^{3/2}
       +O(\delta^2)\right).
\end{equation}
In particular,
\begin{equation}\label{aux:differenceexp}
 u_+-u_-=
 \frac{27\sqrt{-1}\,h_0}{32}\delta^{3/2}
 +O(\delta^{5/2}).
\end{equation}
A small loop around $\delta=0$ induces a transposition in each colliding
pair and fixes the remaining four points of the spectral cover.
\end{theorem}
\begin{proof}
Write $s=h^2/q=-8/3+w$. The cubic equation for $s$ becomes exactly
\[
 w^3-4w^2-16\delta=0.
\]
Set $\delta=t^2$. Solving recursively at $w=0$ gives
\[
 w_\pm=\pm2\sqrt{-1}\,t-\frac12t^2
              \mp\frac{5\sqrt{-1}}{16}t^3+O(t^4).
\]
The corresponding $h$ and $u$ are
\[
 h=h_0\sqrt{1-\tfrac38w},\qquad
 u=h_0\sqrt{1-\tfrac38w}\left(\frac83+\frac12w\right).
\]
The coefficient of $w$ in the second expression vanishes. Expanding it
through order $t^3$ gives \eqref{aux:uexpansion}. Since the two branches are
interchanged by $t\mapsto-t$, their difference is odd in $t$; the first
nonzero odd term is of degree three, which gives the sharper remainder
in \eqref{aux:differenceexp}. The other roots are simple at $\delta=0$ and
therefore have trivial local permutation. Repeating the argument at
$-h_0$ gives the second transposition.
\end{proof}

\begin{remark}[Spectral scaling near the caustic]
The relative exponential is
$\exp(-(u_+-u_-)/z)$. Formula \eqref{aux:differenceexp} identifies the natural
balance $\delta=O(z^{2/3})$ near the caustic. This spectral scaling alone does not construct a uniform Airy
parametrix. The permutation of eigenvalues also leaves the monodromy of
a normalized flat basis and its Stokes multipliers undetermined. We
establish the local $A_2$ germ separately in Theorem~\ref{aux:A2germ} below.
Applying the global analytic coalescence theory of \cite{CDGCoal} requires
additional hypotheses beyond the exponent $3/2$.
\end{remark}

\paragraph{The transverse critical germ.}
The completed potential gives a geometric local model for the exponent
in Theorem~\ref{aux:puiseux}.

\begin{theorem}[Two transverse $A_2$ singularities]\label{aux:A2germ}
Fix $q\ne0$ and put $\delta=r/q-16/27$. At $\delta=0$ the potential
\eqref{aux:completedpotential} has exactly two non-Morse critical points,
labelled by $h_0^2=-8q/3$. Each has Hessian corank one and is an $A_2$
critical germ with a transverse one-parameter unfolding. More precisely,
a holomorphic splitting off of three nondegenerate quadratic variables
leaves a one-variable function $g(\eta,\delta)$ with
\begin{equation}\label{aux:A2derivatives}
 g_\eta(0,0)=g_{\eta\eta}(0,0)=0,\qquad
 g_{\eta\eta\eta}(0,0)=\frac3{h_0^2},\qquad
 g_{\eta\delta}(0,0)=\frac{27}{32}.
\end{equation}
After further holomorphic changes of coordinate and subtraction of a
function of the parameter alone, this is the usual $A_2$ family
$\eta^3/3-\tau(\delta)\eta$, with $\tau'(0)\ne0$.
\end{theorem}
\begin{proof}
At the point $p_{h_0}$ write $h=h_0$. Then
$q=-3h^2/8$, $r=-2h^2/9$, $v=-2h/3$.
In the coordinates $(\alpha,\beta,\gamma,v)$ the Hessian is
\[
 \begin{pmatrix}
 h^3/2&h&0&0\\
 h&8/h&3&3/h\\
 0&3&2h&3/2\\
 0&3/h&3/2&3/(2h)
 \end{pmatrix}.
\]
Its upper left $3\times3$ minor is $3h^3/2\ne0$, and its kernel is
spanned by $n=(h^{-2},-1/2,0,1)^{\mathsf T}$. Thus its rank is three.
Choose the kernel coordinate $\eta$ with derivative vector $n$, and
eliminate three complementary critical equations by the holomorphic
implicit function theorem. The holomorphic Morse splitting lemma
separates their nondegenerate quadratic part. At the origin the cubic
coefficient in the reduced function is the third directional derivative
along $n$, because the Hessian kills $n$. Direct differentiation gives
\[
 D_n^3W\big|_{p_h}=3/h^2.
\]
At fixed $q$, $\partial_\delta W=q/v$. The mixed derivative of the
reduced function is therefore
\[
 D_n(q/v)\big|_{p_h}=27/32;
\]
terms from the eliminated variables vanish after pairing with the
vector in the kernel of the Hessian. These calculations prove \eqref{aux:A2derivatives}.
A one-variable germ with a nonzero cubic term and a nonzero coefficient
of $\delta\eta$ is the miniversal $A_2$ unfolding, up to a
term depending only on the parameter and holomorphic changes of
coordinates. This follows also by
Weierstrass preparation of $g_\eta$ and removal of its linear term.
The other critical points have nonzero Hessian determinant by
Theorem~\ref{aux:residueiso} and the sextic discriminant computation.
\end{proof}

The reduced Taylor expansion can be written without choosing a normal form:
\[
 g(\eta,\delta)=\phi(\delta)
 +\frac{27}{32}\delta\eta+\frac{\eta^3}{2h_0^2}
 +O(\eta^4,\delta\eta^2,\delta^2\eta).
\]
Its two critical points satisfy
\[
 \eta_\pm=\pm\frac{3\sqrt{-1}h_0}{4}\delta^{1/2}+O(\delta).
\]
The difference of their critical values is
\[
 \frac{27\sqrt{-1}h_0}{32}\delta^{3/2}+O(\delta^{5/2}),
\]
in agreement with \eqref{aux:differenceexp}. This establishes the local singularity of Airy type. A uniform global
Airy approximation, with specified rapid-decay cycles and Stokes
constants, would require further analysis.

\begin{proposition}[The positive nonsemisimple locus]\label{aux:wall}
At $r=16q/27$,
\begin{equation}\label{aux:wallchar}
 \det(u-C)=
 \left(u^2-\frac{784q}{27}\right)
 \left(u^2-\frac{64q}{27}\right)
 \left(u^2+\frac{512q}{27}\right)^2.
\end{equation}
The quantum ring is isomorphic to
\[
 \C^4\times\C[\epsilon]/(\epsilon^2)\times\C[\epsilon]/(\epsilon^2),
\]
although the multiplication operator $C$ is diagonalizable. Hence
nonsemisimplicity of the ring does not force this operator to have a
nontrivial Jordan block.
\end{proposition}
\begin{proof}
The repeated points on \textup{(\ref{hroots})} have $H^2=-8q/3$. In either length-two local algebra, $u=4H+H^3/(2q)$ has derivative
\[
 \frac{du}{dH}=4+\frac{3H^2}{2q}=0.
\]
It is therefore constant modulo the square of the local parameter and acts as a scalar on each dual-number block. The remaining four points are simple. This proves the assertion and \eqref{aux:wallchar}.
\end{proof}
\subsection{Local braids and coalescence}\label{aux:sec:braids}

Fix $q\ne0$ and vary $\kappa=r/q$. Off the three collision values
$16/27,-16,-64/27$, the eight eigenvalues give a map to the configuration
space of eight unordered points of $\mathbb C$. A local meridian gives
an element of its braid group, up to conjugacy after choosing a path
from a global base point. We call a counterclockwise exchange of two
points a positive half-twist. The topological sign is fixed independently
of the convention for its action on exceptional objects.

\begin{theorem}[The three local braids of critical values]
\label{aux:localbraids}
The local braids of the anticanonical critical values are as follows.
\begin{center}
\renewcommand{\arraystretch}{1.12}
\small
\begin{tabular}{c|c|c|c}
\toprule
$\kappa$ & type & local braid in separated pairs & discriminant order\\
\midrule
$16/27$ & two $A_2$ critical germs & two commuting cubes & $6$\\
$-16$ & one pair of coincident values & one square & $2$\\
$-64/27$ & two pairs of coincident values & two commuting squares & $4$\\
\bottomrule
\end{tabular}
\end{center}
Here ``cube'' and ``square'' mean the third and second powers of the
half-twist in the corresponding two-point disc. Thus the first meridian is conjugate to a product of cubes of disjoint
half-twists; its permutation records only the two transpositions. The other two meridians are pure braids:
their critical points stay distinct.
\end{theorem}
\begin{proof}
For two unordered values $u_1,u_2$, the squared difference
$d=(u_1-u_2)^2$ is a local coordinate on the quotient by interchange.
The winding number of $d$ around zero is the exponent of the half-twist.
At $\kappa=16/27$, Theorem~\ref{aux:puiseux} gives
$d=c\delta^3(1+O(\delta))$, $c\ne0$, in each of the two disjoint clusters.
This proves the two cubes. The other four values stay in disjoint discs,
so no additional local braid is present.

At $\kappa=-16$, the two colliding values are obtained from $h$ and
$-h$, where $h^2=-8q$; their common value is zero. On the sextic branch
regard
\[
 \kappa(h)=\frac{h^6+4qh^4}{16q^3},\qquad
 u(h)=4h+\frac{h^3}{2q}.
\]
At this point $\kappa'(h)\ne0$ and
$du/d\kappa=h/16$. With $\delta=\kappa+16$ the difference of the two
values is
\begin{equation}\label{aux:Maxwellone}
 u_h(\kappa)-u_{-h}(\kappa)=\frac h8\delta+O(\delta^2).
\end{equation}
Its square has order two, proving the single pure half-twist square.

At $\kappa=-64/27$ put $h^2=-16q/3$. The sextic branch at $h$ collides
with the quadratic branch $2\sqrt r$, with $\sqrt r=2h/3$ at the
collision. Both values are $4h/3$. Their respective derivatives are
$9h/64$ and $-9h/32$. For $\delta=\kappa+64/27$,
\begin{equation}\label{aux:Maxwelltwo}
 u_h(\kappa)-u_{\mathrm{quad}}(\kappa)
       =\frac{27h}{64}\delta+O(\delta^2).
\end{equation}
The opposite pair gives the same order. The two distinct common values
lie in disjoint discs, giving two commuting squares. Squaring each
local value difference also yields the orders $6,2,4$ of the full
discriminant, in agreement with \eqref{aux:fulldiscriminant}.
\end{proof}

\begin{remark}[From local braids to categorical labels]
The theorem determines local conjugacy classes of braids. To express
them as global words in a fixed foundation of eight objects, one must
also choose paths from a reference parameter point and signs for the
normalized solutions. At an $A_2$ point, the cube and a single mutation
have the same transposition, but this does not identify their actions.
Relating these local braids to changes of the asymptotic collection with
its Gamma marking remains part of the Gamma II comparison along those
paths.
\end{remark}

\paragraph{Parameter monodromy.}
\begin{proposition}\label{aux:trivialparameterholonomy}
For fixed $z\ne0$, the ordinary local holonomy of the quantum connection
around a sufficiently small meridian of any of the three collision
divisors is the identity.
\end{proposition}
\begin{proof}
The matrices $M_H,M_\xi$ are polynomial in $q,r$. The connection form
$M_H\,dq/(qz)+M_\xi\,dr/(rz)$ is holomorphic on a small simply connected
neighborhood of a point of any collision divisor with $q,r\ne0$.
It is flat and therefore has a holomorphic fundamental matrix on that
neighborhood. A meridian contracts inside it, proving the assertion.
\end{proof}
The meridian is noncontractible in the locus with \emph{distinct eigenvalues},
which is why its spectral braid can be nontrivial. Asymptotic eigenbases and their Stokes comparisons can be multivalued
even though the local fundamental matrix is holomorphic. This distinction also separates
these parameter loops from the $z$-turn in
Proposition~\ref{Gammamonodromy}.

\paragraph{The grading residue at a semisimple collision.}
The two negative collision values leave the quantum algebra semisimple.
Consequently the primitive idempotents and their nonzero square norms
vary holomorphically there, after choosing a local square root for each
normalization. Let $\Psi$ be the resulting orthonormal idempotent frame,
and set
\[
 \mathsf U=\Psi^{-1}C\Psi=\operatorname{diag}(u_i),\quad
 \Xi=\Psi^{-1}M_\xi\Psi=\operatorname{diag}(\xi_i),\quad
 V=\Psi^{-1}\mu\Psi,\quad
 A_r=\Psi^{-1}(r\partial_r\Psi).
\]
Here $\mathsf U$ denotes the diagonal eigenvalue matrix; the line bundle
class is denoted by $U$.

\begin{proposition}[Divisibility of the off-diagonal grading residue]
\label{aux:residuevanishing}
For each pair that coalesces at $\kappa=-16$ or $-64/27$, one has
\begin{equation}\label{aux:Vdivisibility}
 (\xi_i-\xi_j)V_{ij}=(u_j-u_i)(A_r)_{ij},\qquad
 V_{ij}=O(u_i-u_j).
\end{equation}
The same assertion holds with $i,j$ interchanged. In particular these
off-diagonal grading residues vanish on the respective collision divisor.
\end{proposition}
\begin{proof}
Differentiate $C=\Psi\mathsf U\Psi^{-1}$ and use the mixed flatness
identity $r\partial_rC=M_\xi+[M_\xi,\mu]$. Its off-diagonal entries in
the $\Psi$ frame give the first equality in \eqref{aux:Vdivisibility}.
At $\kappa=-16$ the colliding sextic branches have $\xi=-h,h$, so the
difference is nonzero. At $\kappa=-64/27$ a colliding sextic/quadratic
pair has $\xi=-h/3,2h/3$, again distinct. Thus division by
$\xi_i-\xi_j$ is holomorphic in a neighborhood, proving the assertion.
The opposite pair has the same property.
\end{proof}

The vanishing of these residues provides an input for holomorphic
coalescence analysis. Cotti--Dubrovin--Guzzetti \cite{CDGCoal} relate such
conditions to the vanishing of Stokes entries; see also Sabbah's
formulation \cite{SabbahCDG}. To apply their results, stated in independent
eigenvalue coordinates, to this slice in two parameters, one must verify
the required deformation and formal extension hypotheses. The residue
calculation alone therefore gives no conclusion about all Stokes
entries. At the positive $A_2$ divisor, primitive idempotents cease to be
holomorphic, so the proof above does not apply.
Xu studies regularized Stokes limits under degeneration of the
irregular data \cite{XuLimits}; Tang and Xu describe boundary
asymptotics of the isomonodromy equations~\cite{TXBoundary}.

\subsection{Differential coupling at the quadratic branches}\label{app:coupledblock}\label{aux:sec:block}
All cup product expressions in this appendix refer to the fixed frame
\textup{(\ref{flatbasis})}. We work on $(\C^*)^2$, and choose a square root of $r$
only when distinguishing the two one-dimensional branches.

\paragraph{The projector.}
\begin{proposition}\label{aux:proj}
The element
\begin{equation}\label{aux:p0}
 p_0=\frac12-\frac{H^3\xi}{4qr}
\end{equation}
is an idempotent for the quantum product. Its multiplication operator
$\mathcal P=M_{p_0}$ has rank two and equals
\begin{equation}\label{aux:Poperator}
 \mathcal P=I-\frac{M_H^3M_\xi}{4qr}.
\end{equation}
Its image is spanned by $p_0$ and
\begin{equation}\label{aux:p1}
 p_1=\xi\star p_0=\frac{\xi-H}{2}-\frac{H^3}{4q}.
\end{equation}
The multiplication on this subspace satisfies
$H\star p_i=0$, $p_0\star p_1=p_1$, and
$p_1\star p_1=rp_0$.
\end{proposition}
\begin{proof}
In quantum monomials the element is
$1-H^{\star3}\star\xi/(4qr)$. On the six-point branch
$H^{\star3}\star\xi=4qr$, whereas on the $H=0$ branch it is zero.
The same identity holds on the length-two local algebras at the caustic,
or directly by the relations of Theorem~\ref{ring}. Hence it is the
projector onto the two $H=0$ points. The table in
Theorem~\ref{fixedframe} gives
$H^{\star3}\star\xi=H^3\xi+2qr$ in the classical frame, proving
\eqref{aux:p0}. Multiplying it by $\xi$ gives \eqref{aux:p1}; a second multiplication
uses $\xi\star\xi=H\star\xi+r$ and $H\star p_0=0$.
\end{proof}

The Poincar\'e pairing on this subspace is particularly simple:
\begin{equation}\label{aux:blockpairing}
 (p_0,p_0)=-\frac1{2qr},\quad
 (p_1,p_1)=-\frac1{2q},\quad (p_0,p_1)=0.
\end{equation}
Here $(a,b)=\int_Xa\cup b$, equivalently the Frobenius pairing for the
quantum algebra.

\paragraph{Second fundamental forms.}
For $i=q,r$ define the second fundamental form of $\im\mathcal P$ by
\[
 \mathcal B_i(s)=(I-\mathcal P)\nabla_{i\partial_i}s,
 \qquad s\in\im\mathcal P.
\]
Since both divisor multiplication operators preserve this image, their
$z^{-1}$ terms vanish after applying $I-\mathcal P$. Consequently the
forms do not depend on $z$.

\begin{theorem}[Nonhorizontal quadratic block]\label{aux:leakage}
Put
\[
 v_0=\frac14+\frac{H^3\xi}{8qr},\qquad
 v_1=\frac{\xi-H}{4}+\frac{H^3}{8q}.
\]
Then $v_0,v_1\in\ker\mathcal P$ are linearly independent, and
\begin{equation}\label{aux:secondforms}
 \begin{array}{c|cc}
 &p_0&p_1\\\hline
 \mathcal B_q&v_0&v_1\\
 \mathcal B_r&v_0&0
 \end{array}
\end{equation}
throughout $(\C^*)^2$. In particular,
$\rank\mathcal B_q=2$ and $\rank\mathcal B_r=1$.
The rank-two spectral subbundle is therefore not a subconnection.
\end{theorem}
\begin{proof}
The expressions \eqref{aux:p0}--\eqref{aux:p1} give
\[
 q\partial_qp_0=r\partial_rp_0=\frac{H^3\xi}{4qr},\qquad
 q\partial_qp_1=\frac{H^3}{4q},\qquad r\partial_rp_1=0.
\]
Using \eqref{aux:Poperator}, or the multiplication table, gives
\[
 \mathcal P(H^3\xi)=-qr+\tfrac12H^3\xi,\qquad
 \mathcal P(H^3)=qH+\tfrac12H^3-q\xi.
\]
Subtracting these projections yields \eqref{aux:secondforms}. Since
$\mathcal P^2=\mathcal P$, the results lie in its kernel. The component
of $v_0$ in degree zero is nonzero while $v_1$ has only degrees one and
three; $v_1\ne0$. Their independence proves the ranks.
\end{proof}

\begin{remark}[Scope of the obstruction]
The quadratic factor $u^2-4r$ does not define an independently varying
quantum subsystem of rank two in this spectral frame. The theorem is
compatible with a $z$-dependent formal normal form, sectorial
diagonalization, and the projective bundle decompositions in the
parameter regimes considered in \cite{IK,SS,HYYZ}.
\end{remark}

\paragraph{Formal correction terms.}
Let $s$ be either square root of $r$. Define the primitive idempotent
\begin{equation}\label{aux:es}
 e_s=\frac12(p_0+s^{-1}p_1)
 =\frac14+\frac{\xi-H}{4s}-\frac{H^3}{8qs}
                         -\frac{H^3\xi}{8qr}.
\end{equation}
It has eigenvalue $u_s=2s$ for $C$ and square norm
$(e_s,e_s)=-1/(4qr)$. On a simply connected neighborhood of positive parameters, a normalized
eigenvector is
$\psi_s=2\sqrt{-1}\sqrt{qr}\,e_s$.

\begin{lemma}\label{aux:gradingblock}
The grading operator has no component inside this block:
\[
 \mathcal P\mu\mathcal P=0.
\]
More explicitly, $\mu p_0=-4v_0$ and $\mu p_1=-2v_1$.
\end{lemma}
\begin{proof}
Apply the diagonal grading \textup{(\ref{grading})} to
\eqref{aux:p0}--\eqref{aux:p1}, and use Theorem~\ref{aux:leakage}.
\end{proof}

\begin{proposition}[Two formal correction terms]\label{aux:formalquadratic}
Fix a positive point away from $r/q=16/27$, and normalize the leading
term to $e_s$. The canonical formal solution of the $z$-equation at
$u_s=2s$ begins
\begin{equation}\label{aux:formalexpansion}
 e^{-2s/z}\bigl(e_s+z a_s+z^2 b_s+O(z^3)\bigr),
\end{equation}
where
\begin{equation}\label{aux:abformal}
 a_s=-\frac{H^2}{16qs}-\frac{H^2\xi}{16qr},\qquad
 b_s=-\frac{H}{32qs}-\frac{H\xi}{32qr}.
\end{equation}
Multiplying by $2\sqrt{-1}\sqrt{qr}$ gives the expansion with normalized
leading eigenvector $\psi_s$.
\end{proposition}
\begin{proof}
For $v(z)=\sum_{k\geq0}v_kz^k$, substitution of
$e^{-u_s/z}v(z)$ in \textup{(\ref{qconnection})} gives
\begin{equation}\label{aux:formalrecursion}
 (C-u_s)v_{k+1}=(\mu+k)v_k.
\end{equation}
The multiplication table gives
\[
 (C-2s)a_s=\mu e_s,\qquad
 (C-2s)b_s=(\mu+1)a_s.
\]
Moreover $\mathcal Pa_s=\mathcal Pb_s=0$ and
$\mathcal P\mu a_s=\mathcal P\mu b_s=0$.
The last equalities show that the next solvability conditions in the
$e_s$ direction hold without adding a multiple of $e_s$ to either of
these two coefficients. Indeed the Frobenius pairing makes $C$
self-adjoint and $\mu$ skew-adjoint. At each positive order, projection
of \eqref{aux:formalrecursion} to the $e_s$ line fixes the remaining scalar
multiple through the nonzero integer $k$. At order zero the condition
is $(e_s,\mu e_s)=0$, automatic by skew-adjointness. Since the eigenvalue
is simple, this is the usual unique formal recursion with fixed leading
term. It proves \eqref{aux:formalexpansion}.
\end{proof}

\begin{remark}[Formal coefficients and the Gamma marking]
Determining the Gamma objects at the regular end requires the analytic
argument of Section~\ref{sec:spinorMB}: the exact Mellin--Bessel identities
identify the two spinor markings, and the sector estimates determine
their asymptotic columns. The sign used there at $s=-\sqrt r$ is the
opposite of the convention above for $\psi_s$. The formal coefficients
alone determine neither these Gamma objects nor an integral thimble
identification.
\end{remark}

\section{Period asymptotics and integral structures}
\label{app:periodhelix}
This appendix studies the quantitative asymptotics of the normalized
period and the integral constraints imposed by the Euler form and the
helix. The full-turn identities give necessary compatibility relations;
identifying the canonical columns still requires the analysis of
Section~\ref{sec:transportStokes}.

\subsection{Quantitative Gamma I asymptotics}\label{app:periodasymptotics}
\paragraph{The positive sector.}
The normalized period gives a quantitative complement to the proof by
a positive integral in Section~\ref{sec:markedjump}. Besides the scalar
concentration studied in \cite{LKH}, the argument must control the
cohomology-valued direction and the terms of negative relative degree. Set $z=1$ and follow the anticanonical ray $q=r=t^2$. Write
\[
 J_X(t):=J(t^2,t^2,1)=t^{2H+2\xi}\sum_{d,e\geq0}t^{2(d+e)}A_{d,e}.
\]
Let $G(t)$ be its coefficient of $1$ in classical cohomology. Terms with $e<d$ contain $\xi-H$ and have zero scalar part, so
\begin{equation}\label{aux:period}
 G(t)=\sum_{e\geq d\geq0}
 \frac{(2d)!}{(d!)^5e!(e-d)!}\,t^{2(d+e)}.
\end{equation}

\begin{lemma}[Unique saddle]\label{aux:saddle}
For general $q,r>0$, the scalar series with weights $q^d r^e t^{2(d+e)}$ has a unique interior saddle $(A,B)$, characterized by
\begin{equation}\label{aux:saddleEq}
 A^3=4q(B-A),\qquad B(B-A)=r,\qquad B>A>0.
\end{equation}
Its exponential rate is $2(A+B)=T(q,r)$, and
\begin{equation}\label{aux:Gasym}
 G_{q,r}(t)\sim C(q,r)t^{-2}e^{T(q,r)t},\qquad C(q,r)>0.
\end{equation}
\end{lemma}
\begin{proof}
For $d/t\to A$ and $e/t\to B$, Stirling's formula gives the entropy
\[
 \Phi(A,B)=A\log(4q)+B\log r-3A\log A-B\log B
 -(B-A)\log(B-A)+2(A+B).
\]
Its gradient vanishes exactly at \eqref{aux:saddleEq}. On $B>A>0$, the negative Hessian is
\[
 K=\begin{pmatrix}
 3/A+1/(B-A)&-1/(B-A)\\
 -1/(B-A)&1/B+1/(B-A)
 \end{pmatrix},
\]
which is positive definite. The entropy tends to minus infinity at infinity, and its directional
derivatives at the boundary force the maximum to lie in the interior.
Strict concavity therefore gives a unique critical point. Eliminating $B$ yields $A^6+4qA^4=16q^2r$, and hence the spectral rate \textup{(\ref{T})}. Substitution into the entropy gives $\Phi(A,B)=2(A+B)$.

Stirling's formula gives the following prefactor near the saddle:
\[
 \frac{\sqrt2}{(2\pi)^3}t^{-3}A^{-2}[B(B-A)]^{-1/2}(1+O(t^{-1})).
\]
A two-dimensional lattice Laplace expansion multiplies this by $2\pi t/\sqrt{\det K}$. Thus \eqref{aux:Gasym} holds, with
\[
 C(q,r)=\frac{\sqrt2}{(2\pi)^2 A^2\sqrt r\sqrt{\det K}}.
\]
To justify the localization, split the lattice into a compact
neighborhood of $(A,B)$, its complement in a large compact triangle, and
the unbounded tail. Strict concavity gives an exponential gap on the
second region. Stirling bounds and the term $-2(A+B)\log(A+B)$ control
the tail. The entropy extends continuously to the boundary, where its
maximum is strictly smaller. These bounds also allow fixed polynomial
or logarithmic weights in the summand.
\end{proof}

\paragraph{The Gamma direction.}
For $e\geq d$, the coefficient has the exact expression
\begin{equation}\label{aux:gammaratio}
 A_{d,e}=\Ga_X
 \frac{\Gamma(1+2H+2d)}{\Gamma(1+H+d)^5
 \Gamma(1+\xi+e)\Gamma(1+\xi-H+e-d)}.
\end{equation}
At $d/t\to A$, $e/t\to B>A$, the ratio to its scalar factorial coefficient is
\[
 \Ga_X\left[\frac{4(e-d)}{d^3}\right]^H
 \left[\frac1{e(e-d)}\right]^\xi(1+O(t^{-1})),
\]
as an equality in the finite-dimensional classical cohomology algebra. At the saddle, \eqref{aux:saddleEq} makes this
\[
 \Ga_X\,t^{-2H-2\xi}q^{-H}r^{-\xi}(1+O(t^{-1})).
\]
The divisor exponential cancels these factors. Applying the preceding Laplace argument with these weights gives
\begin{equation}\label{aux:positiveGamma}
 \frac{q^H r^\xi t^{2H+2\xi}
 \sum_{e\geq d}q^d r^e t^{2(d+e)}A_{d,e}}
 {G_{q,r}(t)}\longrightarrow\Ga_X.
\end{equation}
The ratio expansions are uniform in an interior saddle neighborhood. Outside it, the finite Taylor coefficients in nilpotent cohomology variables introduce at most polynomial and logarithmic factors, absorbed by the entropy gap. Thus the limit holds in the full cohomology algebra.

\paragraph{Negative relative degrees.}
\begin{lemma}[An exponential bound]\label{aux:negative}
For a fixed norm on $H^*(X,\C)$ and some fixed integer $M$,
\[
 \left\|t^{2H+2\xi}\sum_{0\leq e<d}t^{2(d+e)}A_{d,e}\right\|
 \leq C(1+t)^M e^{6t}\qquad(t\geq1).
\]
\end{lemma}
\begin{proof}
Write $E=\xi-H$ and $n=d-e\geq1$. The extra numerator is
\[
 R_{-n}(E,1)=E\prod_{j=1}^{n-1}(E-j).
\]
In a nilpotent algebra truncated in complex degree four, expansion of all factors in \textup{(\ref{I})} gives
\[
 \|A_{d,e}\|\leq C(1+\log(d+e+2))^M
 \frac{(2d)!(d-e-1)!}{(d!)^5e!}.
\]
Indeed, after factoring out the factorials, the remaining Taylor coefficients are polynomials of bounded degree in harmonic sums. Now
\[
 \frac{(d-e-1)!}{d!}\leq\frac1{e!},\qquad
 \binom{2d}{d}\leq4^d,
\]
so the factorial term is bounded by $4^d/((d!)^2(e!)^2)$. Summing the resulting majorant gives a product of modified Bessel series,
\[
 \sum_{d,e\geq0}\frac{(4t^2)^d(t^2)^e}{(d!)^2(e!)^2}
 =I_0(4t)I_0(2t)\leq e^{6t}.
\]
Polynomial weights in $d,e$ are controlled by fixed Euler derivatives of these series and contribute a polynomial in $t$. The prefactor $t^{2H+2\xi}$ is polynomial in $\log t$. Increasing $M$ proves the claim.
\end{proof}

\begin{theorem}[Gamma I at the unit parameter point]\label{aux:GammaSeed}
For $X$ as above,
\begin{equation}\label{aux:GammaLimit}
 \lim_{t\to+\infty}\frac{J_X(t)}{G(t)}=\Ga_X.
\end{equation}
Together with Theorem \ref{spectral}, this proves the original Gamma conjecture I for $X$ at $q=r=1$.
\end{theorem}
\begin{proof}
At $q=r=1$, the positive root $a$ of \textup{(\ref{a})} satisfies $a>8/5$, since $(8/5)^3+4(8/5)^2<16$. Consequently
\[
 T^2=\frac{a(a+8)^2}{4}>\frac{4608}{125}>36,
\]
so $T>6$. Equations \eqref{aux:Gasym} and \eqref{aux:positiveGamma} identify the limit of the $e\geq d$ sector. Lemma \ref{aux:negative} makes the $e<d$ sector exponentially negligible after dividing by $G(t)$. This proves \eqref{aux:GammaLimit}. The limit identifies the principal asymptotic class by
\cite[Corollary 3.6.9]{GGI}, whose Property $\mathcal O$ hypothesis was
verified above.
\end{proof}

\paragraph{A uniform parameter region.}
For $q,r>0$ define
\[
 J_{q,r}(t)=J(qt^2,rt^2,1),\qquad
 R(q,r)=4\sqrt q+2\sqrt r.
\]
The norm estimate in Lemma~\ref{aux:negative}, with the weights restored, gives
\begin{equation}\label{aux:weightednegative}
 \left\|q^Hr^\xi t^{2H+2\xi}
   \sum_{e<d}q^dr^et^{2(d+e)}A_{d,e}\right\|
 \leq C_K(1+t)^M e^{R(q,r)t}
\end{equation}
on compact subsets $K$ of the positive quadrant. Indeed, the scalar
majorant is $I_0(4\sqrt q\,t)I_0(2\sqrt r\,t)$, and the logarithmic
and derivative factors are uniformly controlled on $K$.

\begin{theorem}[Uniform normalized Gamma limit]\label{aux:uniformgamma}
Let
\[
 \mathcal W=\{(q,r)\in\mathbb R_{>0}^2:T(q,r)>R(q,r)\}.
\]
For every compact $K\subset\mathcal W$ there are constants
$c_K,C_K>0$ and an integer $M$ such that, uniformly on $K$ as $t\to\infty$,
\begin{align}
 G_{q,r}(t)&=C(q,r)t^{-2}e^{T(q,r)t}(1+O_K(t^{-1})),\label{aux:uniformperiod}\\
 \left\|\frac{J_{q,r}(t)}{G_{q,r}(t)}-\Ga_X\right\|
 &\leq \frac{C_K}{t}+C_Kt^{M+2}e^{-c_Kt}.
 \label{aux:uniformbound}
\end{align}
In particular the error is $O_K(t^{-1})$. The set $\mathcal W$ is open,
is invariant under simultaneous positive rescaling of $q,r$, and contains
$(1,1)$.
\end{theorem}
\begin{proof}
The saddle $(A,B)$ of Lemma~\ref{aux:saddle} depends smoothly on $q,r$.
On a compact positive parameter set it stays in a compact subset of
$0<A<B$; the Hessian is uniformly positive and the prefactor in
\eqref{aux:Gasym} is bounded away from zero. The tail estimates of that lemma
are consequently uniform.

To estimate the error in the localized sum, set
$y_1=(d-At)/\sqrt t$, $y_2=(e-Bt)/\sqrt t$ and restrict first to
$|y|\leq t^{1/10}$. Taylor expansion of the entropy and the Stirling
prefactor gives a Gaussian times a smooth amplitude. The terms of
order $t^{-1/2}$ are odd polynomials in $y$; their integrals vanish.
Replacing integrals of a polynomial times a Gaussian by the
corresponding shifted lattice sums changes the normalized expression by $O(t^{-N})$ for every
fixed $N$, as follows by taking Fourier transforms of Gaussians.
The Taylor remainder of order four and the uniform Stirling remainder
are $O(t^{-1})$ after integration. The complement of this local region
inside an interior saddle neighborhood is exponentially small in a
positive power of $t$, and the rest has the uniform entropy gap already
proved. This proves \eqref{aux:uniformperiod}.

The same expansion applies coefficientwise to the Gamma ratio amplitude
in \eqref{aux:gammaratio}. After multiplication by
$q^Hr^\xi t^{2H+2\xi}$ it is smooth at the saddle, has leading value
$\Ga_X$, and has a uniform $O(t^{-1})$ asymptotic expansion there. The finite nilpotent expansion gives logarithmic factors of bounded
degree away from the saddle, which are absorbed by the entropy gap. Thus the normalized
positive sector is $\Ga_X+O_K(t^{-1})$.

Finally $\min_K(T-R)>0$. Divide \eqref{aux:weightednegative} by
\eqref{aux:uniformperiod}, decreasing this positive minimum if necessary,
to obtain \eqref{aux:uniformbound}. The assertions about $\mathcal W$ follow
from continuity and the degree-$1/2$ homogeneity of $T,R$. The strict
inequality at $(1,1)$ was proved in Theorem~\ref{aux:GammaSeed}.
\end{proof}

\begin{remark}
On the region $\mathcal W$, the absolute majorant proves the estimate. Outside it, the bound on terms of negative relative degree is
too coarse to compare exponents, so it gives no obstruction to Gamma I.
The next corollary establishes Gamma I beyond this region by a separate
propagation theorem.
\end{remark}

\begin{corollary}[Gamma I at positive parameters by propagation]\label{aux:GlobalCor}
The global Gamma I theorem of Hu--Ke--Li--Su \cite[Theorem 1.1]{Global}
implies that $X$ has the Gamma I property at every $q,r>0$, including
$r=16q/27$.
\end{corollary}
\begin{proof}
Use real logarithms to lift the positive quadrant to the small K\"ahler parameter space. Theorem \ref{spectral} places every such point in the region $(\SR)$
where the rightmost eigenvalue is simple. Any positive path from $(1,1)$ to a target point is contained in a connected complex neighborhood inside that open region. Theorem \ref{aux:GammaSeed} supplies the seed and the cited theorem propagates it along the path. No semisimplicity of the other eigenvalues is required. The coordinate axes themselves are not included.
\end{proof}

\begin{remark}[A direct proof at positive parameters]
The propagation argument uses \cite{Global}. The exact positive central
charge integral in Section~\ref{sec:markedjump} gives an independent
direct proof of Gamma I at positive parameters. These proofs establish
the limiting direction; the uniform error estimate for the normalized
$J$-function on $\mathcal W$ remains a separate quantitative result.
\end{remark}

\subsection{Helix monodromy and the fixed-phase basis change}\label{app:helix}
Use column coordinates in the exceptional $K$-basis $\mathcal F$ and put
$L=[\OO_X(H)]$, $U=[\OO_\pi(1)]$. To keep matrix notation distinct from
Chern characters, let $A$ denote the base Euler matrix and $P$ the matrix
of tensoring by $\OO_{Q^3}(1)$ in the base foundation:
\begin{equation}\label{APhelix}
 A=\begin{pmatrix}
 1&4&16&40\\0&1&5&14\\0&0&1&5\\0&0&0&1
 \end{pmatrix},\qquad
 P=\begin{pmatrix}
 -1&0&0&-4\\4&0&0&15\\0&1&0&-11\\0&0&1&5
 \end{pmatrix}.
\end{equation}
Then the Euler matrix is $G=\begin{psmallmatrix}A&B\\0&A\end{psmallmatrix}$,
where
\begin{equation}\label{BfromP}
 B=A(I+P)=\begin{pmatrix}
 16&20&56&120\\4&6&19&44\\0&1&6&19\\0&0&1&6
 \end{pmatrix}.
\end{equation}
The identity matrix of size four is implicit in the block formulas.

\begin{theorem}[The integral helix and monodromy at the regular end]\label{helixmatrix}
Tensoring by the two line bundles and by the anticanonical bundle gives
the matrices
\begin{align}
 T_H&=\begin{pmatrix}P&0\\0&P\end{pmatrix},&
 T_U&=\begin{pmatrix}0&-P\\I&I+P\end{pmatrix},\label{twistmatrices}\\
 \mathcal R=T_H^2T_U^2
 &=\begin{pmatrix}
 -P^3&-P^3(I+P)\\
 P^2(I+P)&P^2(I+P+P^2)
 \end{pmatrix}.
 \label{Rblocks}
\end{align}
They satisfy
\begin{equation}\label{Serreidentities}
 \mathcal R=G^{-\mathsf T}G,\qquad
 [\mathsf S_X]=\mathcal R^{-1}=G^{-1}G^{\mathsf T},\qquad
 \mathcal R^{\mathsf T}G\mathcal R=G.
\end{equation}
Moving the first object through the next seven by
\eqref{rightKmutation}, and repeating this operation eight times, induces
$\mathcal R$ on the entire foundation. Its rational Jordan form is
$J_5(1)\oplus J_3(1)$.
\end{theorem}
\begin{proof}
The spinor relation $[S(1)]=4[\OO]-[S]$, the two shifts
$\OO\mapsto\OO(1)\mapsto\OO(2)$, and the relation for $\OO(3)$
in \eqref{baseKrelation} give exactly the four columns of $P$.
The projective bundle relation $U^2-(1+L)U+L=0$ gives $T_U$, while
pullback gives $T_H$. Squaring and multiplying yields \eqref{Rblocks}.

Serre duality fixes the order of the transposes in
\eqref{Serreidentities}, independently of terminology for a Coxeter matrix.
If $K$ is the matrix of $\mathsf S_X$, then
$\chi(a,b)=\chi(b,Ka)$ implies
$G=K^{\mathsf T}G^{\mathsf T}$, hence $K=G^{-1}G^{\mathsf T}$.
Since $[4]$ acts trivially on $K_0$, the inverse of $K$ is tensoring by
$\omega_X^{-1}$, namely $\mathcal R$. The isometry follows either from
Serre duality or from tensoring both arguments by the same line bundle.
The helix identity \eqref{helixperiod} proves the full-turn statement.

Under the rational Chern character, $\mathcal R$ becomes multiplication
by $e^{c_1(X)}$. Multiplication by $c_1(X)$ has two Lefschetz strings,
of lengths five and three, since the even Betti numbers are
$1,2,2,2,1$ and $c_1(X)$ is ample. Exponentiation preserves the nilpotent
Jordan block lengths. Alternatively direct multiplication gives
\begin{equation}\label{Rranks}
 \rank(\mathcal R-I)^k=6,4,2,1,0\quad(k=1,2,3,4,5),
\end{equation}
which proves the asserted Jordan form.
\end{proof}

\begin{proposition}[Full turn of the Gamma framing]\label{Gammamonodromy}
Fix $q,r$ and a branch of the standard large-radius normalization
$\mathcal L$. In the Gamma-framed solution space, continuation around a
large circle with $\log z\mapsto\log z+2\pi i$ has matrix $\mathcal R$.
This describes monodromy at the regular end without assuming Gamma II.
\end{proposition}
\begin{proof}
Use the Gamma framing stated in Section~\ref{sec:Euler}:
\[
 \mathfrak s(E)=(2\pi)^{-2}\mathcal L(q,r,z)
 z^{-\mu}z^\rho\bigl(\widehat\Gamma_X\Ch(E)\bigr),\qquad
 \rho=c_1(X)\cup.
\]
The normalized matrix $\mathcal L$ is single-valued near the regular end.
Since every eigenvalue of $\mu$ is integral, $e^{-2\pi i\mu}=I$.
Continuation therefore multiplies the right-hand initial data by
$e^{2\pi i\rho}$. But
\[
 e^{2\pi i c_1(X)}\widehat\Gamma_X\Ch(E)
   =\widehat\Gamma_X\Ch(E\otimes\omega_X^{-1}).
\]
This proves the assertion by Theorem~\ref{helixmatrix}. A positive loop
in the local coordinate $1/z$ has the opposite orientation and gives
$\mathcal R^{-1}$. This orientation distinction is independent of the
choice of convention for the Stokes matrix \cite{Integral,GGI,CDGHelix}.
\end{proof}

\begin{remark}[Full-turn compatibility and Gamma II]
The identity $G^{-\mathsf T}G=\mathcal R$ expresses the compatibility
between the Euler pairing and the prescribed full-turn monodromy. If a
Stokes Gram matrix is compared in this same convention and marking,
the product of its opposite Stokes factors has this form. An inverse or
a conjugate appears if the orientation or basis is changed. These
full-turn identities alone do not identify $G$, the Gram matrix
of the unmutated natural foundation, with the Stokes matrix at a
prescribed phase. Theorem~\ref{mainUnit09} identifies the Gamma columns
analytically and gives $S_{\phi_0}=D_1^{\mathsf T}GD_1$ in the specified
mutated foundation. A monodromy conjugacy class or characteristic
polynomial contains insufficient information for this identification. Formula \eqref{Rblocks} gives the full
turn in the natural foundation.
\end{remark}

In the natural foundation $\mathcal F$, the basis change at the unit point
and chosen phase is
\begin{equation}\label{unitbasischange}
 D_1=\begin{pmatrix}
 4&0&0&1&0&0&0&0\\
 -11&1&1&-4&0&0&0&0\\
 5&0&0&0&0&-1&1&0\\
 -1&0&0&0&0&0&0&0\\
 0&-4&0&0&1&0&0&0\\
 0&11&0&0&0&-5&0&1\\
 0&-5&0&0&0&1&0&0\\
 0&1&0&0&0&0&0&0
 \end{pmatrix}.
\end{equation}
Indeed, the first column is
$[L^{-1}]=4[S_0]-11[\OO]+5[L]-[L^2]$; tensoring this identity by $U$
gives the second column through $[B_1]=[\OO]-[L^{-1}U]$.
The remaining columns are the classes in \eqref{objectclasses09}.
Thus $S_{\phi_0}=D_1^{\mathsf T}GD_1$, and the full-turn action in this
foundation is $D_1^{-1}\mathcal R D_1$, as in \eqref{unitdethelix09}.

\subsection{The integral lattice generated by line bundles}\label{app:integrallattice}
\begin{proposition}[The line lattice over odd quadrics]\label{familylattice}
Let $\nu=2m+1\geq3$, and let $\Lambda_{\mathrm{line},\nu}$ be the
subgroup of $K_0(X_\nu)$ generated by all line bundles. Then
\begin{equation}\label{familylatticeeq}
 K_0(X_\nu)/\Lambda_{\mathrm{line},\nu}
       \cong (\Z/2^{m+1}\Z)^2.
\end{equation}
The two generators are the classes of $S_{\nu,0}$ and $S_{\nu,1}$.
Tensoring by $L$ acts as $-I$ on this quotient, tensoring by $U$ exchanges
the generators, and the anticanonical full turn acts as the identity.
\end{proposition}
\begin{proof}
The full exceptional collection on the odd quadric
\cite{Kapranov,HK} gives the free integral basis
$(S_\nu,\OO,L,\ldots,L^{\nu-1})$ for $K_0(Q^\nu)$.
Put $d=2^{m+1}$. The spinor sequence gives
$(1+L)[S_\nu]=d[\OO]$. Injectivity of the rational Chern character and
$H^{\nu+1}=0$ imply $(L-1)^{\nu+1}=0$ in the integral $K$-group.
The monic integer polynomial
\[
 P_\nu(x)=\frac{(x-1)^{\nu+1}-2^{\nu+1}}{x+1}
\]
is defined because $\nu+1$ is even. Since $d^2=2^{\nu+1}$,
$(1+L)(P_\nu(L)+d[S_\nu])=0$. Multiplication by $1+L$ is injective:
after rational Chern character it is multiplication by the unit
$1+e^H$, and the $K$-group is torsion-free. Consequently
\[
 L^\nu=-d[S_\nu]+\sum_{j=0}^{\nu-1}c_jL^j,
 \qquad c_j\in\Z.
\]
All positive and negative powers of $L$ are generated by
$1,L,\ldots,L^\nu$, using $(L-1)^{\nu+1}=0$ and its unit constant
term. The line lattice on the base is therefore precisely
$\langle1,L,\ldots,L^{\nu-1},d[S_\nu]\rangle$.
Its quotient is cyclic of order $d$.

Orlov's formula splits $K_0(X_\nu)$ integrally as
$K_0(Q^\nu)\oplus U K_0(Q^\nu)$. The relation
$U^2-(1+L)U+L=0$, including $U^{-1}=(1+L-U)L^{-1}$, shows that
its line subgroup splits into the corresponding two base line
subgroups. This proves \eqref{familylatticeeq} and the two stated
generators. Modulo lines, $L[S_{\nu,0}]=-[S_{\nu,0}]$ and likewise
for $S_{\nu,1}$. The projective bundle relation gives
$U[S_{\nu,0}]=[S_{\nu,1}]$ and
$U[S_{\nu,1}]=[S_{\nu,0}]$. Finally
$\omega_{X_\nu}^{-1}=L^{\nu-1}U^2$, and $\nu-1$ is even.
\end{proof}

In particular, the two Gamma-marked sections of
Theorem~\ref{familytwocolumn} represent the two missing cyclic directions
of the quantum Gamma lattice generated by line bundles. This describes their arithmetic role without asserting that they belong
to a common asymptotic basis. The proposition follows integrally from
the known exceptional and spinor relations and makes no claim to a new
general structure theorem in $K$-theory.
For $\nu=3$ it specializes to the following fourfold computation.

\begin{proposition}[Index of the lattice generated by line bundles]\label{aux:lattice}
Let $L_X\subset K_0(X)$ be the subgroup generated by all line bundles on $X$. Then
\[
 K_0(X)/L_X\cong(\Z/4\Z)^2,\qquad [K_0(X):L_X]=16.
\]
Thus tensoring the structure sheaf by line bundles and applying integral
mutations among their classes cannot produce a full integral exceptional
basis.
\end{proposition}
\begin{proof}
On $Q^3$, the Chern character gives the integral relation
\[
 [\OO(3)]=-4[S]+15[\OO]-11[\OO(1)]+5[\OO(2)].
\]
The line classes $\OO,\OO(1),\OO(2),\OO(3)$ therefore have determinant of absolute value $4$ in the exceptional $K$-basis. They generate every $\OO(j)$, $j\in\Z$: writing $L=[\OO(1)]$, the relation $(L-1)^4=0$ implies the usual integral finite-difference recursion, including negative powers. This identity holds integrally because $K_0(Q^3)$ is torsion-free and its Chern character is rationally injective. Thus the sublattice generated by line bundles on $Q^3$ has index $4$.

On $X$, put $U=[\OO_\pi(1)]$. The integral projective bundle relation is
\[
 U^2-(1+L)U+L=0.
\]
Since $L$ is a unit, every positive or negative power of $U$ reduces to a linear combination of $1,U$ with coefficients in $\Z[L,L^{-1}]$. Moreover
\[
 K_0(X)=K_0(Q^3)\oplus U K_0(Q^3).
\]
All line bundles on $X$ are $L^jU^\ell$. Their span is exactly the direct sum of two copies of the base line lattice. The base quotient is cyclic of order four, generated by the spinor class, because the other three members of the exceptional basis are line classes and the displayed relation imposes exactly $4[S]=0$ in the quotient. Hence the total quotient is $(\Z/4\Z)^2$, with generators the classes of $\pi^*S$ and $\pi^*S\otimes\OO_\pi(1)$. Mutations act by integral linear combinations and do not escape the line subgroup.
\end{proof}

\begin{remark}[The integral subgroup and split-generation]
The proposition rules out a full exceptional collection of line bundles
and any construction of an integral basis that stays inside $L_X$.
It does not rule out generating a category by cones and arbitrary direct
summands, since taking a direct summand need not preserve a subgroup of
$K_0$. The index calculation therefore gives no stronger claim about
categorical generation.
\end{remark}

\paragraph{The action on the quotient.}
Let $L_X$ continue to denote the subgroup generated by all line bundles.
Write $\sigma=[\pi^*S]$, $\upsilon=[\pi^*S\otimes\OO_\pi(1)]$ in
$K_0(X)/L_X$.
\begin{proposition}[Full-turn action on the finite quotient]
\label{aux:helixquotient}
In the ordered quotient basis $(\sigma,\upsilon)$ of
$(\mathbb Z/4\mathbb Z)^2$, the actions are
\begin{equation}\label{aux:quotientactions}
 \overline T_H=-I_2,\qquad
 \overline T_U=\begin{pmatrix}0&1\\1&0\end{pmatrix},\qquad
 \overline{\mathcal R}=I_2.
\end{equation}
Consequently the anticanonical full turn preserves the proper line
sublattice of index 16 and acts trivially on its quotient. Compatibility
with the full turn therefore does not show that a proposed integral
basis generates the quotient.
\end{proposition}
\begin{proof}
The spinor relation gives $L\sigma=-\sigma$ modulo line classes,
and similarly $L\upsilon=-\upsilon$. Since $U\sigma=\upsilon$,
\[
 U\upsilon=U^2\sigma=(1+L)U\sigma-L\sigma=\sigma
 \quad\bmod L_X.
\]
Both actions square to the identity, proving the assertion for
$\mathcal R=L^2U^2$. In particular helix extension of line classes by
mutations and anticanonical twists cannot leave $L_X$.
\end{proof}

An integral Gamma II proof must account for this quotient. Constructing
two complex solutions at the added critical points does not establish
integral saturation. The quotient also contains arithmetic information
that the Jordan form in Theorem~\ref{helixmatrix} does not retain.

\subsection{Pairings and saturation}
The explicit full turn reduces the calculation of the cross block when
the integral basis and its marking are fixed.
\begin{lemma}[Recovery of the cross block in a fixed foundation]
\label{aux:crossrecovery}
Suppose a candidate Gram matrix for a bilinear pairing or Stokes data,
expressed in a specified ordered integral foundation with two blocks,
has the form
\[
 S_X=\begin{pmatrix}A&X_0\\0&A\end{pmatrix}.
\]
If its full-turn matrix in \emph{the same marked basis} is
$S_X^{-\mathsf T}S_X=\mathcal R$ from \textup{(\ref{Rblocks})}, then
\begin{equation}\label{aux:crossunique}
 X_0=A^{\mathsf T}\mathcal R_{12}=A(I+P)=B.
\end{equation}
\end{lemma}
\begin{proof}
The assumed equation is equivalent to $S_X^{\mathsf T}\mathcal R=S_X$.
Its upper-right block is $A^{\mathsf T}\mathcal R_{12}=X_0$.
Since $\mathcal R_{12}=-P^3(I+P)$ and $-A^{\mathsf T}P^3=A$,
this is \eqref{aux:crossunique}. The last identity follows from the base
Serre formula in dimension three, or from \textup{(\ref{APhelix})}.
\end{proof}

Recovering $X_0$ requires the full monodromy matrix in the chosen basis;
its eigenvalues or a conjugacy $M\sim\mathcal R$ are insufficient. One
must also identify the two diagonal blocks with $A$ for rapid-decay
cycles in the total space, which has not been done here. A marked
integral transform from the two known base foundations would reduce
the sixteen cross coefficients to a comparison of full turns.
Constructing that transform remains an analytic and integral problem.

\begin{lemma}[An integral saturation criterion]\label{aux:saturationtest}
Let $\Lambda$ be a free integral lattice of rank eight with an integral
bilinear form. If eight elements span a full-rank subgroup $\Lambda_0$
and their Gram matrix has determinant $1$ or $-1$, then
$\Lambda_0=\Lambda$.
\end{lemma}
\begin{proof}
Write the eight elements as columns of an integral matrix $D$ in a basis
of $\Lambda$. If $G_\Lambda$ is the ambient Gram matrix, then
$\det G_{\Lambda_0}=(\det D)^2\det G_\Lambda$. All factors are integers,
so $|\det D|=1$. This is the index of $\Lambda_0$ in $\Lambda$.
\end{proof}
To apply the criterion to rapid decay cycles, one must first establish
a free integral group of rank eight and an integral self-pairing,
including the transport that identifies opposite decay sectors. The
complex rank computation in Section~\ref{aux:sec:transform} does not
establish these integral properties. Once they are known, a verified
Gram matrix $G$ would imply saturation because $\det G=1$.

\section*{Acknowledgement.}
\addcontentsline{toc}{section}{Acknowledgement}
The author would like to thank Prof. Bohui Chen, An-Min Li, and Guosong Zhao for
their constant support. The author is particularly grateful to Andrea Brini,
Todor Milanov, Kenta Osuga, Arpan Saha, Peng Zhao, and Jingling Yang for
inspiring online seminar talks during the pandemic, and to Doan Nhat Minh and
Van Nguyen for enjoyable and fruitful collaborations. The author has also
benefited from conversations and discussions with Prof. Jianxun Hu, Huazhong Ke,
Bohan Fang, and Junxiao Wang. In particular, Jianxun Hu suggested initiating this investigation from Gamma II rather than Conjecture O during discussions after his seminar talk at Sichuan University; Huazhong Ke proposed adopting an alternative example at a Gromov-Witten conference in Shandong University. The paper expands on the author's oral
presentation at Symmetry 2025 (The 5th International Conference on Symmetry), held in Hangzhou. The author also thanks Sung-Soo Kim, Kimyeong Lee, Futoshi Yagi, 
Satoshi Nawata, and Rui-Dong Zhu for useful discussions and collaborations in mathematical
physics, as well as friends and colleagues for stimulating discussions at 
various conferences. Finally, the author gratefully acknowledges the
hospitality and support of the Mainz Institute for Theoretical Physics (MITP)
of the Cluster of Excellence PRISMA+ (Project ID 390831469).

\section*{Statements and Declarations}
\addcontentsline{toc}{section}{Statements and Declarations}
\paragraph{Funding.}
This work was supported by the National Natural Science Foundation of China
(NSFC) under Grant Nos. 11501470, 11426187, and 11791240561, and was partially
supported by NSFC Grant No. 11671328, the Chengdu Science and Technology
Program under Grant No. 2025-YF09-00007-SN, and the Fundamental Research
Funds for the Central Universities under Grant Nos. 2682021ZTPY043,
2682025ZTPY001, 2682025ZTPY057, and 2682025ZTO002.

\paragraph{Competing Interests.}
The author declares that there are no competing interests relevant to
the content of this article.

\paragraph{Data availability.}
This article is theoretical. The proofs and explicit algebraic calculations
are contained in the text; no empirical data are used.

\begingroup\small\raggedright
\phantomsection

\endgroup
\end{document}